\documentclass[leqno, 11pt]{amsart}
 \usepackage[%
   tmargin=1in,bmargin=1in,%
   lmargin=1in,rmargin=1in,%
   marginparsep=0.2in, marginparwidth=1.0in%
   ]{geometry}

\usepackage[T1]{fontenc}
\usepackage[utf8]{inputenc}

\usepackage[mathcal]{eucal}
\usepackage{amssymb}
\usepackage[usenames,dvipsnames]{xcolor} 
\usepackage[normalem]{ulem}
\usepackage{wasysym}
\usepackage{amsthm}
\usepackage{bbold}
\usepackage{comment}
\usepackage{enumitem}
\usepackage{amsmath}
\usepackage{tikz-cd}
\usepackage{stmaryrd} 
\usepackage{etoolbox} 
\usepackage{microtype}
\usepackage{adjustbox}
\usepackage{bbm}

\usepackage[backend=biber,style=alphabetic,]{biblatex}
\DeclareFieldFormat{postnote}{#1}

\usepackage[unicode]{hyperref} 

\definecolor{dark-red}{rgb}{0.5,0.15,0.15}
\definecolor{dark-blue}{rgb}{0.15,0.15,0.6}
\definecolor{dark-green}{rgb}{0.15,0.6,0.15}
\definecolor{navy}{HTML}{03005e}

\hypersetup{
  colorlinks,
  linkcolor=navy, 
  citecolor=navy, 
  urlcolor=navy, 
  pdftitle={Higher Zariski Geometry},
  pdfauthor={Aoki, Ko and Barthel, Tobias and Chedalavada, Anish and Schlank, Tomer and Stevenson, Greg},
  bookmarksdepth=2,
}

\usepackage[nameinlink,capitalise,noabbrev]{cleveref}

\numberwithin{equation}{section}

\usepackage[all]{xy}
\xyoption{line}
\usepackage{graphicx}
\usepackage{mathtools}
\usepackage{quiver}

\usepackage{caption}

\newtheorem{ThmAlpha}{Theorem}

\newtheorem{CorAlpha}[ThmAlpha]{Corollary}

\newtheorem{theorem}[equation]{Theorem}
\newtheorem*{theorem*}{Theorem}
\newtheorem*{MainThm*}{Main Theorem}
\newtheorem{proposition}[equation]{Proposition}
\newtheorem{lemma}[equation]{Lemma}
\newtheorem{corollary}[equation]{Corollary}
\newtheorem*{question*}{Question}

\theoremstyle{definition}
\newtheorem{definition}[equation]{Definition}
\newtheorem*{definition*}{Definition} 
\newtheorem{construction}[equation]{Construction}
\newtheorem{observation}[equation]{Observation}
\newtheorem{notation}[equation]{Notation}
\newtheorem{recollection}[equation]{Recollection}

\theoremstyle{remark}

\newtheorem{example}[equation]{Example}
\newtheorem*{example*}{Example}

\newtheorem{warning}[equation]{Warning}

\newtheorem{remark}[equation]{Remark}
\newtheorem*{remark*}{Remark}

\tikzset{
    labelrotatebelow/.style={anchor=north, rotate=90, inner sep=1.0mm}
}
\tikzset{
    labelrotateabove/.style={anchor=south, rotate=90, inner sep=1.0mm}
}
\SelectTips{cm}{10}

\renewcommand{\emptyset}{\varnothing}

\usepackage{todonotes}

\newcommand{\CAlg}{\operatorname{CAlg}}
\newcommand{\End}{\operatorname{End}}
\newcommand{\Ext}{\operatorname{Ext}}
\newcommand{\Fun}{\operatorname{Fun}}
\newcommand{\Hom}{\operatorname{Hom}}
\newcommand{\Ho}{\mathrm{ho}\hspace{1pt}}
\newcommand{\Idl}{\operatorname{Idl}}
\newcommand{\Ind}{\operatorname{Ind}}

\newcommand{\LTop}{\operatorname{LTop}} 
\newcommand{\RTop}{\operatorname{RTop}}

\newcommand{\Map}{\operatorname{Map}}
\newcommand{\Mod}{\operatorname{Mod}}

\newcommand{\Perf}{\operatorname{Perf}}

\newcommand{\Pro}{\operatorname{Pro}}

\newcommand{\QCoh}{\operatorname{QCoh}}

\newcommand{\Rad}{\operatorname{Rad}}

\newcommand{\Shv}{\operatorname{Shv}}
\newcommand{\Spc}{\operatorname{Spc}}
\newcommand{\Spec}{\operatorname{Spec}}
\newcommand{\Str}{\operatorname{Str}}

\newcommand{\Tot}{\operatorname{Tot}}

\newcommand{\ad}{\operatorname{ad}}

\newcommand{\cofib}{\operatorname{cofib}}

\newcommand{\colim}{\operatorname{colim}}

\newcommand{\disc}{\operatorname{disc}}

\newcommand{\ev}{\operatorname{ev}}
\newcommand{\fib}{\operatorname{fib}}

\newcommand{\id}{\operatorname{id}}

\newcommand{\lex}{\operatorname{lex}} 

\newcommand{\loc}{\operatorname{loc}}

\newcommand{\op}{\operatorname{op}}

\newcommand{\rig}{\operatorname{rig}}
\newcommand{\st}{\operatorname{st}}

\newcommand{\Ab}{\mathrm{Ab}}
\newcommand{\Bal}{\mathrm{Zar}}

\newcommand{\Catperf}{\Cat^{\textnormal{perf}}}
\newcommand{\Catex}{\Cat^{\textnormal{ex}}}
\newcommand{\Cat}{\mathrm{Cat}_{\infty}}
\newcommand{\Catbig}{\widehat{\mm{Cat}}_{\infty}}

\newcommand{\Fin}{\mathrm{Fin}}
\newcommand{\G}{\mathcal{G}}
\newcommand{\GBal}{\G_{\Bal}}
\newcommand{\GZar}{\G_{\mathrm{cZar}}}
\newcommand{\PrLst}{\mathrm{Pr^{L}_{st}}}
\newcommand{\PrLstomega}{\mathrm{Pr^{L}_{\omega, \st}}}
\newcommand{\PrL}{\mathrm{Pr^{L}}}
\newcommand{\PrLomega}{\mathrm{Pr^{L}_{\omega}}}
\newcommand{\Sp}{\mathrm{Sp}}
\newcommand{\Zar}{\mathrm{Zar}}
\newcommand{\Dir}{\mathrm{Dir}}
\newcommand{\GDir}{\G_{\Dir}}

\newcommand{\unit}{\mathbbm{1}}
\newcommand{\twoCAlg}{\mathrm{2CAlg}}
\newcommand{\twoCAlgrig}{\twoCAlg^{\mathrm{rig}}}
\newcommand{\twoModbig}{\widehat{\Mod}{}}
\newcommand{\twoModcg}{\twoModbig^{\mm{cg}}}
\newcommand{\Idem}{\mm{Idem}}

\newcommand{\recollement}[5]{%
  \xymatrix{{#1}
    \ar[r]|-{#2}
    & #3 \ar[r]|-{#4} \ar@<1ex>[l]^-{{#2}_!} \ar@<-1ex>[l]_-{{#2}^*}
    & #5, \ar@<1ex>[l]^-{{#4}!} \ar@<-1ex>[l]_-{{#4}^*}
  }
}
\newcommand{\noloc}{%
  \nobreak\mspace{6mu plus 1mu}{:}
  \nonscript\mkern-\thinmuskip
  \mathpunct{}\mspace{2mu}
}

\newcommand{\Yo}{\text{\usefont{U}{min}{m}{n}\symbol{'110}}}
\DeclareFontFamily{U}{min}{}
\DeclareFontShape{U}{min}{m}{n}{<-> dmjhira}{}
\DeclareUnicodeCharacter{05E6}{\tsadi}

\newcommand{\cat}[1]{\mathcal{#1}}

\newcommand{\bE}{\mathbb{E}}
\newcommand{\cG}{\mathcal{G}}
\newcommand{\cO}{\mathcal{O}}
\newcommand{\cS}{\mathcal{S}}
\newcommand{\cX}{\mathcal{X}}
\newcommand{\cY}{\mathcal{Y}}

\definecolor{DefColor}{rgb}{0.6,0.15,0.25}
\definecolor{DefColorA}{HTML}{156315}
\definecolor{DefColorB}{HTML}{8a0314}
\definecolor{DefColorC}{HTML}{016989}
\definecolor{DefColorD}{HTML}{3a0289}

\newcommand{\mdef}[1]{\textcolor{DefColorA}{#1}}
\newcommand{\tdef}[1]{\mdef{\emph{#1}}}

\newcommand{\bb}[1]{\mathbb{#1}}
\newcommand{\cc}[1]{\EuScript{#1}}
\newcommand{\mm}[1]{\mathrm{#1}}
\DeclareMathOperator{\PShv}{\cc{P}}

\DeclareMathOperator{\simarrow}{\stackrel{\sim\hspace{.2ex}}{\smash{\longrightarrow}\rule{0pt}{0.4ex}}}

\begin{document}

\title[Higher Zariski Geometry]{Higher Zariski Geometry}

\author[Aoki]{Ko Aoki}
\author[Barthel]{Tobias Barthel}
\author[Chedalavada]{Anish Chedalavada}
\author[Schlank]{Tomer Schlank}
\author[Stevenson]{Greg Stevenson}
\date{\today}

\makeatletter
\patchcmd{\@setaddresses}{\indent}{\noindent}{}{}
\patchcmd{\@setaddresses}{\indent}{\noindent}{}{}
\patchcmd{\@setaddresses}{\indent}{\noindent}{}{}
\patchcmd{\@setaddresses}{\indent}{\noindent}{}{}
\makeatother

\address{Ko Aoki, Max Planck Institute for Mathematics, Vivatsgasse 7, 53111 Bonn, Germany}
\email{aoki@mpim-bonn.mpg.de}
\urladdr{}

\address{Tobias Barthel, Max Planck Institute for Mathematics, Vivatsgasse 7, 53111 Bonn, Germany}
\email{tbarthel@mpim-bonn.mpg.de}
\urladdr{https://sites.google.com/view/tobiasbarthel/}

\address{Anish Chedalavada, Johns Hopkins University, 404 Krieger Hall, Baltimore, MD 21218, United States}
\email{achedal1@jh.edu}
\urladdr{https://aragogh.github.io}

\address{Tomer Schlank}
\email{}
\urladdr{}

\address{Greg Stevenson, Aarhus University, Department of Mathematics, Ny Munkegade 118, bldg.\ 1530, 8000 Aarhus C, Denmark}
\email{greg@math.au.dk}
\urladdr{}

\begin{abstract}
    We revisit the classical constructions of tensor-triangular geometry in the setting of stably symmetric monoidal idempotent-complete $\infty$-categories, henceforth referred to as 2-rings. In this setting, we produce a Zariski topology, a Zariski spectrum, a category of locally 2-ringed spaces (more generally $\infty$-topoi), and an affine spectrum-global sections adjunction, based on the framework of ``$\infty$-topoi with geometric structure'' as developed by Lurie in \cite{LurieDAG5}. Using work of Kock and Pitsch, we compute that the underlying space of the Zariski spectrum of a 2-ring recovers the Balmer spectrum of its homotopy category. These constructions mirror the analogous structures in the classical Zariski geometry of commutative rings (and commutative ring spectra), and we also demonstrate additional compatibility between classical Zariski and higher Zariski geometry. For rigid 2-rings, we show that the descent results of Balmer and Favi admit coherent enhancements. As a corollary, we obtain that the Zariski spectrum fully faithfully embeds rigid 2-rings into locally 2-ringed $\infty$-topoi. In an appendix, we prove a ``stalk-locality principle'' for the telescope conjecture in the rigid setting, extending earlier work of Hrbek.
\end{abstract}

\subjclass[2020]{14A20, 18F99, 18G80, 55P42, 55U35}
  \maketitle
  \thispagestyle{empty}

\vspace{-5pt}
 \begin{figure}[h]
   \includegraphics[scale=0.3]{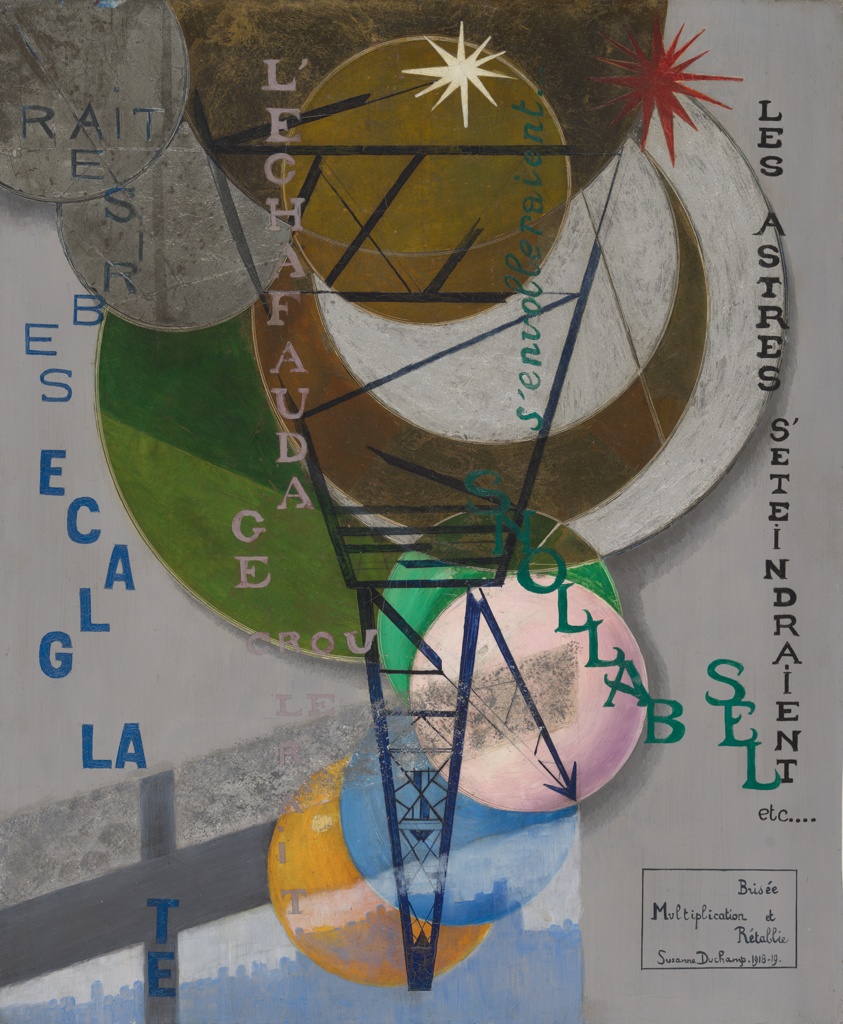}
   \caption[Cover Caption]{\emph{Broken and Restored Multiplication,} Suzanne Duchamp, 1918--19\footnotemark}
 \end{figure}
\footnotetext{{\url{https://www.artic.edu/artworks/134050/broken-and-restored-multiplication} (visited on 07/17/2025)}}

\newpage

\setcounter{tocdepth}{1}

 \section{Introduction}\label{sec:introduction}

\subsection{Overview} 

The Zariski--Grothendieck approach to algebraic geometry studies a given commutative ring $R$ by representing it as the global sections of a sheaf of rings on a topological space $\Spec R$. This topological space is known as the Zariski spectrum, and the natural sheaf it carries is referred to as its structure sheaf. It is first observed by Joyal in \cite{JoyalZariski} that the topological space $\Spec R$ and its structure sheaf are essentially determined by the poset of localizations $R \to R[S^{-1}]$ for multiplicatively closed subsets $S\subseteq R$. The Zariski spectrum of a commutative ring along with its sheaf of rings is an example of a \emph{locally ringed space}, i.e., a pair $(X,\cO_X)$ consisting of a topological space together with a sheaf $\cO_X$ of rings on $X$ whose stalks are local rings. The assignment of a commutative ring to this pair leads to the following basic structural result in algebraic geometry.

\begin{theorem*}[Fundamental Adjunction of Zariski Geometry]
  There is an adjunction of the form
  \begin{equation}\label{eq:zarfundadj}
        \Spec \colon \CAlg \rightleftarrows \{\mm{Locally\ Ringed\ Spaces}\}^{\op} \noloc \Gamma
      \end{equation}
where $\Spec$ assigns to a commutative ring its Zariski spectrum equipped with its structure sheaf, and $\Gamma$ sends a locally ringed space to its ring of global sections. Furthermore, the left adjoint $\Spec$ is fully faithful, and in particular the map from any commutative ring $R$ to the global sections of the structure sheaf on $\Spec R$ is an equivalence.
\end{theorem*}

The main goal of this paper is to develop a categorification of Zariski geometry called \emph{higher Zariski geometry} and to relate it to the subject of tensor-triangular geometry. To motivate our choice of setting, recall that Thomason \cite{Thomason1997} shows that any quasicompact quasiseparated scheme $X$ (e.g., an affine scheme) can be reconstructed from its perfect derived category $\Perf_{X}$. Here, it is crucial to consider $\Perf_{X}$ as a triangulated category together with its tensor structure: While a theorem of Bondal--Orlov \cite{BondalOrlov2001} allows the reconstruction from the triangulated structure alone under restrictive hypotheses on $X$, this is in general impossible. Furthermore, as we will explain in more detail below, a satisfying theory of descent requires working in a homotopically enriched context. In fact, viewing $\Perf_{X}$ as a \emph{stably symmetric monoidal} $\infty$-category, i.e., a stable $\infty$-category equipped with a biexact symmetric monoidal structure $\otimes$, the subject of generalized Tannakian-style reconstruction theorems of suitable schemes or geometric stacks $X$ from $\Perf_{X}$ has a rich history, as for example in the works of Lurie \cite{lurieSpectralAlgebraicGeometry}, To\"en \cite{Toen2007}, and Ben-Zvi--Francis--Nadler \cite{ben-zviIntegralTransformsDrinfeld2010a}. 

With that motivation in mind, we refer to an idempotent-complete stably symmetric monoidal $\infty$-category as a \emph{(commutative) 2-ring} and we write $\twoCAlg$ to denote the (large) $\infty$-category of (small) 2-rings. $\twoCAlg$ is a rich theatre of study, extending the scope of algebraic geometry significantly. As mentioned above, any quasicompact quasicompact scheme $X$ may be viewed as a commutative 2-ring through the passage to its derived category of perfect complexes $\Perf_{X}$, but there are many more examples: the category of (equivariant or motivic) spectra in homotopy theory, derived or stable categories in representation theory, the Kasparov category of separable $C^{\ast}$-algebras in non-commutative geometry, or Fukaya categories in symplectic geometry via Kontsevich's homological mirror symmetry conjecture \cite{KontsevichICM1994}. The next two theorems constitute the main results of this paper; for more precise formulations, we refer to \cref{thm:2zariski_geometry}, \cref{cor:rigidspec_fullyfaithful}, and \cref{thm:2zariski_balmerspectrum}.

\begin{theorem*}[Fundamental Adjunction of Higher Zariski Geometry]
    There is an $\infty$-category of locally $2$-ringed spaces consisting of pairs $(\cX, \cO)$ where $\cX$ is a space and $\cO$ is a sheaf of $2$-rings on $\cX$ satisfying a locality condition. This comes equipped with an adjunction
        \begin{equation}\label{eq:hzarfundadj}
            \Spec \colon \twoCAlg \rightleftarrows \{\mm{Locally\ 2\text{-}Ringed\ Spaces}\}^{\op}  \noloc  \Gamma,
          \end{equation}
    where $\Gamma$ takes a locally $2$-ringed space to its $2$-ring of global sections. Furthermore, the left adjoint is fully faithful on the full subcategory of $\twoCAlg$ consisting of rigid $2$-rings\footnote{that is, those $\cc{K} \in \twoCAlg$ such that every object of $\cc{K}$ is dualizable, see \cref{def:rigid}}. In particular, any rigid $2$-ring $\cc{K}$ is canonically equivalent to the global sections of the structure sheaf on $\Spec \cc{K}$.
\end{theorem*}

The homotopy category of a 2-ring naturally admits the structure of a \emph{tensor-triangulated category}. The idea of studying tensor-triangulated categories by assigning to them a notion of Zariski geometry is first proposed by Balmer in \cite{balmerSpectrumPrimeIdeals2005}, as a mechanism to unify classification results and other mutually transferrable techniques employed in the 
examples above. In this field, known as \emph{tensor-triangular} or \emph{tt-geometry}, one associates to a small tensor-triangulated category $\cc{K}$ a topological space $\Spc \cc{K}$ called the \emph{Balmer spectrum}. Many structural features of $\cc{K}$ can be encoded in terms of the topology of $\Spc \cc{K}$, and its study and computation are of principal importance in the theory.

In the vein of Joyal's characterization of the Zariski spectrum discussed above, the work of \cite{KockPitsch17} constructs the Balmer spectrum of a tensor-triangulated category $\cc{K}$ entirely out of the poset of its categorical localizations $\cc{K} \to \cc{K}[W^{-1}]$ which respect the tensor-triangulated structure. This perspective underlies our approach to the Zariski geometry of 2-rings, and will also facilitate the following key computation.

        \begin{theorem*}
          For any $\cc{K} \in \twoCAlg$, there is a natural homeomorphism between the space underlying $\Spec \cc{K}$ and $\Spc \Ho \cc{K}$, the Balmer spectrum of the homotopy category $\Ho \cc{K}$ of $\cc{K}$. 
        \end{theorem*}

The relevance of the study of 2-rings and their Zariski geometry is thus retroactively justified by the computation above: The tensor-triangulated structure alone is not enough to produce a structure sheaf on $\Spc \cc{K}$ giving rise to the fundamental adjunction. We view the fundamental adjunction as an essential desideratum of any algebro-geometric theory, and as such the goal of this paper is to give a higher-algebraic reincarnation of the conceptual framework of tt-geometry which allows one to state and prove it. Accordingly, we hope that this work will serve as an entry point for those versed in homotopical algebra to immerse themselves in the rich field of tt-geometry, as well as an opportunity for technicians of the field to add a new set of methods to their toolkit.

\subsection{Geometries}\label{ssec:1B} In \cite{LurieDAG5}, Lurie introduces a framework to articulate the basic structures of ``derived geometry'' in various contexts. This works by assigning to a small Grothendieck site $\cc{G}$ an $\infty$-category of ``$\infty$-topoi with local $\cc{G}$-structure'' along with a structure theory of schemes, affine schemes, and functors of points analogous to the usual constructions in algebraic geometry. We provide a detailed account of the theory in \cref{sec:geometries}. Here we content ourselves with enough of an account to state our main results.

\begin{remark*}
An $\infty$-topos is a left-exact localization of an $\infty$-category of presheaves on a small $\infty$-category.  For the reader unfamiliar with the language of topoi and $\infty$-topoi, we review standard notation in \cref{ssec:conventions}. We recommend \cite{rezkLecturesTopos} as an overview of the theory. The main examples in this paper appear as the $\infty$-categories of sheaves of spaces on a topological space or small Grothendieck site.    
\end{remark*}

  Following Lurie, a \emph{geometry} is a triple $(\cG,\cG^{\ad},\tau)$ consisting of:
    \begin{enumerate}
        \item An idempotent-complete small $\infty$-category $\cG$ which admits finite limits.
        \item A class of morphisms $\cG^{\ad}\subseteq \cG$ closed under left cancellation, retracts, and base change in $\cG$.
        \item A Grothendieck topology $\tau$ on $\cG$ which is generated by morphisms in $\cG^{\ad}$.
    \end{enumerate}
    We refer to morphisms in~$\cat{G}^{\ad}$ and covers in~$\tau$ as \emph{admissible morphisms} and \emph{admissible covers}, respectively. We are interested in studying categories of the form $\Pro(\cG)$ where $\cG$ is equipped with the structure of a geometry. Admissible morphisms are meant to capture basic open inclusions, and the Grothendieck topology abstracts covers of affine spectra by basic opens. We will often return to the following example to orient the reader's intuition.

\begin{example*}
    The \emph{classical Zariski geometry} consists of the following data:
        \begin{enumerate}
            \item $\GZar=(\CAlg^{\omega})^{\op}$, the opposite of the category of finitely presented commutative rings.
            \item Admissible morphisms $\GZar^{\ad}$ correspond to localization maps $R \to R[x^{-1}]$ for $x \in R$, where $R$ is a finitely presented commutative ring.
            \item  A finite collection $\{R\to R[x_i^{-1}]\}_{i \in I}$ is declared to generate an admissible covering sieve if the set $\{x_i\}_{i \in I} \subset R$ generates the unit ideal. 
            \end{enumerate}
\end{example*}

Recall that $\twoCAlg$ denotes the $\infty$-category of 2-rings, i.e., that of small idempotent-complete stably symmetric monoidal $\infty$-categories. Our first result generalizes the classical Zariski geometry on commutative rings to the setting of 2-rings. 

\begin{ThmAlpha}\label{thmalph:zariski2rings}
    The following triple $(\GBal, \cG_{\Bal}^{\ad}, \tau)$ defines a geometry on the $\infty$-category of 2-rings: 
        \begin{enumerate}
            \item $\GBal=(\twoCAlg^{\omega})^{\op}$, the opposite of the category of compact 2-rings.
             \item The class of admissible morphisms $\cG_{\Bal}^{\ad}$ corresponds to the principal Verdier localizations 
                \[
                    \cc{K}\to \cc{K}/\langle a \rangle \in \twoCAlg^{\omega,[1]}
                \] 
            (rather, their idempotent completions) for objects $a \in \cc{K}$. Here, $\twoCAlg^{\omega,[1]}$ denotes the arrow category on $\twoCAlg^{\omega}$.
            \item A finite collection of admissible morphisms $\{f_i\colon\cc{K}\to\cc{K}_{i}\}_{i \in I}$ is declared to generate a covering sieve in $\tau$ if the kernel of $\prod_{I} f_{i}\colon\cc{K} \to \prod_{I}\cc{K}_{i}$ consists of $\otimes$-nilpotent objects.
        \end{enumerate}
    We will refer to the above data as the \emph{Zariski geometry} on the opposite of the category of 2-rings. 
\end{ThmAlpha}      

Given the data of a geometry $\cG$, the results of \cite{LurieDAG5} construct a theory of locally $\cG$-ringed spaces together with an affinization functor. There is an $\infty$-category $\LTop(\cG)$ of \emph{$\infty$-topoi with local $\cc{G}$-structure}, informally presented as:
        \begin{enumerate}
            \item Objects are given by pairs $(\cX, \cO)$ where $\cX$ is an $\infty$-topos and $\cc{O}$ is an $\Ind(\cG^{\op})$-valued sheaf on $\cX$ satisfying a locality condition.
            \item Morphisms between objects $f\colon (\cc{X}, \cc{O_{X}}) \to (\cc{Y}, \cc{O_{Y}})$ are given by pairs 
            \[f^{\ast}\colon \cc{X} \to \cc{Y}\in \LTop^{[1]}\] along with a morphism of associated sheaves \[\gamma\colon f^{\ast}\cc{O_{X}} \to \cc{O_{Y}} \in \Str^{\loc}_{\cG}(\cY)^{[1]}\] satisfying a locality condition intrinsic to $\cG$.  
    \end{enumerate} 
The main abstract input to higher Zariski geometry is then the following result:

\begin{theorem*}[Lurie]\label{thm:omnibusspectra}
    For any geometry $\cG$, there exists an adjunction of the following form
        \[
            \Spec^{\cG}\colon \Ind(\cG^{\op}) \rightleftarrows \LTop(\cc{G}) \noloc\Gamma_{\cc{G}},
        \]
    where $\Gamma_{\cG}$ is informally given by sending $(\cX, \cO) \in \LTop(\cG)$ to the global sections $\cO(\cX)$. In particular, for all $(\cX, \cO_{\cX}) \in \LTop(\cG)$ one has a natural equivalence
  \begin{equation*}
        \Map_{\LTop(\cat{G})}(\Spec^{\cat{G}}(-),(\cat{X},\cO_{\cat{Y}})) \simeq \Map_{\Ind(\cG^{\op})}(-,\Gamma_{\cat{G}}(\cat{Y},\cO_{\cat{Y}})).
      \end{equation*}
      of functors $\Pro(\cG) \coloneq\Ind(\cG^{\op})^{\op} \to \cS$.
    The left adjoint $\Spec^{\cG}$ in the above theorem is referred to as the \emph{absolute spectrum} functor.
\end{theorem*}

Specialized to the classical Zariski geometry $(\GZar, \cG_{\Zar}^{\ad}, \tau)$, this recovers the fundamental adjunction of classical Zariski geometry \eqref{eq:zarfundadj}. Turning attention to the Zariski geometry $(\GBal, \cG_{\Bal}^{\st, \ad}, \tau)$ on 2-rings from \cref{thmalph:zariski2rings}, we obtain the fundamental adjunction of higher Zariski geometry \eqref{eq:hzarfundadj}:

\begin{CorAlpha}\label{cor:absoluteGBalspectrum}
    Let $(\cX, \cO) \in \LTop(\GBal)$. There is a natural equivalence
        \[
            \Map_{\twoCAlg}(\cc{K},\Gamma_{\GBal}(\cX, \cO)) \simeq \Map_{\LTop(\GBal)}(\Spec^{\GBal} \cc{K}, (\cX, \cO)).
        \]
\end{CorAlpha}

Henceforth we will use the notation $\Spec \coloneq \Spec^{\GBal}$ and $\Gamma \coloneq \Gamma_{\GBal}$ for the adjoint pair of functors supplied by this result. The absolute spectrum $\Spec^{\cG}$ in \cref{thm:omnibusspectra} is given by an explicit construction; our next result will identify $\Spec$ and its accompanying sheaf of 2-rings in terms of the classical constructions of tensor-triangular geometry. First, we show that the underlying $\infty$-topos of the Zariski spectrum can be fully expressed in terms of the Balmer spectrum.
    
\begin{ThmAlpha}\label{thmalph:comparison}
    Let $\cc{K} \in \twoCAlg$. The underlying $\infty$-topos of the absolute spectrum $\Spec\cc{K}$ may be identified with the $\infty$-topos $\Shv(\Spc \cc{K})$ of sheaves on the Balmer spectrum of the underlying tt-category of $\cc{K}$. Moreover, this identification is natural in $\cc{K}$.
\end{ThmAlpha}

\cref{cor:absoluteGBalspectrum} and \cref{thmalph:comparison} jointly provide a method to produce maps from the Balmer spectrum of a 2-ring into arbitrary $\infty$-topoi with $\GBal$-structure, enhancing a prior universal property of the Balmer spectrum expressed in the language of \emph{support data} as in \cite{balmerSpectrumPrimeIdeals2005}. In \cref{sec:supp}, we explain how an $\infty$-topos with $\GBal$-structure along with a map from a 2-ring $\cc{K}$ into its global sections naturally admits a support datum for $\cc{K}$. We then identify the resultant maps from $\Spec \cc{K}$ supplied by this support datum and \cref{cor:absoluteGBalspectrum}.

Part of the motivation for working within the abstract framework of geometries is its flexibility: the theory allows us to accommodate comparisons to other geometries, such as expressed through Thomason's theorem (see \cref{thm:affineness} taken from forthcoming work of the third author \cite{ChedRecon}), and paths the way towards an \'etale geometry of 2-rings which we hope to address in future work.

\subsection{Zariski descent} The absolute spectrum of a 2-ring $\cc{K}$ comes equipped with a natural sheaf of 2-rings, corresponding to its local $\GBal$-structure. Under the assumption that $\cc{K}$ is rigid, we will show that this sheaf recovers the ``structure presheaf'' of \cite{Balmer02} upon evaluation against quasicompact open subsets. 

Let us first justify our passage to the $\infty$-categorical setting by reviewing the shortcomings of descent in the framework of classical tt-geometry. We remark that partial results addressing these problems have been previously obtained in \cite{BalmerFaviGluing}, \cite{balmerGluingRepresentationsIdempotent2009}. 

\subsubsection{Motivation: descent in tensor-triangular geometry}

Let $\cc{K}$ be an idempotent-complete tt-category, and $\cc{I} \subseteq \cc{K}$ be a thick tensor ideal. The \emph{Karoubi quotient} of $\cc{K}$ by the ideal $\cc{I}$ is defined to be the idempotent completion of the localization $\cc{K}[W^{-1}]$ where $W$ refers to the class of morphisms with cofiber contained in $\cc{I}$. We write $\cc{K}/\cc{I}$ to denote this construction\footnote{Due to the implicit idempotent completion, this notation is nonstandard; see also \cref{warn:nonstandard}.}.

The ``structure presheaf'' on the spectrum of a tt-category, following \cite[\S 5]{Balmer02}, is defined as follows: To any open set $U \subseteq \Spc \cc{K}$, consider the thick tensor ideal $\cc{K}_{U^{c}} \coloneq \{a \in \cc{K} \mid \forall \cc{P} \in U, \ a\in \cc{P}\}$ of \emph{objects supported away from $U$}. One defines a presheaf of tt-categories on $\Spc \cc{K}$ via the assignment 
    \begin{equation}\label{eq:presheafofttcats}
        \widetilde{\cc{O}}_{\cc{K}}\colon \{U \subseteq \Spc \cc{K}\} \mapsto \cc{K}/\cc{K}_{U^{c}}, 
    \end{equation}
where an inclusion of open sets $U \subseteq V$ is sent to the induced map $\cc{K}/\cc{K}_{V^{c}} \mapsto \cc{K}/\cc{K}_{U^{c}}$, i.e., the Karoubi quotient of the source by the smallest thick tensor ideal containing the image of $\cc{K}_{U^{c}}$. 

Although this construction is natural from the perspective of tt-geometry, it is rarely the case that it extends to a sheaf (or more properly a stack) of categories on the Balmer spectrum. Let us study this failure---and how to rectify it via enhancement---using a classical example regarding ``nonuniqueness of gluing for morphisms'' along open subsets of the Balmer spectrum. 
  
\begin{example*}
Let $X = \bb{A}^{2}_{\bb{Z}} \setminus 0$ and $\cc{K} \coloneq \Ho \Perf_{X}$ be the tt-category of perfect complexes on $X$. A standard \v{C}ech complex computation with the open subsets $\{t_{1} \neq 0\},\ \{t_{2} \neq 0\} \subseteq \Spec \bb{Z}[t_{1},t_{2}]$ yields the following identifications  
    \[
        \hom_{\cc{K}}(\cc{O}_{X}, \Sigma \cc{O}_{X})\cong \Ext_{\Shv(X;\mm{Ab})}^{1}(\cc{O}_{X}, \cc{O}_{X}) \eqcolon \mm{R}^{1}\Gamma(X, \cc{O}_{X}) \cong \frac{1}{t_{1}t_{2}}\bb{Z}\left[\frac{1}{t_{1}},\frac{1}{t_{2}}\right].
    \]
Based on Thomason's theorem \cite{Thomason1997}, the main result of \cite{Balmer02} provides an equivalence $\Spc \cc{K} \simeq X$ which moreover identifies the assignment $\widetilde{\cO}_{\cc{K}}$ on $\Spc \cc{K}$ with the assignment $U \mapsto \Perf_{U}$ on quasicompact open subsets of $X$. Let $\{U_{i}\}_{i\in I}$ be a collection of affine open subschemes covering $X$. The vanishing of quasicoherent cohomology on affine schemes implies that $\prod_{i \in I}\hom_{\Perf_{U_{i}}}(\cc{O}_{X}|_{U_{i}}, \cc{O}_{X}|_{U_{i}}) = 0$ and hence the morphism presheaf      
    \[
        \hom_{\widetilde{\cc{O}}_{\cc{K}}}(\unit, \Sigma \unit)\colon U \subseteq X \mapsto \hom_{\cc{K}/\cc{K}_{U^{c}}}(\cc{O}_{X}, \Sigma \cc{O}_{X})
    \]
fails to be separated; there are maps which are locally zero which are not globally so.
\end{example*}

In the example above, we encounter a problem that already occurs upstream of tt-geometry. The nontriviality of the connecting homomorphism $\delta$ in the Mayer--Vietoris exact sequence below 
    \[
        \cdots \rightarrow \Gamma(\{t_{1}t_{2} \neq 0\}, \cc{O}) \xrightarrow{\delta} \mm{R}^{1}\Gamma(\bb{A}^{2}_{\bb{Z}}\setminus 0, \cc{O}) \rightarrow \mm{R}^{1}\Gamma(\{t_{1} \neq 0\}, \cc{O}) \oplus \mm{R}^{1}\Gamma(\{t_{2} \neq 0\}, \cc{O}) \rightarrow \cdots 
    \] 
implies that the right hand restriction map is not injective, and the assignment $U \mapsto \mm{R}^{i}\Gamma(X, \cc{O}_{X})$ fails to be a separated presheaf. Instead, a scheme $X$ and a choice of open cover $\coprod_{I}U_{i} \twoheadrightarrow X$ give rise to the following quasi-isomorphism of chain complexes in $\Shv(X; \Ab)$
    \begin{equation}\label{eq:totalization}
        \bb{I}^{\bullet}\cc{O}_{X} \simarrow \Tot\left(\prod_{I}\bb{I}^{\bullet}{\cc{O}_{U_{i}}} \xrightarrow{d_{0}-d_{1}} \prod_{I^{\times 2}}\bb{I}^{\bullet}\cc{O}_{U_{i}\cap U_{j}} \xrightarrow{d_{0} -d_{1} + d_{2}} \cdots\right),
    \end{equation}
where we write the prefix $\bb{I}^{\bullet}$ to indicate a choice of injective resolution in $\Shv(X; \Ab)$.

Let $\cc{U}^{\bullet}$ denote the \v{C}ech nerve of the open cover $\coprod_{I}U_{i} \twoheadrightarrow X$. Under the Dold--Kan correspondence between cochain complexes and simplicial objects in $\Shv(X; \Ab)$, the right hand side of \eqref{eq:totalization} is quasi-isomorphic to the diagonal of the bicosimplicial object associated to $\bb{I}\cc{O}_{\cc{U}^{\bullet}}$, see \cite[\S 4.2.2]{goerssSimplicialHomotopyTheory2009}. This latter object is a model for the homotopy limit in the simplicial model category of chain complexes in $\Shv(X; \Ab)$, giving the \emph{homotopy coherent} sheaf-theoretic statement 
    \[
        \cc{O}_{X} \simeq \mm{holim}_{\Delta}\cc{O}_{\cc{U}^{\bullet}} 
    \] 
prior to having passed to cohomology. In particular, this rectifies the failure of descent observed above, and begets a local-to-global spectral sequence with the following signature:
    \begin{equation}\label{rem:localtoglobalspectralsequence}
        \mm{\check{H}}_{\cc{U}}^{p}(X, \mm{R}^{q}\Gamma(\cc{U}^{\bullet}, \cc{O}_{\cc{U}^{\bullet}})) \Rightarrow \mm{R}^{p+q}\Gamma(X, \cc{O}_{X})  
    \end{equation}
via the identification in \eqref{eq:totalization}. Using higher Zariski descent, we will observe an analogue in the general tt-geometric setting below.

\subsubsection{Zariski descent in \texorpdfstring{$\twoCAlg$}{2CAlg}}

The discussion above suggests that one should expect an analogous homotopy coherent statement on mapping objects, which will necessitate working with an enhancement of the triangulated setting. Opting to work now in $\twoCAlg$, we obtain a descent result which aligns with this intuition. Recall that a 2-ring is \emph{rigid} if every object is dualizable. 

\begin{ThmAlpha}\label{thmalph:sheaf}
      Let $\cc{K}$ be a rigid 2-ring. The equivalence of \cref{thmalph:comparison} identifies the natural $\GBal$-structure on $\Spec \cc{K}$ with a $\twoCAlg$-valued sheaf on $\Spc \cc{K}$, denoted $\cc{O}_{\cc{K}}$, which upon passage to homotopy categories agrees with the presheaf of \eqref{eq:presheafofttcats} on quasicompact open subsets.
\end{ThmAlpha}
    
Let us indicate some sample consequences of this result, and explain the connection to prior gluing results in tt-geometry. These statements may fail outside the rigid setting; see \cref{warning:nonrigid}.
  
\begin{enumerate}
    \item[(1)]   A direct consequence of the above is the functor-of-points embedding for rigid 2-rings: The absolute spectrum functor $\Spec\colon \twoCAlg \to \LTop(\GBal)$ is fully faithful on the full subcategory of rigid 2-rings; in particular, any rigid $2$-ring $\cc{K}$ can be reconstructed from its absolute spectrum $(\Spc \cc{K},\cc{O}_{\cc{K}})$ by passing to global sections. This statement is in fact equivalent to the statement of \cref{thmalph:sheaf}, although we will not pursue the reverse direction here. 
    \item[(2)] Given a rigid 2-ring $\cc{K}$, objects $x, y \in \cc{K}$, and an open cover $\coprod_{I} U_{i} \twoheadrightarrow \Spc \cc{K}$ with \v{C}ech nerve $\cc{U}^{\bullet}$, there is a local-to-global spectral sequence with signature
    \[
        \mm{\check{H}}_{\cc{U}}^{p}(\Spc \cc{K}, \pi_{q}\Map_{\cO_{\cc{K}(U)}}(x,y)) \Rightarrow \pi_{q-p}\Map_{\cc{K}}(x,y)
    \]
    which is identified with the Bousfield--Kan spectral sequence for the cosimplicial object $\Map_{\cO_{\cc{K}}(\cc{U}^{\bullet})}(x,y)$; here, we abuse notation by continuing to write $x,y$ for the images of the respective objects under the restriction maps for $\cO_{\cc{K}}$. In the same way that the local-to-global spectral sequence of \eqref{rem:localtoglobalspectralsequence} generalizes Mayer--Vietoris sequences in sheaf cohomology, the local-to-global spectral sequence for a 2-ring is a generalization of the Mayer--Vietoris sequence for mapping objects constructed in \cite[\S 2]{BalmerFaviGluing}.
    \item[(3)]     Consider the functor $\mathfrak{pic} \colon \twoCAlg \to \Sp_{\geq 0}$ which sends a 2-ring to its space of invertible objects, considered as an $\bb{E}_{\infty}$-object (and hence a connective spectrum) via the tensor product operation. This functor commutes with all limits by \cite[Proposition 2.2.3]{mathewPicardGroupTopological2016}. Keeping the above notation, one obtains a similar local-to-global spectral sequence with signature
    \[
        \mm{\check{H}}_{\cc{U}}^{p}(\Spc \cc{K}, \pi_{q}\mathfrak{pic}(\cO_{\cc{K}}(U))) \Rightarrow \pi_{q-p}\mathfrak{pic}(\cc{K})
    \]
      which generalizes the Mayer--Vietoris sequence for Picard elements of \cite[\S 4]{BalmerFaviGluing}. 
\end{enumerate}

In \cref{sec:slp} we will use a strong form of \cref{thmalph:sheaf} (see \cref{thm:indsheaf}) to establish a locality criterion for the telescope conjecture in rigid 2-rings.

\begin{ThmAlpha}\label{thmalph:slp}
    Let $\cc{K}$ be a rigid 2-ring. Then $\Ind(\cc{K})$ satisfies the telescope conjecture if and only if $\Ind(\cc{K}/\cc{P})$ satisfies the telescope conjecture for all primes $\cc{P} \in \Spc \cc{K}$.
\end{ThmAlpha} 

This proves the \emph{stalk-locality principle} for the telescope conjecture outlined by Hrbek in \cite[\S 3]{hrbekTelescopeConjectureHomological2025} in the generality of arbitrary rigid 2-rings, extending work of \cite{Hrbek2024}.

\subsection{Zariski descent for modules and the local-to-global principle}\label{ssec:introZariskiLGP}

For $X$ a classical quasicompact quasiseparated scheme and $Z \subseteq X$ a closed subset with quasicompact open complement, the statement that the full $\QCoh_{X}$-submodule $\QCoh_{X, Z} \subseteq \QCoh_{X}$, i.e., the subcategory of complexes supported on $Z$, is compactly generated is first demonstrated in \cite{bondalGeneratorsRepresentabilityFunctors2003}. This is done via a topological ``reduction principle'' that inductively reduces the statement to a local property, which is then checked on affine schemes. Generalizing this, \cite{antieauBrauerGroupsEtale2014} supplies a Zariski local-to-global principle for the compact generation of objects of $\Mod_{\QCoh_{X}}(\PrL)$ (now where $X$ is allowed to be a quasicompact quasiseparated spectral scheme). However, in both cases, the basic arguments relied on facts about the Zariski geometry of the 2-ring $\Perf_X$. 

We will first show that the assignment of a 2-ring $\cc{K}$ to the category of modules over $\Ind(\cc{K})$ satisfies Zariski descent. In accordance with the observation above, we will also record an analog of the Zariski local-to-global principle for the Zariski spectrum of a 2-ring. These statements are succinctly stated below.

\begin{ThmAlpha}\label{thmalph:mod}
    Let $\cc{K}$ be a rigid 2-ring. The following assignments extend to $\CAlg(\Catbig)$-sheaves on $\Spec \cc{K}$. 
    \begin{enumerate}
    \item\label{thmalph:moda} The assignment $\twoModbig$ which sends a quasicompact open subset $U$ to $\Mod_{\Ind(\cO_{\cc{K}}(U))}(\PrL)$.
    \item\label{thmalph:modb} The assignment $\twoModcg$ which sends a quasicompact open subset $U$ to the full subcategory of $\twoModbig(U)$ consisting of those objects whose underlying presentable $\infty$-categories are compactly generated.
    \item\label{thmalph:modc} The assignment $\Mod$ which sends a quasicompact open subset $U$ to $\Mod_{\cO_{\cc{K}}(U)}(\Catperf)$.
    \end{enumerate} 
\end{ThmAlpha}

A consequence of the main result of \cite{ChedRecon} is that the Zariski local-to-global statement of \cite{antieauBrauerGroupsEtale2014} for spectral schemes (and more generally Dirac spectral schemes, suitably defined) follows entirely from the statement of \cref{thmalph:mod}. This uses an enhanced version of the reconstruction result of \cite{Balmer02} which is logically independent from the results above, see \cref{ssec:tt_thomason}.

\subsection{Outline of present work}

In \cref{sec:2rings}, we will study the $\infty$-category of 2-rings. The primary purpose of this section is to revisit localization theory in the rigid and nonrigid settings.
\Cref{sec:geometries} is a detailed recollection of the material on geometries from \cite{LurieDAG5}, to the extent necessary for our applications. In \cref{sec:zariski}, we introduce the Zariski geometry on the $\infty$-category of 2-rings and relate it to tensor-triangular geometry, proving \cref{thmalph:zariski2rings} and \cref{thmalph:comparison}.
\Cref{sec:structuresheaf} is dedicated to the proofs of our results on Zariski descent, namely \cref{thmalph:sheaf}, its consequences, and \cref{thmalph:mod}. \cref{sec:supp} contains a discussion of support data from the perspective of higher Zariski geometry. \cref{sec:slp} is dedicated to the proof of \cref{thmalph:slp}.  

\subsection{Conventions and Notation}\label{ssec:conventions}

We work with the framework of $\infty$-categories, developed in the formalism of quasicategories in \cite{LurieHTT} and \cite{LurieHA}. Our notations and terminology will mostly conform to these references. We will generally regard ordinary categories as $\infty$-categories via their quasicategorical nerves. 
    
\subsubsection{Notation and $\infty$-categorical recollections}For the reader's benefit, we have collected key recurrent notions and mathematical notation below.
    \begin{enumerate}
    \item  $\cS$ denotes the $\infty$-category of spaces, otherwise known as $\infty$-groupoids. $\Cat$ denotes the $(\infty,2)$-category of small $(\infty,1)$-categories (see \hyperref[item:sizematters]{Item (h)}), functors, and natural transformations between them. $\Sp$ denotes the $\infty$-category of spectra.
    \item For a small $\infty$-category $\cc{C}$ we write $\Map(-,-)\colon \cc{C}^{\op} \times \cc{C} \to \cS$ to refer to the mapping space bifunctor. In the case that $\cc{C}$ is stable, we write $\Hom(-,-)\colon \cc{C}^{\op} \times \cc{C} \to \Sp$ to refer to the mapping spectrum bifunctor, which is specified by the relation $\Omega^{\infty + n}\Hom(x,y) \simeq \Map(x, \Omega^{n}Y)$.
    \item For a natural number $n$, $\cS_{\leq n}$ refers to the full subcategory of $\cS$ spanned by the $n$-truncated spaces. $\cS_{\leq -1}$ refers to the full subcategory spanned by $\emptyset$ and $\ast$, and $\cS_{\leq -2}$ refers to the full subcategory spanned by $\ast$. 
    \item For a natural number $n$, we write $[n] \in \Cat$ to refer to the free composable chain of morphisms which is of length $n$.\footnote{what is otherwise referred to as $\Delta^{n}$} Accordingly, the notation $\cc{C}^{[n]}$ denotes the $\infty$-category $\Fun([n], \cc{C})$ of composable chains of morphisms in $\cc{C}$ which are of length $n$. The notation $\cc{C}^{\simeq}$ will denote the maximal $\infty$-subgroupoid contained in $\cc{C}$. 
    \item We decorate functor $\infty$-categories $\Fun(\cc{C}, \cc{D})$ with the superscripts $\lex$, $\mm{rex}$ $\mm{ex}$, $\lim$, and $\colim$ to refer to the full subcategories of functors preserving finite limits, finite colimits, finite limits and finite colimits, small limits, and small colimits, respectively. 
    \item The superscripts $\mm{L}$ and $\mm{R}$ refer to the full subcategories of functors admitting left adjoints and right adjoints, respectively. When $\cc{C}, \cc{D}$ admit symmetric monoidal structures, the superscript $\otimes$ refers to the $\infty$-category of functors equipped with a strong symmetric monoidal structure, with morphisms strongly symmetric monoidal natural transformations.  
    \item Given a symmetric monoidal $\infty$-category $\cc{C}$, $\CAlg(\cc{C})$ denotes the $\infty$-category of $\bb{E}_{\infty}$-algebras in $\cc{C}$. When $\cc{C} =\Sp$ we simply write $\CAlg$ and refer to the constituent objects as commutative ring spectra. When the symmetric monoidal structure on $\cc{C}$ is not implicit, we decorate $\CAlg(\cc{C})$ with a superscript to indicate which one is under consideration, e.g., $\CAlg^{\times}(\cc{C})$ for the product (aka Cartesian) symmetric monoidal structure. 
    \item\label{item:sizematters} We assume the existence of a sequence of strongly inaccessible cardinals $\kappa_{0} < \kappa_{1} < \kappa_{2}$,  where $\kappa_{0}$ is the smallest strongly inaccessible cardinal and $\kappa_{i}$ is the smallest strongly inaccessible cardinal dominating $\kappa_{i-1}$. Under this assumption, the set $V_{\kappa_{i}}$ of sets with rank strictly less than $\kappa_{i}$ is a Grothendieck universe. We will refer to elements of the sets $V_{\kappa_{0}}$, $V_{\kappa_{1}}$, $V_{\kappa_{2}}$ as \emph{small}, \emph{large}, and \emph{very large} sets, respectively. Wherever relevant, the decoration $\widehat{-}$ signifies that the category under consideration consists of large objects, e.g., $\widehat{\cS}$ or $\widehat{\mm{Cat}}_{\infty}$ will denote the (very large) $\infty$-categories of large spaces or large $\infty$-categories, respectively. 
    \item $\PrL$ denotes the $\infty$-category of presentable $\infty$-categories and left adjoint functors. The default symmetric monoidal structure on $\PrL$ is that of the \emph{Lurie tensor product}, which corepresents functors out of the Cartesian product which are left adjoint independently in each variable. Given a cardinal $\kappa$, the notation $\mm{Pr^{L}_{\kappa}} \subseteq \PrL$ denotes the wide subcategory spanned by $\kappa$-accessible presentable $\infty$-categories with functors preserving $\kappa$-compact objects. We will refer to $\omega$-presentable $\infty$-categories as compactly generated presentable $\infty$-categories. $\PrLst$ denotes the $\infty$-category of presentable stable $\infty$-categories. 
    \item Given a small $\infty$-category $\cc{C}$, the notation $\cc{P}(\cc{C})$ refers to the category of presheaves on $\cc{C}$. This is the free cocompletion of $\cc{C}$ under small colimits, and is equivalent to the category of functors $\Fun(\cc{C}^{\op}, \cS)$. The Yoneda embedding $\cc{C} \to \cc{P}(\cc{C})$ is denoted by $\Yo$. The full subcategory $\Ind(\cc{C}) \subseteq \cc{P}(\cc{C})$ of ind-objects consists of the closure of the Yoneda image of $\cc{C}$ under filtered colimits.  The $\infty$-category of pro-objects of $\cc{C}$ is defined as $\Pro(\cc{C}) \coloneq \Ind(\cc{C}^{\op})^{\op}$. The $\infty$-categories $\Ind(\cc{C})$, $\Pro(\cc{C})$ are identified as the free cocompletion, respectively completion, of $\cc{C}$ under filtered colimits, respectively cofiltered limits. 
    \item An $\infty$-topos is a left exact localization of $\cc{P}(\cc{C})$ for $\cc{C}$ a small $\infty$-category. The notation $\LTop$ refers to the $\infty$-category of $\infty$-topoi with morphisms left-exact left adjoint functors. The notation $\RTop$ refers to the $\infty$-category of $\infty$-topoi with morphisms right adjoint functors admitting left-exact left adjoints.
    \item $\mm{Poset}$ denotes the category of small posets. Given $P \in \mm{Poset}$, if $P$ admits a greatest or a least element it is denoted $\top$ or $\bot$, respectively.
\end{enumerate}

\subsubsection{Other typographical miscellanea}\label{sssec:typo}  \tdef{This color modifier} is used to mark instances of important definitions or notations for the reader's convenience. When writing an adjunction between $\infty$-categories as \[L\colon\cc{C} \rightleftarrows \cc{D} \noloc R,\] it is always understood that $L$ is left adjoint to $R$, in symbols $L \dashv R$, unless otherwise specified.

\subsubsection{Comments on terminology}

The choice of \emph{(commutative) 2-ring} as a moniker for idempotent-complete stably symmetric monoidal $\infty$-categories is aligned with the terminology of Mathew's work \cite{Mathew16}, and is meant to refer to a single level of categorification of classical commutative rings/ring spectra. We recognize that the notation $\twoCAlg$ might initially suggest that we are referencing $(\infty,2)$-categorical structure, and uniform notation to disambiguate these would be highly desirable for future work.

The numerology of ``Zariski geometry'' as opposed to ``2-Zariski geometry'' for the resultant structure on this category may read as contradictory. Given an ordinary commutative ring $R$, any localization $R \to R[S^{-1}]$ is associated to an abelian-enriched symmetric monoidal localization between the one-object categories associated to $R, R[S^{-1}]$; in fact, these two classes of morphisms are equivalent. Similarly, given a $2$-ring $\cc{K}$, the admissible morphisms for the higher Zariski geometry from $\cc{K}$ are associated to certain symmetric monoidal exact localizations $\cc{K} \to \cc{K}[W^{-1}]$. As such, it is often instructive to regard the higher Zariski geometry as an \emph{extension} of the classical Zariski geometry to the $(\infty,1)$-category of ``rings with many objects''. In particular, it is \emph{not} additional structure that arises from passing to a higher category level. A classical instance of this heuristic is \cref{thm:hopkinsneemanbalmer}, due to Balmer and Hopkins--Neeman, which shows that the Zariski spectra of an ordinary commutative ring $R$ and of the 2-ring $\Perf_{R}$ coincide in a highly structured sense. 

Our terminological choice may create confusion for ring spectra, where the classical Zariski spectrum of a ring spectrum $R$ and the higher Zariski spectrum of $\Perf_{R}$ are often not equivalent. To avoid this, we encourage distinguishing between the Zariski geometry of a ring spectrum and that of its associated 2-ring of perfect complexes as much as possible.

\subsection*{Acknowledgements}

We would like to thank Martin Gallauer and Luca Pol for useful feedback on an earlier draft of this work.
KA thanks Shane Kelly for useful conversations.
AC is grateful to David Gepner for helpful conversations and constant encouragement.
TB is supported by the European Research Council (ERC) under Horizon Europe (grant No.~101042990).
The authors would like to thank the Max Planck Institute for Mathematics for its hospitality.
Some of this material is based on work of AC and TB conducted while in residence at the Hausdorff Research Institute for Mathematics during the Trimester Program: ``Spectral Methods in Algebra, Geometry, and Topology'' from September--December 2022, funded by the Deutsche Forschungsgemeinschaft under Germany’s Excellence Strategy – EXC-2047/1 – 390685813. 

\vspace{5pt}

\tableofcontents

\section{The \texorpdfstring{$\infty$}{∞}-category of Commutative 2-Rings}\label{sec:2rings}

In this section, we review the main features of the theory of commutative 2-rings, in particular, the notions of ideals and their localization theory. Commutative 2-rings form an enhancement of tensor triangulated categories, and they are the affine building blocks of the Zariski geometry constructed in the following sections. 

Although the perspectives behind some constructions are new and will be of importance later in this paper, by and large many of the results in this section are classical. In the interest of self-containment, we make use of the opportunity to collect them in one place.

\subsection{Basic notions}\label{ssec:2rings_basics}

Recall that $\Catex$ denotes the $\infty$-category of small stable $\infty$-categories and exact functors. Let $\Catperf$ denote the $\infty$-category of idempotent-complete stable $\infty$-categories and exact functors. By \cite[Lemma 2.20]{BGT}, $\Catperf$ appears as a reflective subcategory of $\Catex$  with reflection $\mdef{(-)^{\natural}}\colon \Catex \to \Catperf$ known as the \tdef{idempotent completion}, given by the formula $\cc{C} \mapsto \Ind(\cc{C})^{\omega}$. Unless otherwise specified, all constructions of this paper are understood to happen internally to $\Catperf$ and are implicitly idempotent completed. 

Following \cite[\S4.1.2]{ben-zviIntegralTransformsDrinfeld2010a}, recall that the adjunction \[\Ind\colon \Catperf \rightleftarrows \PrLst \noloc (-)^{\omega}\] supplies an equivalence $\Catperf \simeq \PrLstomega$. Under this identification, $\Catperf$ inherits a symmetric monoidal structure from the Lurie tensor product. This construction is characterized by an equivalence between the $\infty$-category of exact functors $\cc{C}_1\otimes\dotsb\otimes\cc{C}_n\to\cc{C}'$ and the $\infty$-category of functors $\cc{C}_1\times\dotsb\times\cc{C}_n\to\cc{C}'$ which are exact in each variable. In what follows we will exclusively consider this symmetric monoidal structure on $\Catperf$.

\begin{definition}\label{def:2ring}
    We let $\mdef{\twoCAlg}$ denote the $\infty$-category $\CAlg^{\otimes}(\Catperf)$ whose morphisms consist of symmetric monoidal exact functors. The objects of $\twoCAlg$ are referred to as \tdef{commutative $2$-rings} or simply \tdef{$2$-rings} with commutativity understood implicitly henceforth.
\end{definition}

Concretely, a $2$-ring is a small symmetric monoidal idempotent-complete stable $\infty$-category $\cc{K} = (\cc{K},\otimes)$ such that $\otimes$ is exact in both variables.  

\begin{remark}\label{rem:puntingreaders}
    Many of the results in this section, and indeed in \cref{sec:zariski}, will work without the assumption of idempotent-completeness. The descent results of \cref{sec:structuresheaf} will however fail, see \cref{ex:whyidempotentcomplete}. For this reason, our 2-rings will always be idempotent-complete.
  \end{remark}

\begin{remark}\label{rem:bigandsmall} Since $\Ind \dashv (-)^{\omega}$ provide mutually inverse equivalences between $\Catperf$ and $\PrLstomega$ which admit strong symmetric monoidal refinements by design, there is an induced equivalence
    \begin{equation*}
            \Ind(-)\colon\twoCAlg \simeq \CAlg^{\otimes}(\PrLstomega)\noloc (-)^{\omega}.
    \end{equation*} Given this identification, the reader may wonder why we choose to center ``small'' objects as opposed to ``big'' 2-rings of $\Ind$-objects in our theory. For us, it will sometimes be useful to regard categories of $\Ind$ objects as members of $\PrLst$ and not $\PrLstomega$. In words, we will occasionally work with functors between compactly generated presentably symmetric monoidal stable $\infty$-categories that do \emph{not} preserve compact objects, and hence could not have arisen from application of $\Ind$ to a map of 2-rings. As the basic theory will usually necessitate working with functors which \emph{do} arise from maps of 2-rings, we find it conceptually clearer to work with small objects and pass to large categories only when adding generality.  
\end{remark}

\begin{example}\label{ex:example2rings} We collect some familiar examples of $2$-rings below.
    \begin{enumerate}
        \item \cite[\S 7.2.4]{LurieHA}: Given $R \in \CAlg$ a commutative ring spectrum (or $\bb{E}_{\infty}$-ring spectrum), let $(\Mod_{R}, \otimes_{R})$ be the category of modules over $R$ equipped with the $R$-linear symmetric monoidal structure. The category $\Perf_{R}$ of \emph{perfect complexes over $R$} is defined to be the full monoidal subcategory of dualizable (or equivalently, compact) objects of $(\Mod_{R}, \otimes_{R})$. This forms a 2-ring, and is often the basic example under consideration. The construction $R \mapsto \Perf_R$ promotes to an adjunction
            \begin{equation*}
                \Perf_{(-)}\colon \CAlg \rightleftarrows \twoCAlg \noloc \End_{(-)}(\unit),
            \end{equation*}
        where the right adjoint sends $\cc{C} \in \twoCAlg$ to the endomorphism ring spectrum $\End_{\cc{C}}(\unit) \in \CAlg$, see \cref{rem:smallschwedeshipley}.
        \item \cite[\S 7.1.4]{LurieHA}: Let $\CAlg^{\heartsuit}$ be the category of discrete (or ordinary) commutative rings. As a special case of the previous example, if $R \in \CAlg^{\heartsuit}$, then $(\Mod_{R}, \otimes_{R}) \simeq (D(R), \otimes^{\bb{L}}_{R})$ and $\Perf_{R} \simeq D^{\mm{perf}}(R)$, the monoidal subcategory of perfect complexes\footnote{The superscript $\Mod_{R}^{\heartsuit}$ is used to disambiguate between the derived and underived categories for an ordinary ring, inspired by the observation that $\Mod_{R}^{\heartsuit}$ is the heart of the standard t-structure on $\Mod_{R}$.}. 
        \item Given $R \in \CAlg$ and $G$ a finite group, the category of finite $G$-representations over $R$, given by $\mm{Rep}(G, R) \coloneqq \Fun(BG, \Perf_{R})$ equipped with the pointwise tensor product, is a commutative $2$-ring. 
        \item\label{ex:example2rings3} \cite[\S 2]{lurieSpectralAlgebraicGeometry}: Given $X \in \mm{SpSch}$ a spectral scheme, we write $\QCoh_{X}$ to denote the symmetric monoidal category of quasi-coherent sheaves on $X$. The full monoidal subcategory $\Perf_{X} \subset \QCoh_{X}$ of perfect complexes on $X$ is a commutative 2-ring. 
        \item \cite[\S 5--6]{mathewNilpotenceDescentEquivariant2017}, \cite[\S 6]{balchin2024profiniteequivariantspectratensortriangular}: The symmetric monoidal category of finite genuine $G$-spectra $\Sp_{G}^{\omega}$ for a compact Lie or profinite group $G$ is a commutative 2-ring.
      
    \end{enumerate}
\end{example} \vspace{3pt}

\begin{proposition}\label{prop:2ring_compactgen}
    The $\infty$-category $\twoCAlg$ is compactly generated.
\end{proposition}
\begin{proof}
    By \cite[Corollary 4.2.5]{BGT}, $\Catperf$ is compactly generated, and the tensor product inherited from $\PrLstomega$ preserves colimits and compact objects in each variable. \cite[Corollary~5.3.1.17]{LurieHA} now implies that $\twoCAlg$ is compactly generated.
\end{proof}

The following construction will provide several more key examples of commutative 2-rings, including a family of compact generators for $\twoCAlg$. The requisite material on symmetric monoidal structures and the Lurie tensor product may be found in \cite[\S 4.8.1--4.8.2]{LurieHA}.
\begin{definition}
    Let $\cc{C} \in \Cat$. The \tdef{stabilization} of $\cc{C}$ is defined to be the $\infty$-category $\Fun(\cc{C}^{\op} , \Sp)$, and is denoted $\mdef{\Sp[\cc{C}]}$. From the universal property of $\cc{P}(\cc{C})$, this object may be identified with $\Fun^{\mm{R}}(\cc{P}(\cc{C})^{\op}, \Sp)$ which is canonically equivalent to $\Sp \otimes \cc{P}(\cc{C})$ for the Lurie tensor product on $\PrL$.
\end{definition}  

\begin{construction}\label{cons:freeforgetful} Recall that $\Sp$ admits a unique presentably symmetric monoidal structure preserving colimits in each variable which moreover restricts to compact objects, and is thus an element of $\CAlg(\PrLomega)$. This supplies an adjunction
    \begin{equation}\label{eq:stabilization}
        \Sp \otimes - \colon \PrLomega \rightleftarrows \Mod_{\Sp}(\PrLomega) \noloc \mathrm{fgt}
    \end{equation} 
 whose left adjoint admits a canonical symmetric monoidal refinement, and whose right adjoint is fully faithful and has image the full subcategory $\PrLstomega$ of compactly-generated presentable stable $\infty$-categories. 
Recall also that the functor of presheaves $\cc{P}(-)\colon \Cat \to \PrL$ admits a symmetric monoidal refinement. The induced functor $\cc{P}(-)\colon \CAlg^\times(\Cat) \to \CAlg(\PrL)$ sends a symmetric monoidal $\infty$ category $\cc{C}$ to $\cc{P}(\cc{C})$ equipped with the \emph{Day convolution monoidal structure}, which is the unique presentably symmetric monoidal structure on $\cc{P}(\cc{C})$ such that the Yoneda embedding $\Yo\colon \cc{C} \to \cc{P}(\cc{C})$ admits a symmetric monoidal refinement. 
 
 Now \cite[Proposition 2.8, Proposition 2.43]{BenMoshe2023} together imply that there is an adjunction of the form \begin{equation}\Sp[-]\colon\Cat \rightleftarrows \PrLstomega\noloc (-)^{\omega} \end{equation} whose unit is identified in \cite[Corollary 2.47]{BenMoshe2023}, and may be computed on any $\cc{C} \in \Cat$ as
    \[
        \Sigma^{\infty}_{+} \circ \Yo\colon \cc{C} \to \cc{P}(\cc{C})\to \Sp[\cc{C}]
    \] 
    namely the pointed stabilization of the Yoneda embedding. Furthermore, since the composite left adjoint $\Sp[-] \simeq \Sp \otimes \cc{P}(\cc{C})$ admits a symmetric monoidal refinement there is an induced adjunction
    \begin{equation}\label{eq:freeforgetful}
        \Sp[-]\colon \CAlg^{\times}(\Cat) \rightleftarrows \CAlg(\PrLstomega)\noloc (-)^{\omega},
    \end{equation}
    where the left adjoint is given by $\cc{C} \mapsto \Sp \otimes \cc{P}(\cc{C})$ equipped with the Day convolution monoidal structure.
\end{construction}

\begin{definition}\label{def:free2ring}
The \tdef{free 2-ring} functor is defined to be the composite \[\mdef{\Sp^{\omega}[-]}\colon\CAlg^{\times}(\Cat) \xrightarrow{\Sp[-]} \CAlg(\PrLstomega) \xrightarrow{(-)^\omega} \twoCAlg.\]
\end{definition}
 
\begin{proposition}\label{prop:free2ring}
There is a free-forgetful adjunction
\[
\Sp^{\omega}[-]\colon  \CAlg^{\times}(\Cat) \rightleftarrows \twoCAlg \noloc \mathrm{fgt}
\]
where the right adjoint forgets stability. The unit of the adjunction evaluated on any  $\cc{C} \in \CAlg^{\times}(\Cat)$ is given by the composite $\Sigma^{\infty}_{+} \circ \Yo\colon  \cc{C} \to \Sp^{\omega}[\cc{C}]\subseteq \Sp[\cc{C}]$.
\end{proposition}
\begin{proof}
    This is merely the adjunction described in \eqref{eq:freeforgetful} composed with the identification $(-)^\omega$ between $\CAlg(\PrLstomega)$ and $\twoCAlg$ discussed in \cref{rem:bigandsmall}.
\end{proof}

\begin{notation}
    Recall that $\Fin^{\simeq}$ with the coproduct symmetric monoidal structure is the free symmetric monoidal $\infty$-category on a single generator. We write $\mdef{\Sp^{\omega}\{x\}}$ to denote $\Sp^{\omega}[\Fin^{\simeq}]\in\twoCAlg$.    
\end{notation}

\begin{corollary}\label{cor:free_universal}
    For any $\cc{K}\in \twoCAlg$, the evaluation-at-$\Yo(*)$ functor $\Fun^{\mm{ex},\otimes}(\Sp^{\omega}\{x\},\cc{K}) \to \cc{K}$ is an equivalence.
\end{corollary}

\subsection{Thick tensor ideals and localization theory}\label{ssec:2rings_localization}

\begin{definition}\label{def:2ideals}
    Let $\cc{K}$ be a $2$-ring. A full replete subcategory $\cc{I}\subseteq\cc{K}$ is called a \tdef{thick tensor ideal} (henceforth \tdef{tt-ideal}), if the following conditions hold:
    \begin{enumerate}
        \item $\cc{I} = \mm{Thick}(\cc{I})$, i.e., $\cc{I}$ is closed under cofibers, shifts, and retracts;
        \item for any $x\in\cc{K}$ and $x'\in\cc{I}$, we have $x\otimes x'\in\cc{I}$.
    \end{enumerate}
    We write $\mdef{\Idl(\cc{K})}$ for the poset of tt-ideals ordered by inclusion. 
\end{definition}

\begin{notation}
 Given a subset $S\subset \cc{K}$, we write $\langle S \rangle$ to denote the tt-ideal generated by $S$. When $S = \{x\}$ is a singleton, we omit the braces from the notation.
\end{notation}

\begin{definition}
We say a tt-ideal $\cc{I} \subseteq \cc{K}$ is:
    \begin{enumerate}
        \item \tdef{principal} if $\cc{I} = \langle x \rangle$ for some $x \in \cc{K}$;
        \item \tdef{radical} if $x^{\otimes2}\in\cc{I}$ implies $x\in\cc{I}$ for any $x\in\cc{K}$. Equivalently, by induction, $x^{\otimes n} \in \cc{I}$ for some $n \geq 1$ implies $x \in \cc{I}$.
    \end{enumerate} We write $\mdef{\mm{Prin}(\cc{K})} \subset \Idl(\cc{K})$ to denote the poset of principal ideals and $\mdef{\Rad(\cc{K})} \subseteq \Idl(\cc{K})$  for the poset of radical ideals, ordered by inclusion.
\end{definition}

\begin{notation}
    Given a subset $S\subset \cc{K}$, we write $\mdef{\sqrt S}$ to denote the smallest radical tt-ideal containing $S$. The functor $\sqrt{-}$ assembles into a left adjoint to the inclusion $\Rad(\cc{K}) \hookrightarrow \Idl(\cc{K})$.
\end{notation}

\begin{example}
    Let $\cc{K}\in \twoCAlg$. The smallest tt-ideal of $\cc{K}$ is $\langle 0 \rangle$ and the largest tt-ideal is $\langle\unit\rangle=\cc{K}$. For any map of 2-rings $f\colon\cc{K}\to\cc{L}$, the full subcategory $\ker f \coloneq \cc{K} \times_{\cc{L}} \{0\}\subseteq \cc{K}$, namely the kernel of $f$, is a tt-ideal. This assembles into a functor $\mm{ker}(-)\colon  \twoCAlg_{\cc{K}/-} \to \Idl(\cc{K})$.
\end{example}

\begin{lemma}\label{lem:radicalgeneration}
    Let $\cc{K} \in \twoCAlg$ and $x$ and $y$ objects in $\cc{K}$. Then $\sqrt{x\otimes y}=\sqrt{x}\cap\sqrt{y}$.
\end{lemma}
\begin{proof}
    Since $x\otimes y\in\sqrt x\cap\sqrt y$, we have $\sqrt{x\otimes y}\subseteq\sqrt{x}\cap\sqrt{y}$. We check the converse. Pick $z\in\sqrt x\cap\sqrt y$. We can take $N$ such that $z^{\otimes N}\in\langle x\rangle\cap\langle y\rangle$ holds. From this, we see that
    \begin{equation*}
        z^{\otimes 2N} =z^{\otimes N}\otimes z^{\otimes N} \in\langle x\otimes y\rangle
    \end{equation*}
    holds and this implies $z\in\sqrt{x\otimes y}$.
\end{proof}

\begin{recollection}\label{rec:latticedefinitions} 
    Recall that a \tdef{complete lattice} is a poset $P$ admitting arbitrary colimits (or equivalently, admitting arbitrary limits). A \tdef{frame} is a complete lattice in which finite limits distribute over arbitrary colimits. A frame is \tdef{coherent} if the compact (aka~finite) objects are closed under the formation of finite limits, and every object is a colimit of compact objects. 
\end{recollection}

\begin{notation}
  When working in a distributive lattice, we write $\vee$ and $\wedge $ to denote the coproduct and product\footnote{often referred to as the \emph{join} and \emph{meet}, respectively}, respectively.
\end{notation}

\begin{proposition}\label{prop:latticeiscoherent}
  Given $\cc{K} \in \twoCAlg$, we have:
  \begin{itemize}
  \item[\emph{(a)}] $\Idl(\cc{K})$ is a complete lattice with compact objects $\Idl(\cc{K})^{\omega} = \mm{Prin}(\cc{K})$. 
  \item[\emph{(b)}] $\Rad(\cc{K})$ is a coherent frame with compact objects the radicals of principal tt-ideals.
  \end{itemize}
\end{proposition}
\begin{proof}
  It is easy to see that the intersection of tt-ideals is a tt-ideal, and thus $\Idl(\cc{K})$ admits finite limits. Given $\cc{I}, \cc{J} \in \Idl(\cc{K})$ their coproduct may be formed as $\langle \cc{I} \cup \cc{J} \rangle$. Any union of tt-ideals along a nested sequence of inclusions is a tt-ideal, and thus $\Idl(\cc{K})$ admits filtered colimits. As $\Idl(\cc{K})$ is a poset, these generate all colimits. For the inclusion $\mm{Prin}(\cc{K}) \subseteq \Idl(\cc{K})^{\omega}$ note that for any tt-ideal $\cc{I}$, $\langle a \rangle \in \cc{I}$ if and only if $a \in \cc{I}$, and $a \in \bigcup_{I} \cc{I}_{i}$ if and only if there exists $i$ such that $a \in \cc{I}_{i}$. A similar argument shows that the compact objects of $\Idl(\cc{K})$ are exactly the finitely generated tt-ideals. The reverse inclusion now follows from the fact that any finitely generated $\cc{I}_{i} = \langle a_{1},\dotsc,a_{n}\rangle$ satisfies $\cc{I}_{i} = \langle a_{1} \oplus \dotsb \oplus a_{n} \rangle$ and is thus principal. Part (b) is recorded in \cite[Theorem 3.1.9]{KockPitsch17}.  
\end{proof}
\begin{remark}
  We sketch the proof of part (b) above. Repetition of the argument of (a) coupled with \cref{lem:radicalgeneration} is enough to deduce the characterization of compact objects and closure under finite limits. The requisite generation property in $\Rad(\cc{K})$ follows from $\Rad(\cc{K}) \hookrightarrow \Idl(\cc{K})$ admitting a reflection $\sqrt{-}$ and principal ideals generating $\Idl(\cc{K})$ under colimits. Distributivity is proved as \cite[Lemma 3.1.8]{KockPitsch17} and this crucially employs radicality. 
\end{remark}

Thick tensor ideals are naturally associated to a well-studied class of localizations, which we now recall.

\begin{definition}\label{def:Karoubiloc} Let $\cc{C} \in \Catperf$, and $\cc{I} \subseteq \cc{C}$ a replete stable subcategory closed under retracts. Let $W$ denote the collection of all arrows of $\cc{K}$ whose cofiber lies in $\cc{I}$. The \tdef{Karoubi quotient by $\cc{I}$} is formed by taking the idempotent-completion of $\cc{K}[W^{-1}]$, the Dwyer--Kan localization at the class of morphisms $W$. We denote this construction by $\mdef{L_{\cc{I}}\colon\cc{K} \to \cc{K}/\cc{I}}$ or simply $\cc{K}/\cc{I}$.
\end{definition}

\begin{warning}\label{warn:nonstandard}  Note that the notation of \cref{def:Karoubiloc} is traditionally used to denote the Dwyer--Kan localization as an object of $\Cat^{\mathrm{ex}}$, without idempotent completion. Since we will never leave the idempotent complete setting, we continue to use this notation without fear of confusion.
\end{warning}
   
The terminology above is first introduced in \cite[Appendix A]{calmesHermitianKtheoryStable2025}. The essential features of this construction in the triangulated setting are classical, while \cite[\S 5]{BGT} revisits the same higher-categorically. In \cite[\S 1.3]{nikolausTopologicalCyclicHomology2018a} the Dwyer--Kan localization above is shown to admit a unique symmetric monoidal refinement prior to idempotent completion. The proposition below is essentially a restatement of the results of \emph{loc.\ cit.}\ in the idempotent-complete setting. 

\begin{proposition}\label{prop:verdexists}
  Let $\cc{K} \in \Catperf$ and $\cc{I}\subseteq \cc{K}$ a replete stable subcategory closed under retracts. We have that:
    \begin{enumerate}
        \item[\emph{(a)}] The Karoubi quotient $\cc{K}/\cc{I}$ is a stable $\infty$-category, and the quotient $L_{\cc{I}}\colon\cc{K} \to \cc{K}/\cc{I}$ is exact.
        \item[\emph{(b)}]\label{prop:verdexistsb} The induced functor $\Ind(\cc{K}) \to \Ind(\cc{K}/\cc{I})$ is a localization in the sense of \cite[Definition 5.2.7.2]{LurieHTT}, i.e., it has a fully faithful right adjoint, and moreover has kernel $\Ind(\cc{I})$. The right adjoint is given by the unique colimit preserving functor $\Ind(\cc{K}/\cc{I}) \to \Ind(\cc{K})$ which on $\cc{K}/\cc{I} \to \Ind(\cc{K}) \subset \PShv(\cc{K})$ is the Yoneda embedding $x \mapsto \Map_{\cc{K}/\cc{I}}(-,x)$.  
        \item[\emph{(c)}]\label{prop:verdexistsc} Suppose $\cc{K} \in \twoCAlg$ and $\cc{I}\subseteq \cc{K}$ is a thick tensor ideal. There is a unique way to simultaneously endow the category $\cc{K}/\cc{I}$ and the functor $\cc{K} \to \cc{K}/\cc{I}$ with a stably symmetric monoidal structure. If $\cc{L}$ is another $2$-ring, then composition with $\cc{K} \to \cc{K}/\cc{I}$ induces an equivalence between $\Fun^{\mathrm{ex},\otimes}(\cc{K}/\cc{I},\cc{L})$ and the full subcategory of $\Fun^{\mathrm{ex},\otimes}(\cc{K},\cc{L})$ on those functors which send the objects of $\cc{I}$ to $0$.
    \end{enumerate}
\end{proposition}
\begin{proof}
  Let $W \subseteq \cc{K}$ denote the collection of morphisms with cofiber contained in $\cc{I}$. In \cite[Theorem I.3.3]{nikolausTopologicalCyclicHomology2018a} it is shown that $\cc{K}[W^{-1}]$ is stable and $\cc{K} \to \cc{K}[W^{-1}]$ is exact. It follows that the composite $\cc{K} \to \cc{K}[W^{-1}] \to \Ind(\cc{K}[W^{-1}])^{\omega} \eqcolon \cc{K}/\cc{I}$ is an exact functor of stable $\infty$-categories, yielding part (a). Part (b) is demonstrated in \cite[Proposition I.3.5]{nikolausTopologicalCyclicHomology2018a} with $\cc{K}/\cc{I}$ replaced by $\cc{K}[W^{-1}]$, and the claim follows from noting that idempotent completion induces an equivalence on $\Ind$-objects.

  The conclusions of part (c) are demonstrated for the map $\cc{K} \to \cc{K}[W^{-1}]$ in \cite[Theorem I.3.6]{nikolausTopologicalCyclicHomology2018a}. The claim will follow if we demonstrate that the idempotent-completion map $\cc{K}[W^{-1}] \to \cc{K}/\cc{I}$ admits a unique stably symmetric monoidal structure such that the induced map
\[ \Fun^{\otimes, \mm{ex}}(\cc{K}, \cc{L}) \to \Fun^{\otimes, \mm{ex}}(\cc{K}[W^{-1}], \cc{L})
    \]
    is an equivalence for every $\cc{L} \in \twoCAlg$; this is a special case of \cite[Proposition 4.8.1.10]{LurieHTT}.
\end{proof}

We may rephrase the properties above in the following pleasant fashion. 

\begin{proposition}\label{prop:quotientposet}
For any $\cc{K} \in \twoCAlg$, there exists an adjunction as below 
    \[
        \cc{K}/-\colon \Idl(\cc{K}) \rightleftarrows \twoCAlg_{\cc{K}/-}\noloc \mm{ker}(-),
    \]
where the left adjoint is fully faithful, where the right adjoint preserves filtered colimits, and where $\cc{I} \in \Idl(\cc{K})$ satisfies $\cc{K}/\cc{I} \in (\twoCAlg_{\cc{K}/-})^{\omega}$ if and only if $\cc{I}$ is principal.
\end{proposition}
\begin{proof}
Let $\cc{I} \in \Idl(\cc{K})$. For any map of 2-rings $f\colon \cc{K} \to \cc{L}$, the fiber of $\Fun^{\otimes, \mm{ex}}(\cc{K}/\cc{I}, \cc{L}) \to \Fun^{\otimes, \mm{ex}}(\cc{K}, \cc{L})$ over $\{f\}$ may be naturally identified with $\Fun^{\otimes, \mm{ex}}_{\cc{K}/-}(\cc{K}/\cc{I}, \cc{L})$ and hence \hyperref[prop:verdexistsc]{\cref*{prop:verdexists}\,(c)} gives the expression
\begin{equation}
\Fun^{\otimes, \mm{ex}}_{\cc{K}/-}(\cc{K}/\cc{I}, \cc{L}) \simeq \Map_{\Idl(\cc{K})}(\cc{I}, \mm{ker}(f)) \simeq \begin{cases}
    \ast & \cc{I} \subseteq \mm{ker}(f); \\
    \emptyset & \text{otherwise.}
\end{cases}
\end{equation} yielding the first part of the claim. We now show that $\ker(-)\colon  \twoCAlg \to \Idl(\cc{K})$ preserves filtered colimits. Let $F\colon  I \to \twoCAlg_{\cc{K}/-}$ any filtered system receiving a map $f = \varinjlim_{I}f_{i}\colon \cc{K} \to \varinjlim_{I} F$. Let $x \in \mm{ker}(f)$ arbitrary, it suffices to show that there exists $i \in I$ such that $x \in \mm{ker}(f_{i})$. Equivalently, we must show that there is some $i$ for which $f_{i}$ sends the homotopy class of the automorphism $\id_x$ to the $0$ map. Recall that the forgetful functor $\twoCAlg \to \Cat$ and the functor which passes to underlying $\infty$-groupoids both preserve filtered colimits. We thus have the identification of automorphism groupoids 
    \begin{equation*}
            B\mm{Aut}(f(x)) \simeq {\varinjlim}_{I}B\mm{Aut}_{F(i)}(f_{i}(x))
    \end{equation*}
whence it follows that the homotopy class of $[\id_{x}]$ must map to $0$ in some $F(i)$, by commutation of homotopy groups with filtered colimits. The final claim follows from the exhibition of $\cc{K}/-$ as a fully faithful left adjoint, the fact that the right adjoint preserves filtered colimits, and \cref{prop:latticeiscoherent}. 
\end{proof}

\begin{corollary}\label{cor:basechange}
Let $\cc{K}, \cc{L} \in \twoCAlg$, $f\colon \cc{K} \to \cc{L}$ a morphism and $\cc{I}\in \mm{Idl}(\cc{K})$. The following diagram in $\twoCAlg$ is co-Cartesian
\begin{equation*}
        \xymatrix{
            \cc{K} \ar[r] \ar[d] &
            \cc{L} \ar[d] \\
            \cc{K}/\cc{I}\ar[r] &
            \cc{L}/\langle f(\cc{I}) \rangle.
        }
    \end{equation*}
\end{corollary}
\begin{proof} It is equivalent to show the above is co-Cartesian in $\twoCAlg_{\cc{K}/-}$. Let $\cc{K}' \in \twoCAlg$ arbitrary. From \cref{prop:quotientposet} one has that the fiber of the projection 
\begin{equation}\label{eq:basechangeproj}
\Fun^{\otimes,\mm{ex}}(\cc{K}/\cc{I}, \cc{K}') \times_{\Fun^{\otimes, \mm{ex}}(\cc{K}, \cc{K}')} \Fun^{\otimes,\mm{ex}}(\cc{L}, \cc{K}') \to \Fun^{\otimes,\mm{ex}}(\cc{L}, \cc{K}')
\end{equation}
over any $\{g\} \subseteq \Fun^{\otimes,\mm{ex}}(\cc{L}, \cc{K}')$ may be naturally identified with $\Map_{\Idl(\cc{K})}(\cc{I}, \mm{ker}(g \circ f))$. Since $\Idl(\cc{K})$ is a poset, this latter object is either $\emptyset$ or $\ast$, whence it follows that \eqref{eq:basechangeproj} is an inclusion of connected components of underlying simplicial sets whose essential image is exactly the full subcategory of functors $g\colon \cc{L} \to \cc{K}'$ satisfying $(g \circ f) \colon \cc{I} \mapsto 0$. It follows that there is a natural map $\cc{K}/\cc{I} \otimes_{\cc{K}} \cc{L} \to \cc{L}/\langle f(\cc{I})\rangle$, which moreover induces an equivalence on $\Fun^{\otimes, \mm{ex}}(-, \cc{K}')$ by \hyperref[prop:verdexistsc]{\cref*{prop:verdexists}\,(c)}. As $\cc{K}'$ was allowed to be arbitrary, we may conclude by the Yoneda lemma.
\end{proof}

\begin{corollary}\label{cor:universalquotient}
    Given $\cc{K} \in \twoCAlg$ and $a \in \cc{K}$, the following diagram in $\twoCAlg$ is co-Cartesian 
    \[
    \xymatrix{
        \Sp^{\omega}\{x\} \ar[r]^{x \mapsto 0} \ar[d]_{x \mapsto a} & \Sp^{\omega} \ar[d] \\
        \cc{K} \ar[r] & \cc{K}/\langle a \rangle.
    }
    \]
\end{corollary}
\begin{proof}
  The top of the square induces $\ast \simeq \Fun^{\otimes, \mm{ex}}(\Sp^{\omega}, -) \hookrightarrow  \Fun^{\otimes, \mm{ex}}(\Sp^{\omega}\{x\}, -)$, which is a fully faithful inclusion onto those functors which send $\{x \mapsto 0\}$. As in the proof of \cref{cor:basechange}, we deduce that the top horizontal morphism is a Karoubi quotient, and \cref{cor:basechange} supplies the claim. 
\end{proof}

\begin{corollary}\label{cor:retractsandleftcancel}
    Karoubi quotients in $\twoCAlg$ are closed under compositions, retracts, and left cancellation.
\end{corollary}
\begin{proof}
  We first show closure under left cancellation, namely that given $\cc{K} \in \twoCAlg$ and morphisms $\cc{K} \xrightarrow{f} \cc{K}' \xrightarrow{g} \cc{K}''$ such that $f, g \circ f$ are equivalent to Karoubi quotients at tt-ideals $\cc{I}, \cc{J} \subset \cc{K}$, respectively, then $g\colon \cc{K}'\to \cc{K}''$ is itself a Karoubi quotient. By assumption, $g$ is equivalent to a morphism $\cc{K}/\cc{I} \to \cc{K}/\cc{J}$, and therefore \cref{prop:quotientposet} implies that $\cc{I} \subset \cc{J}$. By the universal property of Karoubi quotients \hyperref[prop:verdexistsc]{\cref*{prop:verdexists}\,(c)}, the natural map $(\cc{K}/\cc{I})/\langle f(\cc{J})\rangle \to \cc{K}/\cc{J}$ supplied by \cref{cor:basechange} is easily seen to induce an equivalence of corepresentable functors in $\twoCAlg_{\cc{K}/}$, from which the claim follows. Closure under composition proceeds the same way, by instead producing a map $\cc{K}/\cc{J} \to (\cc{K}/\cc{I})/\langle f(\cc{J})\rangle$. We now show closure under retracts formed in $\twoCAlg^{[1]}$. Consider a homotopy commutative retract diagram
        \begin{equation*}
        \begin{aligned}
            \xymatrix{
                \cc{K}' \ar[r] \ar[d] &
                \cc{K} \ar[r]^{f} \ar[d] &
                \cc{K}' \ar[d] \\
                \cc{K}'' \ar[r] &
                \cc{K} /\cc{I} \ar[r] &
                \cc{K}'' \rlap{.}
            }
            \end{aligned}
          \end{equation*}
          with notation as above. \cref{cor:basechange} supplies a factorization of the bottom right horizontal arrow as $\cc{K}/\cc{I} \to \cc{K}'/\langle f(\cc{J}) \rangle \to \cc{K}''$, implying that $\cc{K}''$ is a retract of $\cc{K}'/f(J)$. By \cref{prop:quotientposet} the functor $\cc{K}'/$ is fully faithful, and since $\Idl(\cc{K}')$ is a poset, all retracts are isomorphisms so $\cc{K}'' \simeq \cc{K}'/\langle f(\cc{J})\rangle$.          
\end{proof}

\subsection{Rigidity}\label{ssec:2rings_rigidification}

We now restrict our attention to a distinguished class of well-behaved 2-rings, whose unique behavior with respect to localization theory as in \cref{ssec:rigidlocalizationtheory} will be of essential importance in later sections.

\begin{definition}\label{def:rigid}
    We say a symmetric monoidal $\infty$-category is \tdef{rigid} if every object is dualizable. We write $\mdef{\twoCAlgrig} \subset \twoCAlg$ for the full subcategory spanned by rigid 2-rings.
\end{definition}

\begin{remark}\label{rem:horigid}
    A $2$-ring $\cc{K}$ is rigid if and only if $\Ho\cc{K}$ is rigid.
\end{remark}

\begin{lemma}[{\cite[2.1.3]{hoveyAxiomaticStableHomotopy1997a}}]\label{lem:dualizablesclosedunderthicksub}
  Let $\cc{K} \in \twoCAlg$, and $\cat{S} \subseteq \cc{K}^{\mm{dbl}}$ any collection of dualizable objects. Then $\mm{Thick}(\cat{S}) \subseteq \cc{K}^{\mm{dbl}}$, where $\mm{Thick}(\cat{S})$ refers to the closure of $\cat{S}$ under cofibers, shifts, and retracts. 
\end{lemma}

Consequently, for any 2-ring $\cc{K}$ the full subcategory of dualizable objects $\cc{K}^{\mathrm{dbl}} \subseteq \cc{K}$ is stable and idempotent complete. As the tensor product of any two dualizable objects is itself dualizable, the symmetric monoidal structure on $\cc{K}$ must restrict to one on $\cc{K}^{\mm{dbl}}$, so $\cc{K}^{\mm{dbl}}$ is itself a member of $\twoCAlgrig$. 

\begin{proposition}\label{prop:rigidclosureundercolims}
    $\twoCAlgrig$ is a coreflective subcategory of $\twoCAlg$, with coreflection $\cc{K} \mapsto \cc{K}^{\mm{dbl}}$. 
\end{proposition}
\begin{proof}
Let $\cc{K} \in \twoCAlgrig$. For any $\cc{L}\in \twoCAlg$, the inclusion $\cc{L}^{\mm{dbl}} \hookrightarrow \cc{L}$ induces an equivalence $\Map_{\twoCAlg}(\cc{K}, \cc{L}^{\mm{dbl}}) \simeq \Map_{\twoCAlg}(\cc{K}, \cc{L})$ since the image of any symmetric monoidal functor $\cc{K} \to \cc{L}$ must land in the full subcategory $\cc{L}^{\mm{dbl}}$. We conclude by \cite[Corollary 5.2.2.7]{LurieHTT}.
\end{proof}

Finally, we note the behavior of the ``big'' category associated to a rigid 2-ring.  

  \begin{lemma}\label{lem:rigidlycompactlygenerated}
    For $\cc{K} \in \twoCAlgrig$ there is an identification $\Ind(\cc{K})^{\omega} = \Ind(\cc{K})^{\mm{dbl}} \simeq \cc{K}$. 
  \end{lemma}
  \begin{proof}
    The identification $\Ind(\cc{K})^{\omega} = \cc{K}$ follows from the fact that $\cc{K}$ is closed under finite colimits and is idempotent complete. This provides an inclusion $\cc{K} \simeq \Ind(\cc{K})^{\omega} \subseteq \Ind(\cc{K})^{\mm{dbl}}$ by rigidity. The reverse inclusion is a standard argument using the fact that $\unit \in \Ind(\cc{K})^{\omega}$, see for example the proof of \cite[Theorem 2.1.3]{hoveyAxiomaticStableHomotopy1997a}.
  \end{proof}

\begin{example}
    As an example, consider the category $\Mod_R$, which is compactly generated by dualizable objects (as it is compactly generated by elements in the thick subcategory of the unit). The above implies that $\Mod_{R}^{\omega} = \Mod_{R}^{\mm{dbl}} \eqqcolon \Perf_R$.
\end{example}

\begin{definition}
    We say that $\cc{C} \in \CAlg(\PrLstomega)$ is \tdef{rigidly compactly generated} if $\cc{C}^\omega = \cc{C}^\mm{dbl}$, i.e., if the compact and dualizable objects of $\cc{C}$ agree. We will write $\mdef{\CAlg(\PrLstomega)^{\rig}} \subset \CAlg(\PrLstomega)$ to refer to the full subcategory of such objects.
\end{definition}

\cref{lem:rigidlycompactlygenerated} alongside \cref{rem:bigandsmall} imply that passage to $\Ind$ objects provides an equivalence between the $\infty$-category of rigid 2-rings and the $\infty$-category of rigidly compactly generated presentably symmetric monoidal $\infty$-categories.

\subsection{Localization theory in the rigid setting}\label{ssec:rigidlocalizationtheory}

In this subsection we highlight the features of localization theory which are unique to the rigid setting. 

\begin{lemma}\label{lem:rigid_radicalideal}
    If a $2$-ring~$\cc{K}$ is rigid, then $\Rad(\cc{K}) = \Idl(\cc{K})$. 
\end{lemma}
\begin{proof}
    Let $\cc{I}$ be an ideal. We need to show that any $x\in\cc{K}$ satisfying $x^{\otimes2}\in\cc{I}$ satisfies $x\in\cc{I}$. Let $x^{\vee}$ denote its dual. As $\cc{I}$ is an ideal, we have $x\otimes x\otimes x^{\vee}\in\cc{I}$. Since $x$ is a retract of $x\otimes x^{\vee}\otimes x$, this implies $x\in\cc{I}$.
\end{proof}

\begin{lemma}\label{lem:rigidradicalquotient}
    If $\cc{K} \in \twoCAlgrig$, then for any tt-ideal $\cc{I} \subseteq \cc{K}$, the Karoubi quotient $\cc{K}/\cc{I} \in \twoCAlgrig$.  
\end{lemma}
\begin{proof}
The Neeman--Thomason localization theorem \cite[Proposition 1.5]{BGT} implies that the functor $\cc{K} \to \cc{K}/\cc{I}$ is essentially surjective after idempotent completion. Since the image of any $x \in \cc{K}$ must be dualizable in $\cc{K}/\cc{I}$, we may apply \cref{lem:dualizablesclosedunderthicksub} and conclude.
\end{proof}

\begin{corollary}\label{cor:rigidquotientposet}
For any $\cc{K} \in \twoCAlgrig$, there exists an adjunction as below 
\[
\cc{K}/-\colon  \Idl(\cc{K}) \rightleftarrows \twoCAlgrig_{\cc{K}/-}\noloc \mm{ker}(f)
\]
where the left adjoint is fully faithful. The right adjoint moreover preserves filtered colimits, and $\cc{I} \in \Idl(\cc{K})$ satisfies $\cc{K}/\cc{I} \in (\twoCAlgrig_{\cc{K}/-})^{\omega}$ if and only if $\cc{I}$ is principal.
\end{corollary}
\begin{proof}
The right adjoint is given by the inclusion $\twoCAlgrig \hookrightarrow \twoCAlg$ and the right adjoint of \cref{prop:quotientposet}. That this preserves filtered colimits is an immediate consequence of \cref{prop:rigidclosureundercolims}.
\end{proof}

\begin{definition}\label{rec:smashingloc} Let $\cat{C}$ be a symmetric monoidal $\infty$-category.
  \begin{enumerate}
  \item An \tdef{idempotent algebra} of $\cat{C}$ is an object $A \in \cat{C}_{\unit/}$ such that $A \otimes(\unit \to A)$ is an isomorphism. Such an object is uniquely a commutative algebra by \cite[Proposition 4.8.2.9]{LurieHA}, and we write $\mdef{\mm{Idem}(\cat{C})}$ to denote the full subcategory of $\CAlg(\cc{C})$ consisting of idempotent algebras. Note that this is a poset, as the codiagonal $A \otimes A \to A$ is an equivalence for any $A \in \Idem(\cc{C})$ and thus the functor $\Map_{\CAlg(\cc{C})}(A,-)$ takes values in $\cS_{\leq -1}$.
  \item  A \tdef{smashing localization} of $\cat{C}$ is a symmetric monoidal localization $L\colon  \cat{C} \rightleftarrows \cat{D}\noloc R$ witnessing $\cat{D} \simeq \Mod_{A}(\cc{C})$ for $A \in \mm{Idem}(\cat{C})$.
\end{enumerate}
\end{definition}

We recall the following foundational fact due to Miller, recorded as \cite[Theorem 3.3.3]{hoveyAxiomaticStableHomotopy1997a}.

\begin{proposition}\label{prop:finitelocissmashing}
  Given $\cc{K} \in \twoCAlgrig$ and $\cc{I}\subset \cc{K}$ a tt-ideal, the localization $L_{\cc{I}}\colon \Ind(\cc{K}) \rightleftarrows \Ind(\cc{K}/\cc{I})\colon  R_{\cc{I}}$ is smashing, and the tensor idempotent $A$ is equivalent to $R_{\cc{I}}L_{\cc{I}}\unit$.
\end{proposition}

In the interest of self containment we have decided to include a proof of the above; before this, let us discuss our desired consequence. Let $\cc{K}\in \twoCAlgrig$ as before. Recall that \cite[Corollary 4.8.5.21]{LurieHA} provides a fully faithful embedding
\[
  \Mod\colon  \CAlg(\Ind(\cc{K})) \to \CAlg(\PrLstomega)_{\Ind(\cc{K})/}.
  \]
  We write $\Mod^{\omega}\colon \CAlg(\Ind(\cc{K})) \to \twoCAlg_{\cc{K}/}$ to denote the passage to compact objects. 
  
  \begin{notation}\label{not:idempotents}
      For a 2-ring $\cc{K}$, we will write $\mm{Idem}_{\cc{K}} \coloneq \mm{Idem}(\Ind(\cc{K}))$. 
  \end{notation}
  
  \begin{definition}\label{def:finiteidemp}
      Let $\cc{K}$ be a 2-ring. We say an idempotent algebra $A \in \mm{Idem}_{\cc{K}}$ is \tdef{finite} if its associated smashing localization is equivalent to $\Ind(\cc{K}) \to \Ind(\cc{K}/\cc{I})$ for some $\cc{I} \in \Idl(\cc{K})$. We write $\mdef{\mm{Idem^{fin}_{\cc{K}}}} \subseteq \mm{Idem}_{\cc{K}}$ to denote the full subcategory of finite idempotent algebras.
  \end{definition}
  
  The following is an easy consequence of \cref{prop:finitelocissmashing}, compare \cite[Theorem 3.5]{Balmer2011}.
  
  \begin{corollary}\label{cor:fidemtoideal}
    The functor $\Mod^{\omega}$ supplies an equivalence between $\Idem_{\cc{K}}^{\mm{fin}}$ and the full subcategory of $\twoCAlg_{\cc{K}/}$ on objects of the form $\cc{K} \to \cc{K}/\cc{I}$ for $\cc{I} \in \Idl(\cc{K})$. In particular, the composite of $\Mod^{\omega}$ with the functor $\ker(-)$ of \cref{prop:quotientposet} supplies an equivalence $\mm{Idem^{fin}_{\cc{K}}}  \simeq \Idl(\cc{K})$. 
  \end{corollary}

      \begin{recollection}\label{rec:projectionformula}
Let $\cat{E}$, $\cat{F}$ be symmetric monoidal categories and $L'\colon  \cat{E} \rightleftarrows \cat{F} \noloc R'$ an adjunction datum with $L'$ symmetric monoidal. For $x \in \cat{E}$, $y \in \cat{F}$, consider the map \begin{equation}\label{eq:projformula} p\colon  x \otimes R'(y) \to R'(L'(x) \otimes y) \end{equation} which is adjoint to $\mm{id}_{L'(X)} \otimes \epsilon\colon  L'(X\otimes R'(Y)) \simeq L'(X) \otimes L'R'(Y) \to L'(X) \otimes Y$ where $\epsilon$ is the counit of $L' \dashv R'$. We say $L' \dashv R'$ \tdef{satisfies the projection formula} if the map $p$ of \eqref{eq:projformula} is an equivalence for every $x, y$ as above. If $L'$ is a localization, then $L' \dashv R'$ satisfies the projection formula if and only if $L'$ is smashing, see for example \cite[Proposition 5.1.4]{carmeliAmbidexterityHeight2021}.
\end{recollection}
  
\begin{proof}[Proof of \cref{prop:finitelocissmashing}]
  By \cref{rec:projectionformula}, it suffices to show that the pair $L_{\cc{I}} \dashv R_{\cc{I}}$ satisfies the projection formula. Let $x \in \Ind(\cc{K})^{\omega}, y \in \Ind(\cc{K}/\cc{I})^{\omega}$. Note that both $x$ and $y$ are dualizable by \cref{lem:rigidlycompactlygenerated} and \cref{lem:rigidradicalquotient}. Since $L_{\cc{I}}$ is symmetric monoidal, one has the following commuting diagram:
  \[\xymatrix{
      \Map(-\otimes x^{\vee}, R_{\cc{I}}y) \ar[r] & {\Map(L_{\cc{I}}(-\otimes x^{\vee}), y)} \ar[r]& {\Map(L_{\cc{I}}(-),L_{\cc{I}}x \otimes y)}\ar[d] \\
    {\Map(-, x \otimes R_{\cc{I}}y)} \ar[u] \ar[rr]^{\text{\eqref{eq:projformula}}} && \Map(-, R_{\cc{I}}(L_{\cc{I}}x \otimes y))
    }\]    
  where the vertical and top arrows are all equivalences, implying that the bottom arrow is also an equivalence and the projection formula holds in this case. \cref{prop:verdexists} implies that $L_{\cc{I}}$ and $R_{\cc{I}}$ are both colimit-preserving, and hence one deduces the same for all objects in $x \in \Ind(\cc{K})$ and $y \in \Ind(\cc{K}/\cc{I})$. 
\end{proof}
\vspace{5pt}

\section{Geometries}\label{sec:geometries}

This section will serve as a more detailed account of the prerequisite material from \cite{LurieDAG5}; with the exception of \cref{ssec:dirac}, all ideas in this section are due to Lurie. 

\subsection{Geometries and locality}\label{ssec:geometries_basics}

\begin{definition}\label{def:geometry}
    Let $\cG$ be a small $\infty$-category, $\cG^{\ad}$ a wide subcategory of~$\cG$, and $\tau$ a Grothendieck topology on~$\cG$. We say that $(\cG^{\ad},\tau)$ is an \tdef{admissibility structure} on~$\cG$ or that $(\cG,\cG^{\ad},\tau)$ is a \tdef{geometry} if the conditions below are satisfied.
    \begin{enumerate}
        \item\label{i:g} $\cG$ has finite limits and is idempotent complete.
        \item\label{i:gen} $\tau$ is generated by morphisms in $\cG^{\ad}$.
        \item\label{i:bc} $\cG^{\ad}$ is closed under base changes in~$\cG$.
        \item\label{i:canc} If $f$ and $g$ are composable morphisms in~$\cG$ satisfying $g$, $g\circ f\in\cG^{\ad}$, then $f\in\cG^{\ad}$ as well.
        \item\label{i:ret} If $f$ is a retract of~$g$ in $\cG^{[1]}$ satisfying $g\in\cG^{\ad}$, then $f\in\cG^{\ad}$ as well.
    \end{enumerate}
    We refer to morphisms in~$\cG^{\ad}$ and covers in~$\tau$ as admissible morphisms and admissible covers, respectively.
\end{definition}

\begin{example}\label{ex:discretegeometry}
    Let $\cG$ be an arbitrary idempotent-complete $\infty$-category having finite limits. There is an admissibility structure where the only admissible morphisms are equivalences and the topology is discrete. We call this the \tdef{discrete geometry} and write $\mdef{\cG_{\disc}}$ for it.
\end{example}

\begin{definition}\label{def:geometry_transformation}
    A morphism between geometries is a functor between the underlying $\infty$-categories which preserves finite limits, admissible morphisms, and admissible covers.
\end{definition}

\begin{example}\label{ex:discretetransformation}
    Let $\cG$ be a geometry. Then the identity functor determines a morphism $\cG_{\disc}\to\cG$ from its underlying discrete geometry.
\end{example}

  Given an admissibility structure on $\cc{G}$, objects in $\Ind(\cc{G}^{\op})$ admit a natural notion of locality with respect to the admissibility structure. This will be an instance of a more general notion of locality for $\Ind(\cG^{\op})$-valued sheaves on an arbitrary $\infty$-topos, which we discuss in the next subsection. 
  
  To motivate what follows, let us study the key example of the Zariski geometry on commutative ring spectra, following \cite{LurieDAG5} and \cite{lurieDerivedAlgebraicGeometry2011a}. The notion of locality discussed above will recover the notion of a \emph{local ring spectrum}. For the reader who is uncomfortable with ring spectra, we remark that no validity is lost in opting to mentally substitute everywhere the category of ordinary commutative rings.

\begin{definition}\label{def:classicalzariski}
    The \tdef{classical Zariski geometry} consists of the following data:
        \begin{enumerate}
            \item $\GZar=(\CAlg^{\omega})^{\op}$, the opposite of the category of compact $\bb{E}_{\infty}$ rings.
            \item Admissible morphisms correspond to localization maps $R \to R[x^{-1}]$ for $x \in \pi_{0}R$
            \item   A finite collection $\{R\to R[x_i^{-1}]\}_{i \in I}$ is declared to generate a covering sieve if the set $\{x_i\}_{i \in I} \subset \pi_{0}R$ generates the unit ideal. 
        \end{enumerate}
\end{definition}

A commutative ring spectrum $R \in \CAlg$ is said to be \tdef{local} if $\pi_{0}R$ is a local ring, namely that it admits a unique maximal ideal $\mathfrak{m} \subset \pi_{0}R$. This condition is equivalent to demanding that for every collection $f_{1},\dotsc,f_{n} \in \pi_{0}R$ satisfying $f_{1} + \dotsb + f_{n} = 1$, there is some $i$ so that $f_{i}$ is a unit\footnote{Let $S \in \CAlg^{\heartsuit}$ be an arbitrary commutative ring. We claim that $S \setminus S^{\times}$ is additively closed only if it is an ideal (and hence necessarily maximal). Recall that $S \setminus S^{\times}$ is exactly the set of elements of $S$ which are contained in some maximal ideal. Suppose $\mathfrak{m}, \mathfrak{m}'$ are distinct maximal ideals of $S$. The ideal sum $\mathfrak{m} + \mathfrak{m}' = S$ must contain a unit, and so there are nonunit elements $f \in \mathfrak{m},\ f' \in \mathfrak{m}'$ which sum to $1$. Hence $S \setminus S^{\times}$ is not additively closed in this case.}.  Using this latter formulation and the universal property of Zariski localization, the following lemma is not difficult to show.

\begin{lemma}\label{lem:zariskilocality}
  $R \in \CAlg$ is a local ring spectrum if and only if for every $S \in \CAlg^{\omega}$ and elements $f_{1},\dotsc,f_{n} \in \pi_{0}S$ such that $f_{1}+\dotsb+f_{n} = 1$, every map $S \to R$ admits a factorization $S \to S[f_{i}^{-1}] \dashrightarrow R $ for some $i$. Equivalently, the induced map
  \begin{equation}\label{eq:zariskilocality}
    \coprod_{i=1}^{n} \Map_{\CAlg}(S[f_{i}^{-1}], R) \to \Map_{\CAlg}(S, R)
  \end{equation}
  is a surjection on $\pi_{0}$.
\end{lemma}
 Recall that $\CAlg \simeq \Ind(\CAlg^{\omega})$ via a composite identification \begin{equation}\label{eq:calgislexfunct}\CAlg \simarrow \Ind(\CAlg^{\omega}) \coloneq \Fun^{\lex}(\GZar, \cc{S}) \end{equation} which sends a ring spectrum $R$ to the representable functor $\Map_{\CAlg}(-, R)$. Under this embedding, the condition of \eqref{eq:zariskilocality} may be phrased entirely in terms of \cref{def:classicalzariski}. Explicitly, $R$ is a local ring spectrum if and only if for every $S \in \GZar$ and admissible covering sieve $\coprod U_{i} \to S$, the following induced map is an epimorphism on passing to $\pi_{0}$
\[
  \coprod R(U_{i}) \to R(S) 
  \]
 where we have identified $R$ with its image under the embedding of \eqref{eq:calgislexfunct}.

\subsection{\texorpdfstring{$\cc{G}$}{G}-structures} Let $\cc{G}$ be a geometry. We will discuss how to obtain a ``point-free'' analog of the notion of locality for $\Ind(\cG^{\op})$-valued sheaves over an arbitrary topos. We will keep the example of the classical Zariski geometry in mind throughout, later discussing how this perspective recovers and generalizes the notion of a locally ringed space.

\begin{definition}
    Given an $\infty$-topos $\cX \in \LTop$ and an $\infty$-category $\cc{C}$ admitting small limits, the category of \tdef{sheaves on $\cX$ with coefficients in $\cc{C}$} is defined by $\Shv(\cc{X}; \cc{C}) \coloneq \Fun^{\lim}(\cc{X}^{\op}, \cc{C})$. In the case that $\cX$ appears as the $\infty$-category of sheaves of spaces on a small Grothendieck site $\cc{V}$, then \cite[Corollary 1.3.1.8]{lurieSpectralAlgebraicGeometry} supplies a canonical equivalence between $\cc{C}$-valued sheaves on $\cc{V}$ and $\cc{C}$-valued sheaves on the $\infty$-topos $\Shv(\cc{V};\cS)$. In this case we will implicitly identify these notions and simply write $\Shv(\cc{V}; \cc{C})$.
\end{definition}

\begin{observation}\label{obs:sheafislexfunct}
    Here, we note that the equivalence recorded in \eqref{eq:calgislexfunct} is a generic feature of $\infty$-topoi.  Let $\cX$ be an $\infty$-topos, and suppose that $\cc{C} \simeq \Ind(\cc{G}^{\op})$ for $\cc{G}$ a small $\infty$-category. The adjoint functor theorem \cite[Corollary 5.5.2.9, Remark 5.5.2.10]{LurieHTT} supplies an equivalence $\Shv(\cc{X}; \cc{C}) \simeq \Fun^{\mm{R}}(\cc{X}^{\op}, \cc{C})$. Consider the following identifications
 \begin{equation}\label{eq:sheafislexfunct}\begin{aligned}
    \Shv(\cc{X}; \cc{C})  \simeq & \Fun^{\mm{R}}(\cc{X}^{\op}, \cc{C}) \\
     \simeq & \Fun^{\mm{R}}(\cc{X}^{\op}, \Fun^{\lex}(\cc{G},\cS)) \\
     \simeq & \Fun^{\mm{R}, \lex}(\cc{X}^{\op} \times \cc{G}, \cS) \\
     \simeq  & \Fun^{\lex}(\cc{G}, \Fun^{\mm{R}}(\cc{X}^{\op}, \cS)) \\
     \simeq & \Fun^{\lex}(\cc{G}, \cc{X})\\
  \end{aligned}\end{equation}
where the equivalence in the last line is obtained from the Yoneda embedding $\cX \simeq \Fun^{\mm{R}}(\cX^{\op}, \cS)$ for any $\infty$-topos $\cX$. Unwinding the equivalence above, this sends a sheaf $\cc{F}\colon \cc{X}^{\op} \to \cc{C}$ to the left exact functor \[\Map_{\Ind(\cc{G}^{\op})}(-, \cc{F})\colon \cc{G} \to \cc{X} \subseteq \Fun(\cc{X}^{\op}, \cS).\] 
\end{observation}

 Given the data of a fixed geometry $\cc{G}$ and an $\infty$-topos $\cc{X}$ one may define a category of local $\cc{G}$-structures on $\cc{X}$. 

  \begin{definition}\label{def:omnibusgeometries} Let $\cat{X}$ be an $\infty$-topos and $\cat{G}$ a geometry. 
    \begin{enumerate}
    \item\label{def:omnibusgeometries1}  A \tdef{$\cat{G}$-structure on~$\cat{X}$} is a left-exact functor $\cO\colon \cat{G} \to \cat{X}$ such that for any admissible cover $\{U_i \to X\}_{i \in I}$, the induced map
    \begin{equation*}
        \coprod_{i \in I}\cO(U_i) \to \cO(X)
    \end{equation*}
    is an effective epimorphism in~$\cat{X}$. 
  \item\label{def:omnibusgeometries2}     let $\cO$, $\cO'\in \Str_{\cat{G}}(\cat{X})$ be two $\cat{G}$-structures on~$\cat{X}$. A morphism $\alpha\colon \cO \to \cO'$ is said to be \tdef{local} if for any admissible map $U \to X$, the natural diagram 
    \begin{equation*}
        \xymatrix{\cO(U) \ar[r] \ar[d] & \cO'(U) \ar[d] \\
        \cO(X) \ar[r] & \cO'(X)
        }
    \end{equation*}
    is a pullback square. We write $\Str_{\cc{G}}^{\loc}(\cc{X})$ for the wide subcategory of $\Str_{\cc{G}}(\cc{X})$ spanned by local morphisms.
    \end{enumerate}
\end{definition}

\begin{definition}\label{def:localobject}
  An object $X \in \Ind(\cG^{\op})$ is said to be \tdef{local with respect to $\cG$} if it corresponds to a local $\cG$-structure on $\cS$ via the equivalence $\Ind(\cG^{\op}) \simeq \Fun^{\lex}(\cG, \cS)$. Similarly, a map of local objects $X \to Y \in \Ind(\cG^{\op})^{[1]}$ is said to be a \tdef{local map} if it corresponds to a local map of associated $\cG$-structures on $\cS$.
\end{definition}

\begin{example}
The definitions above are best approached from the lens of classical algebraic geometry. Let $\cX$ be an $\infty$-topos and $\cO$ be a $\GZar$-structure on $\cX$. Given a point $p_{\ast}\colon\cc{S} \to \cX \in \RTop$, the pullback 
    \[
        p^{\ast} \circ \cO  \in \Fun^{\lex}(\CAlg^{\omega, \op}, \cS) \simeq \Shv(\ast; \CAlg) \simeq \CAlg
    \] 
corresponds to a local ring spectrum by \cref{lem:zariskilocality}. Given a map $\cO \to \cO' \in \Str^{\loc}_{\GZar}(\cX)$, the induced map
    \[
        p^{\ast} \circ \cO \to p^{\ast} \circ \cO' \in \CAlg^{[1]} 
    \] 
corresponds to a \emph{local map} of local ring spectra, namely it induces local ring homomorphisms on $\pi_{0}$. When $\cX$ appears as the $\infty$-topos of sheaves on a space, this shows that $\cO$ corresponds to a sheaf of rings with local stalks at every point. In this sense, the definition of a local $\GZar$ structure provides a ``point-free'' generalization of the concept of a \emph{locally ringed space}\footnote{It is important to note that the stalkwise condition is not equivalent to the condition that a left exact functor $\cO\colon \GZar \to \cX$ is a $\GZar$-structure unless one demonstrates additional properties of $\cO$ or $\cX$. Thus, having a local $\GZar$-structure on is slightly stronger than having a locally-ringed space.}.
\end{example}

Recall from \cref{obs:sheafislexfunct} the equivalence $\Fun^{\lex}(\cG, \cX) \simeq \Shv(\cX; \Ind(\cG^{\op}))$. In the case $\cG = \cG_{\mm{disc}}$, a $\cG$-structure is no more data than a left exact functor from $\cG$ to $\cX$, and thus an $\Ind(\cG^{\op})$-valued sheaf on $\cc{X}$. We record this observation below. 

\begin{example}\label{ex:discretegeometry_local}
    Let $\cG$ be an idempotent-complete $\infty$-category having finite limits. For any $\infty$-topos $\cX$ there is a canonical equivalence $\Str_{\cG_{\disc}}^{\loc}(\cX)=\Str_{\cG_{\disc}}(\cX) \simeq \Shv(\cX;\Ind(\cG^{\op}))$.
\end{example}

\begin{remark}
  For a general geometry $\cG$ and $\infty$-topos $\cX$, there is a forgetful functor from $\Str^{\loc}_{\cG}(\cX) \to \Str_{\G_{\mm{disc}}}(\cX)$ which is faithful in the sense that the induced morphism \[\Map_{\Str^{\loc}_{\cG}(\cX)}(\cc{F},\cc{G}) \to \Map_{\Str_{\cG_{\disc}}(\cX)}(\cc{F},\cc{G})\] is an inclusion of connected components for any local $\cG$-structures $\cc{F}, \cc{G}$ on $\cX$. Thus, \hyperref[def:omnibusgeometries1]{\cref*{def:omnibusgeometries}\,(a)}, resp.  \hyperref[def:omnibusgeometries2]{\cref*{def:omnibusgeometries}\,(b)}, are \emph{properties} of an underlying $\cG_{\disc}$ structure on $\cX$, resp. of a morphism of underlying $\G_{\disc}$ structures.
\end{remark}

The definition below is recorded as \cite[Definition~1.4.8]{LurieDAG5}\footnote{In \emph{loc.\ cit.}, Lurie refers to the opposite of $\LTop(\cG)$ as the $\infty$-category of $\cG$-structured $\infty$-topoi.}.

\begin{definition}\label{def:structuredtopoi}
Fix a geometry $\cG$. Let $\overline{\LTop}\to\LTop$ denotes the universal $\infty$-topos fibration, namely the co-Cartesian fibration associated to the forgetful functor $\id\colon \LTop \to \Catbig$, and consider the $\infty$-category 
    \[
        \Fun(\cG, \overline{\LTop})\times_{\Fun(\cG,\LTop)}\LTop 
    \] 
whose objects may be formally identified as pairs $\cX \in \LTop$, $\cO\colon \cG \to \overline{\LTop}\times_{\LTop} \{\cX\} \simeq \cX$. We define the subcategory 
    \[
        \mdef{\LTop(\cG)} \subseteq \Fun(\cG, \overline{\LTop}) \times_{\Fun(\cG,\LTop)}\LTop 
    \] 
of \tdef{$\cG$-structured $\infty$-topoi} as follows:
  \begin{enumerate}
  \item  An object $(\cX, \cO) \in \LTop(\cG)$ if and only if the functor $\cO \colon \cG \to \cX$ is a $\cG$-structure.
  \item A morphism $(\cX,\cO_{\cX}) \to (\cY,\cO_{\cY}) \in \LTop(\cG)^{[1]}$ if and only if for every admissible morphism $U \to X$ in $\cG$, the diagram below is Cartesian in $\cY$
    \[\xymatrix{
        f^{\ast}\cO_{\cX}(U) \ar[r] \ar[d] &  \cO_{\cY}(U) \ar[d] \\
         f^{\ast}\cO_{\cX}(X) \ar[r] & \cO_{\cY}(X)
      }
    \]
    where $f^{\ast}\colon \cX \to \cY$ is the underlying morphism in $\LTop$.
  \end{enumerate}
\end{definition}

\begin{remark}
  Given an $\infty$-topos $\cX$, the fiber $\LTop(\cG) \times_{\LTop}\{\cX\} \simeq \Str^{\mm{loc}}_{\cG}(\cX)$. 
\end{remark}

\begin{remark}
 Let $f^{\ast}: \cX \to \cY \in \LTop^{[1]}$ be any morphism. Since $f^{\ast}$ is left exact, it follows that for any $\cG$-structure $\cO \in \Str^{\loc}_{\cG}(\cX)$, the composite $f^{\ast}\circ \cO$ itself supplies a $\cG$-structure on $\cY$. By the left exactness of $f^{\ast}$ it follows that if \[\alpha\colon \cO \to \cO' \in \Str_{\cG}^{\loc}(\cX)^{[1]}\] is a local morphism of $\cG$-structures on $\cX$ then $f^{\ast}(\alpha)$ must be a local morphism of of $\cG$-structures on $\cY$. Motivated by this observation, \cite[Proposition 1.4.11]{LurieDAG5} states that the forgetful map $\LTop(\cG) \to \LTop$ is a co-Cartesian fibration, and it is classified by the functor $\Str^{\loc}_{\cG}(-) \colon \LTop \to \Catbig$.
\end{remark}

\subsection{Affine spectra}
We now review the construction of affine spectra with respect to a geometry $\cG$. For details on the material below, we refer to \cite[\S 2.1, \S 2.2]{LurieDAG5}. Let us first recall how a morphism of geometries gives rise to a restriction functor between their categories of structured $\infty$-topoi.

\begin{construction}\label{obs:restrictionfunctor}
  Fix a morphism of geometries $\alpha\colon \cG \to \cG^{\prime}$. Precomposition with $\alpha$ yields a functor
  \[ - \circ \alpha \colon \Fun(\cG', \overline{\LTop})  \to \Fun(\cG, \overline{\LTop}).
    \]
     Since the inclusion of constant diagrams $\LTop \to \Fun(\cc{C}, \LTop)$ is functorial in precomposition for small $\infty$-categories $\cc{C}$, there is an induced functor
    \[
\Fun(\cG', \overline{\LTop})\times_{\Fun(\cG',\LTop)}\LTop  \to \Fun(\cG, \overline{\LTop})\times_{\Fun(\cG,\LTop)} \LTop
\]
pointwise given by sending a pair $(\cX, \cO)$ to $(\cX, \cO \circ \alpha)$. Given $\cX \in \LTop$, if $\cO\colon \cG' \to \cX$ is a $\cG'$-structure then it is easy to see that $\cO \circ \alpha\colon \cG \to \cX$ is a $\cG$-structure on $\cX$; similarly, applying precomposition with $\alpha$ to a local morphism of $\cG'$-structures on $\cX$ yields a local transformation of $\cG$-structures on $\cX$. As a result, we may restrict the functor contructed above to $\LTop(\cG')$ to obtain the following map
    \[ \mdef{\mm{res}_{\alpha}}\colon \LTop(\cG^{\prime}) \to \LTop(\cG)\] referred to as the  \tdef{restriction along $\alpha$}, pointwise given by sending $(\cX, \cO) \in \LTop(\cG')$ to $(\cX, \cO \circ \alpha) \in \LTop(\cG)$.
\end{construction}

The statement below appears as \cite[Theorem~2.1.1]{LurieDAG5}. 

\begin{theorem}[Lurie]\label{thm:relativespec}
    For a morphism $\cG'\to \cG$ of geometries, the induced functor $\mm{res}\colon \LTop(\cG) \to \LTop(\cG')$ admits a left adjoint, denoted by $\mdef{\Spec_{\cG'}^{\cG}}$.
\end{theorem}

The left adjoint of the theorem above is referred to as the \tdef{relative spectrum} functor. For illustrative purposes and for applications later in the next section, we will also record the behavior of the functor above under the embedding of the category of $\cG$-structures into $\Ind(\cG^{\op})$-valued sheaves on a given $\infty$-topos.

\begin{lemma}\label{lem:resintermsofsheaves}
  Let $\alpha: \cG_{1} \to \cG_{2}$ be a morphism of geometries, and let $\alpha_{\ast}: \Ind(\cG^{\op}_{2}) \to \Ind(\cG^{\op}_{1})$ indicate the ind-existent right adjoint to the opposite of $\alpha$. There is a commutative square of the form
  \[\xymatrix{
      \Fun^{\lex}(\cG_{2}, \cX) \ar[d]^{\eqref{eq:sheafislexfunct}} \ar[r]^{- \circ \alpha} & \Fun^{\lex}(\cG_{1}, \cX) \ar[d]^{\eqref{eq:sheafislexfunct}} \\
      \Shv(\cX; \Ind(\cG^{\op}_{2})) \ar[r]^{\alpha_{\ast}} &  \Shv(\cX; \Ind(\cG^{ \op}_{1}))
    }
  \]
  where the bottom horizontal arrow sends a sheaf $\cc{F}: \cX^{\op} \to \Ind(\cG^{\prime \op})$ to the composite $\alpha_{\ast}\circ \cc{F}$. 
\end{lemma}

\begin{proof}
  Let $\alpha^{\ast} \dashv \alpha_{\ast}$. As in \cref{obs:sheafislexfunct}, consider the embedding \[\Fun^{\lex}(\cG_{i}, \cX) \simeq \Fun^{\lex}(\cG_{i}, \Fun^{\mm{R}}(\cX^{\op}, \cS)) \hookrightarrow \Fun(\cG_{i}\times \cX^{\op}, \cS)\]
  under which the top horizontal morphism of the claim is identified with the restriction of the functor given by precomposition with $\alpha \times \id$. Symmetrically, there is an embedding \[\Fun(\cX, \Fun^{\lex}(\cG_{i},\cS)) \hookrightarrow \Fun(\cG_{i} \times \cX^{\op}, \cS)\] for which the restriction of $\alpha \times \id$ is identified with the functor given by postcomposition with \[- \circ \alpha\colon \Fun^{\lex}(\cG_{2}, \cS)\to  \Fun^{\lex}(\cG_{1}, \cS)\] which is canonically identified with the right adjoint $\alpha_{\ast}$ of the theorem, yielding the claim.
\end{proof}

For the remainder of this subsection we will fix a geometry $\cG = (\cG,\cG^{\ad},\tau)$. Recall that $\overline{\LTop}\to\LTop$ denotes the universal $\infty$-topos fibration and define the \tdef{global sections functor} as the map 
    \[
        \Gamma \colon \overline{\LTop} \to \cS
    \]
corepresented by $(\cS,\ast)$. Note that $\cS$ is initial in $\LTop$, where the unique map from $\cS$ to any $\infty$-topos $\cX$ sends $\ast$ to the terminal object of $\cX$. It follows that the restriction of $\Gamma$ to $\overline{\LTop} \times_{\LTop} \cX \simeq \cX$ is corepresented by the terminal object $\bf{1} \in \cX$. Now consider the composite
    \[
        \LTop(\cG) \times \cG \rightarrow \Fun(\cG, \overline{\LTop}) \times \cG \xrightarrow{\ev} \overline{\LTop} \xrightarrow{\Gamma} \cS.
    \] Its adjoint factors as
    \[
        \xymatrix{\LTop(\cG) \ar[r] \ar[d]_{\Gamma_{\cG}} & \Fun(\cG,\cS) \\
        \Fun^{\lex}(\cG,\cS) \ar[r]^-{\simeq} & \Ind(\cG^{\op}) \ar[u]}
    \]
    and we refer to $\mdef{\Gamma_{\cG}}$ as the \tdef{$\cG$-structured global sections functor}. 

    \begin{lemma}\label{lem:globalsectionsisglobalsections}
      Given an $\infty$-topos $\cX$, the composite \[\Str^{\loc}_{\cG}(\cX) \simeq \LTop(\cG) \times_{\LTop}\{\cX\} \hookrightarrow \LTop(\cG) \xrightarrow{\Gamma_{\cG}} \Ind(\cG^{\op})\]
      may be identified with the functor \[\Str^{\loc}_{\cG}(\cX) \subseteq \Fun^{\lex}(\cG, \cX) \simarrow \Shv(\cX; \Ind(\cG^{\op})) \xrightarrow{\ev_{\bf{1}}} \Ind(\cG^{\op})  \]
      where $\ev_{\bf{1}}$ sends a sheaf $\cc{F}$ to its value on the terminal object $\bf{1} \in \cX$.
    \end{lemma}

    \begin{proof}
      It suffices to show that the map $h: \Fun^{\lex}(\cG, \cX) \to \Fun^{\lex}(\cG, \cS)$ given by postcomposing with the map $\Gamma\colon \cX \to \cS$ may be identified through the equivalence of \cref{obs:sheafislexfunct} with the map $\ev_{\bf{1}}$. By definition, $\Gamma$ is corepresented by the terminal object of $\cX$. Let $\Gamma^{\ast}\colon\cS \to \cX$ indicate the left adjoint to this map, i.e., the unique map sending $\ast \mapsto \bf{1}$. We obtain a commutative square
      \[\xymatrix{
          \cX \ar[r]^{\Gamma} \ar[d]^{\simeq} & \cS \ar[d]^{\simeq} \\
          \Fun^{\mm{R}}(\cX^{\op}, \cS) \ar[r]^{- \circ \Gamma^{\ast}} & \Fun^{\mm{R}}(\cS^{\op}, \cS)}
        \] where the vertical equivalences arise from the Yoneda embedding. Once again invoking the embedding $\Fun^{\lex}(\cG, -) \hookrightarrow \Fun(\cG \times -^{\op}, \cS)$, the map $h$ is identified with the restriction of  \[- \circ (\id \times \hspace{1pt}\Gamma^{\ast})\colon \Fun(\cG \times \cX^{\op}, \cS) \to \Fun(\cG \times \cS^{\op}, \cS)\] whose restriction along $\Shv(-;\Ind(\cG^{\op})) \hookrightarrow \Fun(\cG \times -^{\op}, \cS)$ is readily identified as the functor \[\Shv(\cX; \Ind(\cG^{\op})) \to \Shv(\ast; \Ind(\cG^{\op})) \simeq \Ind(\cG^{\op})\] given by sending a sheaf $\cc{F}$ to its value on $\Gamma^{\ast}(\ast) \simeq \bf{1} \in \cX$. 
    \end{proof}

The functor $\Gamma_{\cG}$ constructed above is right adjoint to the inclusion 
    \[
        \Ind(\cG^{\op}) \simeq \Str_{\cG_{\disc}}(\cS)
        \simeq \LTop(\cG_{\disc}) \times_{\LTop} \{\cS\} \subseteq \LTop(\cG_{\disc}).
      \] To see this, note that for every $(\cX, \cO) \in \LTop(\cG_{\disc})$, the space of maps $(\cS, \cO') \to (\cX, \cO) \in \LTop(\cG_{\disc})^{[1]}$ may be identified with the space of maps $\cO' \circ \Gamma^{\ast} \to \cO \in \Fun^{\mm{lex}}(\cG, \cX)^{[1]}$; however, this may equivalently be identified with the space of maps $\cO' \to \cO \circ \Gamma_{\ast} \eqcolon \Gamma_{\cG}(\cX, \cO) \in \Fun^{\mm{lex}}(\cG, \cS) \simeq \Ind(\cG^{\op})$ using the adjunction $\Gamma^{\ast} \dashv \Gamma_{\ast}$. This fact along with \cite[Proposition 5.2.7.8]{LurieHTT} imply that the inclusion $\Ind(\cG^{\op}) \to \LTop(\cG_{\disc})$ is a fully faithful left adjoint, with right adjoint $\Gamma_{\cG}$. By composing the described adjunction with the adjunction of \cref{thm:relativespec}, we obtain the lemma below.

    \begin{lemma}\label{lem:globalsections}
      There is an adjunction of the form
      \[
        \Spec^{\cG}_{\cG_{\disc}}\colon \Ind(\cG^{\op}) \rightleftarrows \LTop(\cG) \noloc\Gamma_{\cG}.
        \]
\end{lemma}

The left adjoint above may be computed via an explicit presentation which we will now recall.  A morphism $f\colon U \to X$ in $\Pro(\cG)$ is called pro-admissible if there exists a pushout square in $\Pro(\cG)$ of the form
    \[
        \xymatrix{U \ar[r] \ar[d]_{f} & \Yo(U') \ar[d]^{\Yo(f')} \\
        X \ar[r] & \Yo(X')}
    \]
 where $f'\colon U' \to X' \in \cG^{\ad,[1]}$. Any admissible morphism in $\Pro(\cG)$ is pro-admissible, from which it follows that any $f\in \cG$ with $\Yo(f) \in \Pro(\cG)$ admissible is already in $\cG^{\ad}$. The wide subcategory of admissible morphisms in $\Pro(\cG)$ will be denoted by $\Pro(\cG)^{\ad}$. $\Pro(\cG)^{\ad}$ contains the equivalences, is stable under pullbacks, and if $f,g,h$ are morphisms in $\Pro(\cG)$ with $h$ admissible, then $f$ is admissible if and only if $g$ is. However, $\Pro(\cG)^{\ad}$ might not be closed under retracts. 

For $X \in \Pro(\cG)$, let $\Pro(\cG)_{/X}^{\ad} \subseteq \Pro(\cG)_{/X}$ be the (essentially small) full subcategory spanned by the admissible morphisms. Furthermore, we will view $\Pro(\cG)_{/X}$ as being endowed with the coarsest Grothendieck topology such that, for any $U \to X$, every admissible cover $\{V_i' \to U'\}_{I}$ in $\cG$, and any morphism $U \to \Yo(U')$, the collection $\{\Yo(V_i)\times_{j(U')} U \to U\}_{I}$ generates a covering sieve of $U \to X$. We are ready for the key construction.

\begin{definition}\label{def:spec}
Let $\cG$ be a geometry and $X$ an object of $\Pro(\cG)$, then we define the \tdef{absolute $\cG$-spectrum} of $X$ as the following object of $\LTop(\cG_{\disc})$
    \[
        \mdef{\Spec^{\cG}(X)} \coloneq (\Shv(\Pro(\cG)_{/X}^{\ad}), \cO_{X})
    \]
where $\cO_{X}$ is defined to be the $\cG_{\disc}$ structure
    \[
        \cO_{X}\colon \cG \xrightarrow{\widetilde{\cO}_{X}}
        \PShv(\Pro(\cG)_{/X}^{\ad}) \xrightarrow{L} \Shv(\Pro(\cG)_{/X}^{\ad}) = \Spec^{\cG}(X).
    \]
In the equation above, $\widetilde{\cO}_{X}$ is adjoint to the composite
    \[
        \cG \times (\Pro(\cG)_{/X}^{\ad})^{\op} \to \cG \times \Pro(\cG)^{\op} = \cG \times \Fun^{\lex}(\cG,\cS) \xrightarrow{\ev} \cS
      \]
     and $L$ denotes sheafification with respect to the topology described in the previous paragraph.
  \end{definition}

\begin{proposition}\label{prop:spec_gstructure}
    $\cO_{X}$ equips $\Spec^{\cG}(X)$ with a $\cG$-structure. 
\end{proposition}

  \begin{lemma}\label{lem:identificationofGstructure}
    Let $\cX$ be the underlying topos of $\Spec^{\cG}(X)$ for some $X \in \Pro(\cG)$. Under the equivalence of \cref{obs:sheafislexfunct}, the object $\cO_{X} \in \Shv(\cX;\Ind(\cG^{\op}))$ corresponds to the sheafification of the forgetful functor $(\Pro(\cG)^{\ad}_{/X})^{\op} \to \Pro(\cG)^{\op} = \Ind(\cG^{\op})$ with respect to the pro-admissible topology.
  \end{lemma}
  \begin{proof}
    Using the same argument of \cref{lem:globalsectionsisglobalsections}, there is commutative square of the form
    \[\xymatrix{
           \Fun^{\lex}(\cG, \cX) \ar[r] \ar[d]^{\simeq} & \Fun^{\lex}(\cG,\cc{P}(\Pro(\cG)^{\ad}_{/X}))  \ar[d]^{\simeq} \\
           \Fun^{\mm{R}}(\cX^{\op}, \Ind(\cG^{\op}))  \ar[r]^{\! \! -\circ L \ \ } & \Fun(\Pro(\cG)^{\ad}_{/X}, \Ind(\cG^{\op}))}
      \]
      where the top arrow arises from postcomposition with $\Shv(\Pro(\cG)^{\ad}_{/X})) \to \cc{P}(\Pro(\cG)^{\ad}_{/X})$ the right adjoint inclusion. Passing to left adjoints, we see that the functor \[L \circ -\colon \Fun^{\lex}(\cG, \cc{P}(\Pro(\cG)^{\ad}_{/X})) \to \Fun^{\lex}(\cG, \cX)\] is identified with the left adjoint to $\Shv(\cX;\Ind(\cG^{\op})) \to \Fun((\Pro(\cG)^{\ad}_{/X})^{\op}, \Ind(\cG^{\op}))$ the functor which forgets a pro-admissible sheaf to its underlying presheaf; i.e., $L$ is the sheafification with respect to the pro-admissible topology. It remains to identify the sheaf corresponding to $\cO_{X}$. As before, we make use of the equivalences
\[  \Fun(\cG , \cc{P}(\Pro(\cG)^{\ad}_{/X})) \simarrow \Fun(\cG \times (\Pro(\cG)^{\ad}_{/X})^{\op}, \cc{S}) \simarrow \Fun((\Pro(\cG)^{\ad}_{/X})^{\op}, \Ind(\cG^{\op})) 
\]
under which $\widetilde{\cO}_{X}$ is mapped to the adjoint of the composite
\[
  \cG \times (\Pro(\cG^{\ad})_{/X})^{\op} \to \cG \times \Pro(\cG)^{\op} = \cG \times \Ind(\cG^{\op}) \xrightarrow{ev} \cS
\]
which is easily seen to be the forgetful functor $(\Pro(\cG^{\ad})_{/X})^{\op} \to \Pro(\cG)^{\op} = \Ind(\cG^{\op})$.
  \end{proof}

  From \cref{lem:globalsectionsisglobalsections} and \cref{lem:identificationofGstructure}, we obtain a natural identification \[\Gamma_{\cG_{\disc}}(\cc{P}(\Pro(\cG)^{\ad}_{/X}), \widetilde{\cO}_{X}) \simeq X\] by evaluating the unique limit-preserving extension of the forgetful functor $\cc{P}(\Pro(\cG)^{\ad}_{/X})^{\op} \to \Ind(\cG^{\op})$ on the terminal object, which is the Yoneda image of $X$. We thus obtain a map
    \[
        \alpha\colon X \to \Gamma_{\cG}(\Spec^{\cG}(X), \cO_{\Spec^{\cG}(X)}) \in \Ind(\cG^{\op})^{[1]}
    \]
by evaluating the sheafification map $\widetilde{\cO}_{X} \to \cO_{X}$ on the terminal object. The theorem below is \cite[Theorem 2.2.12]{LurieDAG5}.

\begin{theorem}\label{thm:spec}
Let $\cG$ be a geometry and let $\cG_{\disc} \to \cG$ be the canonical functor from the discrete geometry $\cG_{\disc}$ on $\cG$ to $\cG$. Then for every $X \in \Pro(\cG)$ the map $\alpha$ constructed above is adjoint to an equivalence \[\Spec_{\cG_{\disc}}^{\cG}X \simeq \Spec^{\cG}X\] via the adjunction of \cref{lem:globalsections}.
\end{theorem}

\begin{corollary}\label{cor:geometry_adjunction}
    For any $(\cY,\cO_{\cY}) \in \LTop(\cG)$ and $X \in \Pro(\cG)^{\op}$, there is an equivalence
    \begin{equation*}
        \Map_{\LTop(\cG)}((\Spec^{\cG}X,\cO_{\Spec^{\cG}X}),(\cY,\cO_{\cY})) \simeq \Map_{\Pro(\cG)^{\op}}(X,\Gamma_{\cG}(\cY,\cO_{\cY})).
    \end{equation*}
    induced by composition with $\alpha$.
  \end{corollary}
  
  \subsection{Example: classical Zariski geometry}
  
Let us return once again to the classical Zariski geometry. The absolute spectrum functor supplies the following adjunction
  \begin{equation}\label{eq:classicalaffinezariski}
     \Spec \colon \CAlg \rightleftarrows \LTop(\GZar) : \Gamma_{\GZar} 
  \end{equation}
  where we have written $\Spec \coloneq \Spec^{\GZar}$. All that remains is to identify the left adjoint, which is identified with the familiar presentation below;  we refer the reader to the proof of \cite[Theorem 2.40]{lurieDerivedAlgebraicGeometry2011a} for a complete account.

  \begin{proposition}
        Given $R \in \CAlg$, $\Spec R \in \LTop(\GZar)$ may be identified with the pair $(\Shv(\Spec \pi_{0}R), \cO)$ where:
\begin{enumerate}
\item $\Spec \pi_{0}R$ refers to the ordinary Zariski spectrum of prime ideals.
\item $\cO$ is the unique sheaf of commutative ring spectra on $\Spec \pi_{0}R$ satisfying 
\[ \cO: D(f) \mapsto R[f^{-1}]\]
  for $D(f)$ a basic open set of $\pi_{0}R$ associated to some element $f$.
  \end{enumerate}
\end{proposition}

In light of the computation above, one finds that the adjunction of \eqref{eq:classicalaffinezariski} recovers the traditional characterization of affine spectral schemes. In the next subsection, we explain how the Balmer spectrum of a 2-ring will similarly recover the underlying space of its absolute Zariski spectrum, for a suitably defined Zariski geometry on $\twoCAlg$.

\subsection{Example: spectral Dirac geometry}\label{ssec:dirac}

The definition below is a spectral variant of the underlying geometry in \cite{hesselholtDiracGeometryCommutative2023}.  

\begin{definition}\label{def:dirac}
    The \tdef{spectral Dirac geometry} consists of the following data:
        \begin{enumerate}
            \item $\GDir=(\CAlg^{\omega})^{\op}$, the opposite of the category of compact commutative ring spectra.
            \item Admissible morphisms correspond to localization maps $R \to R[x^{-1}]$ for $x \in \pi_{\ast}R$.
            \item   A finite collection $\{R\to R[x_i^{-1}]\}_{i \in I}$ is declared to generate a covering sieve if the set $\{x_i\}_{i \in I} \subset \pi_{\ast}R$ generates the unit ideal. 
        \end{enumerate}
      \end{definition}

            \begin{definition}
        A commutative ring spectrum $R$ is \tdef{Dirac-local} if $\pi_{2\ast}R$ is a local ring. A map of $R \to S$ of Dirac-local commutative ring spectra is a \tdef{Dirac-local map} if the induced map on $\pi_{2\ast}$ is a local map of rings.
      \end{definition}

      \begin{theorem}\label{thm:diracspectrum}
        Given a commutative ring spectrum $R$, the underlying $\infty$-topos of $\Spec^{\GDir}R$ is naturally identified with the $\infty$-topos of sheaves on the spectral space $\Spec^{h}\pi_{2\ast}R$ of homogeneous prime ideals of the graded ring $\pi_{2\ast}R$. The structure sheaf corresponding to $\cO_{R}$ is given by the unique sheaf of commutative ring spectra satisfying
        \[          D(f) \coloneq \{\mathfrak{p} \in \Spec^{h}\pi_{2\ast}R\ |\ f^{2} \notin \mathfrak{p}\} \mapsto R[f^{-1}]
          \]
          for $f\in\pi_{\ast}R$.
      \end{theorem}

            The proof of \cref{thm:diracspectrum} requires only cosmetic modifications to the argument of \cite[Theorem 2.40]{lurieDerivedAlgebraicGeometry2011a} which we leave to the reader, see also \cite[Remark 2.25, Proposition 2.35, Theorem 2.26]{hesselholtDiracGeometryCommutative2023}. The following result is a spectral incarnation of \cite[Theorem 2.26.(2)]{hesselholtDiracGeometryCommutative2023}, again utilizing the same proof. 

      \begin{lemma}\label{lem:diracstalks}
        Given a commutative ring spectrum $R$ and point $x^{\ast}\colon \Spec^{\GDir}R \to \cS$ corresponding to a homogeneous prime ideal $\mathfrak{p} \subset \pi_{2\ast}R$, the pullback $x^{\ast}\cO_{R} \in \Str_{\GDir}(\cS)$ corresponds to the Dirac-local ring spectrum $R_{\mathfrak{p}}$.
      \end{lemma}

      We also collect the lemma below for later applications.

      \begin{lemma}\label{lem:diraclocality}
        The category $\Str^{\loc}_{\GDir}(\cS)$ is naturally equivalent to the subcategory $\cc{C} \subset \CAlg$ of Dirac-local commutative ring spectra with Dirac-local maps between them.
      \end{lemma}
      \begin{proof}
        The verification that any object $\cO \in \Str^{\loc}_{\GDir}(\cS)$ corresponds to a Dirac-local commutative ring spectrum utilizes the exact same argument as \cref{lem:zariskilocality}. Given a map of Dirac-local ring spectra $R \to S$, an arbitrary map $R' \to S$ and an element $f \in \pi_{\ast}R'$ where $R'$ is compact, the following is a pullback square
        \[\xymatrix{
            \Map_{\CAlg}(R'[f^{-1}], R) \ar[r] \ar[d] & \Map_{\CAlg}(R'[f^{-1}], S) \ar[d] \\
            \Map_{\CAlg}(R',R) \ar[r] & \Map_{\CAlg}(R',S)
          }
        \]
        if and only if any map $R' \to R$ inverts $f$ in $\pi_{\ast}S$ only when it inverts $f$ in $\pi_{\ast}R$. Allowing $R'$ to range over free commutative ring spectra, we find that this is the case if and only if $R \to S$ is local on $\pi_{2\ast}$. 
      \end{proof}
\vspace{5pt}

\section{Higher Zariski Geometry}\label{sec:zariski}

In this section, we define the Zariski geometry on 2-rings and demonstrate that the resultant spectrum recovers the Balmer spectrum of tensor-triangular geometry. We will supply a quick recollection of the requisite material on Balmer spectra in \cref{ssec:tt_spc}. 

We will then go on to study the local structure of the structure sheaf on the Zariski spectrum of a 2-ring, recovering a variant of the notion of locality previously described in \cite{balmerSpectraSpectraSpectra2010} in the tensor-triangulated setting. We finally conclude by comparing the Zariski geometry on 2-rings with the classical Zariski and Dirac geometries on ring spectra via the functor of perfect complexes. Using this, we recover a classical computational tool described in \cite{balmerSpectraSpectraSpectra2010}, and use this to compute the Zariski spectra of 2-rings which arise as the perfect complexes of ordinary commutative rings. 

\subsection{The Zariski geometry on \texorpdfstring{$\twoCAlg$}{2CAlg}}\label{ssec:zariski_2zariski}

The following is the central new definition of this paper.

\begin{definition}\label{def:zariski}
    The following data defines the \tdef{Zariski geometry} on commutative 2-rings.
        \begin{enumerate}
            \item $\mdef{\GBal}=(\twoCAlg^{\omega})^{\op}$ is the opposite of the $\infty$-category of compact $2$-rings.
            \item The class of admissible morphisms $\mdef{\G_{\Bal}^{\ad}}\subset \GBal$ correspond to the Karoubi quotients $\cc{K} \to \cc{K}' \in (\twoCAlg^{\omega})^{[1]}$. 
            \item A finite collection of admissible morphisms $\{f_i\colon\cc{K}\to\cc{K}_{i}\}_{i \in I}$ is declared to generate a covering sieve if $\bigcap_{i\in I}\ker f_{i} \subseteq \sqrt{0}$.
            \end{enumerate}
          \end{definition}      

 We now prove \cref{thmalph:zariski2rings}.

\begin{theorem}\label{thm:2zariski_geometry} 
    The Zariski data $\GBal$ satisfies the conditions of \cref{def:geometry}, i.e., it is a geometry.
\end{theorem}
\begin{proof}
Since $\twoCAlg$ is compactly generated, $\twoCAlg^{\omega}$ admits all finite colimits. It will thus suffice to show that admissible morphisms are closed under base change, left cancellation, and retracts formed in $\twoCAlg^{[1]}$; these properties are demonstrated in \cref{cor:basechange} and \cref{cor:retractsandleftcancel}, respectively. 
\end{proof}

For use in the next section, we record an alternative characterization of the admissibility structure for the Zariski geometry.

\begin{lemma}\label{lem:generatingclass}
  The class of admissible morphisms specified in part (b) of \cref{def:zariski} may be equivalently characterized as:
  \begin{enumerate}
  \item The class of principal Verdier localizations $\cc{K} \to \cc{K}/\langle a \rangle$ for $a \in \cc{K}$.
  \item The coarsest admissible class in $(\twoCAlg^{\omega})^{\op}$ containing the universal principal quotient $\Sp^{\omega}\{x\} \to \Sp^{\omega}$.
  \end{enumerate}
  \end{lemma}
  \begin{proof}
   Given any admissible morphism $\cc{K}\to \cc{K}' \in \twoCAlg^{\omega}$, the compactness of $\cc{K}, \cc{K}'$ along with the latter part of \cref{prop:quotientposet} implies that $\cc{K}'$ must be associated to a compact and hence principal Karoubi quotient of $\cc{K}$; this yields claim (a). Claim (b) is now an immediate application of part (a) and \cref{cor:universalquotient}.
  \end{proof}

\begin{remark} The previous lemma shows that the nomenclature in \cref{def:zariski} is Hochster dual to the definition of the \emph{Zariski frame} provided in \cite[\S 2]{KockPitsch17}, and this may cause some confusion. We view the results of \cref{ssec:zariski_comparison} as justification for this choice.
\end{remark}

\subsection{Comparison with the Balmer spectrum}\label{ssec:tt_spc}

In this subsection we will compare the $\infty$-topos obtained via the constructions of the previous subsection with the Balmer spectrum. We first review the lattice-theoretic definition of the Balmer spectrum given in \cite[\S 3]{Aoki2023}, which essentially appeared in \cite{KockPitsch17}. The following is a natural continuation of \cref{rec:latticedefinitions}.

\begin{definition}  Recall first that a lattice, i.e., a poset admitting finite limits and colimits\footnote{This is usually called a \emph{bounded} lattice, since we require the empty limit $\top$ and colimit $\bot$ to exist.}, is said to be \tdef{distributive} if finite limits distribute over finite colimits; in other words, finite meets distribute over finite joins. This is equivalent to requiring that finite colimits distribute over finite limits, see \cite[Lemma 1.5]{Johnstone1986-th}. 
\begin{enumerate}
   \item  A morphism of distributive lattices $f\colon L \to L'$ is a functor preserving finite limits and colimits, i.e., an order-preserving function preserving finite meets and joins.
  \item A morphism of frames $f\colon F \to F'$ is a left adjoint functor which preserves finite limits, i.e., an order-preserving function preserving finite meets and arbitrary joins. 
  \item A morphism of coherent frames $f \colon F \to F'$ is a morphism of frames which moreover preserves compact objects. 
  \end{enumerate}
We write $\mdef{\mm{DLat}},\ \mdef{\mm{Frm}},\ \mdef{\mm{Frm^{coh}}}$ to refer to the categories of distributive lattices, frames, and coherent frames (if necessary, see \cref{rec:latticedefinitions}), respectively.
\end{definition}

\begin{definition} Let $L \in \mm{DLat}$. An \tdef{ideal} of $L$ is a downward-closed subset $I \subseteq L$ which is closed under finite colimits (and hence nonempty as it necessarily contains the minimal element $\bot$). We write $\mdef{\Idl(L)}$ to refer to the set of ideals of $L$ partially ordered by inclusion.
\end{definition}

The assignment $L \mapsto \mm{Idl}(L)$ is covariantly functorial in $\mm{DLat}$; given a morphism $f\colon L \to L' \in \mm{DLat}^{[1]}$, one may send an ideal $I \subseteq L$ to the smallest ideal containing $f(I)$ in $L'$. The poset of ideals of a distributive lattice turns out to be a coherent frame, which we note in the remark below.

\begin{remark}\label{rem:idl-is-accessible}
  The coherent frame $\mm{Idl}(L)$ is naturally identified with $\Ind(L) = \Fun^{\lex}(L^{\op}, \cS)$ via the map which assigns to an ideal $I \subseteq L$ the finite product-preserving functor $L^{\op} \to \cS$ via $I \mapsto \ast$ and $L \setminus I \mapsto \emptyset$. In this way, $\mm{Idl}(L)$ may be regarded as the completion of $L$ under filtered colimits.
\end{remark}

\begin{recollection}\label{rec:stoneduality}
  The details surrounding the basic theory recalled here may be found in \cite[\S II.3]{Johnstone1986-th}. Let us review the Stone duality for distributive lattices, which manifests itself in the following diagram of adjoints
  \[
    \xymatrix{
      \mm{DLat} \ar@<.5ex>[r]^-{\mm{Idl}} & \mm{Frm^{coh}} \ar@<.5ex>[l]^-{(-)^{\omega}}  \subseteq \mm{Frm} \ar@<.5ex>[r]^-{|-|} & \mm{Top}^{\op} \ar@<.5ex>[l]^-{\cc{U}} \\
    }
  \]
  where:
  \begin{enumerate}
  \item $\mm{Top}$ is the category of topological spaces. 
  \item $\cc{U}$ sends a topological space $X$ to its poset of open subsets. This is a frame as it admits arbitrary colimits (given by taking unions) and finite limits (given by intersections) and the former clearly distribute over the latter.
  \item $|-|$ sends a frame $F$ to the space whose underlying set of points is given by $\Fun^{\mm{L}}(F, [1])^{\simeq}$ the frame maps from $F$ to $[1]$, and whose topology is given by open sets $U_{a} \coloneq \{f \in \Fun^{L}(F, [1]) \ | \ f(a) = 1\}$ for every $a \in F$.  
  \end{enumerate}
  The adjunctions above admit the following features, which are collectively referred to as the Stone duality for frames and lattices. 
  \begin{enumerate}
  \item[(i)] The adjunction $|-| \dashv \cc{U}$ restricts to an equivalence between the subcategory $\mdef{\mm{Frm^{spa}}} \subseteq \mm{Frm}$ of \tdef{spatial frames}, namely those frames which are in the essential image of $\cc{U}$, and the opposite of the full subcategory $\mdef{\mm{Sob}}\subseteq \mm{Top}$ of \tdef{sober spaces}, namely those spaces whose irreducible closed subsets admit unique generic points.
  \item[(ii)] $\mm{Frm^{coh}} \subseteq \mm{Frm^{spa}}$, and the aforementioned adjunction supplies an equivalence between $\mm{Frm^{coh}}$ and the opposite of the subcategory $\mdef{\mm{Top^{spec}}}\subseteq \mm{Sob}$ of \tdef{spectral spaces}, namely the subcategory with objects given by those sober spaces which are quasicompact with a basis of quasicompact open subsets closed under finite intersections, and morphisms given by quasicompact maps.
    \item[(iii)] The adjoints $\mm{Idl} \dashv (-)^{\omega}$ are mutually inverse, and thus $\mm{DLat} \simeq \mm{Frm^{coh}}$. In particular, $\mm{DLat}^{\op}$ and $\mm{Top^{spec}}$ are equivalent via the composite $(-)^{\omega} \circ \cc{U}$ which sends a spectral space to its distributive lattice of quasicompact open subsets.
    \end{enumerate}
    
    Note that $\mm{DLat}$ admits an involutive duality given by $L \mapsto L^{\op}$, i.e., by flipping the partial order and interchanging joins and meets. This induces an involutive duality $(-)^{\vee}$ on $\mm{Frm^{coh}}$ and $\mm{Top^{spec}}$ known as \tdef{Hochster duality}. On the level of frames, Hochster duality simply sends a coherent frame $F$ to the ideal frame of $(F^{\omega})^{\op}$. On the level of topological spaces, Hochster duality sends a spectral space $X$ to the space $X^{\vee}$ whose underlying set is the same as $X$, and whose topology is generated by an open basis comprising the complements of quasicompact open subsets of $X$. We refer to \cite[\S 1.2]{KockPitsch17} for details.
\end{recollection}

We may now define the Balmer spectrum of a 2-ring, following \cite[\S 3.4]{KockPitsch17}. 

\begin{definition}\label{def:balmer}
  Let $\cc{K} \in \twoCAlg$. By \cref{prop:latticeiscoherent} the poset of radical ideals $\Rad(\cc{K})$ is a coherent frame, hence it is spatial. The \tdef{Balmer spectrum} of $\cc{K}$, denoted by $\mdef{\Spc \cc{K}}$, is the spectral space associated to the Hochster dual frame $\Rad(\cc{K})^{\vee}$. Unwinding the definitions, this is the topological space whose points are labelled by prime tt-ideals of $\cc{K}$, and whose topology is generated by the open subsets $U(a) \coloneq \{\cc{P} \subseteq \cc{K} \text{ prime }| \ a \in \cc{P}\}$ for $a \in \cc{K}$. 
 \end{definition}
 
 Before approaching the proof of \cref{thmalph:comparison}, we will require a brief recollection on the connection between $\infty$-topoi, frames, and topological spaces.

 \begin{recollection}\label{rec:0topoi}
   Given a frame $F$, consider the canonical Grothendieck topology on $F$; namely, a collection $\{c_{\alpha} \to c\}_{\alpha \in A}$ generates a covering sieve if $\bigvee_{A} c_{\alpha} = c$ in $F$. We write $\Shv(F)$ without any decorations to denote the category of sheaves on a given frame $F$ with respect to this Grothendieck topology. From this definition, the category of sheaves on a topological space $X$ is manifestly identified with the category $\Shv(\cc{U}(X))$ of sheaves on its associated frame of open subsets.

   Given an $\infty$-topos $\cc{X}$, we write $\mdef{\mm{Sub}(\cc{X})}$ to refer to its poset of \tdef{subterminal objects}, namely those objects which correspond to elements of $\Fun^{\mm{R}}(\cc{X}, \tau_{\leq -1}\cS)$ under the equivalence $\cc{X} \simeq \Shv(\cc{X})$. \cite[Theorem 6.4.2.1, \S 6.4.5]{LurieHTT} supply the following adjunction
   \[
     \Shv(-)\colon \mm{Frm} \rightleftarrows \LTop \noloc\mm{Sub}(-)
   \]
   whose left adjoint is fully faithful. The essential image of $\Shv(-)$ is spanned by the \emph{0-localic $\infty$-topoi}, see Section 6.4.5 of \emph{loc.\ cit.}\ for details.
 \end{recollection}

We are ready to prove \cref{thmalph:comparison}.
 
\begin{theorem}\label{thm:2zariski_balmerspectrum}
    For any $\cc{K} \in \twoCAlg$ there is an equivalence $\Spec^{\GBal} \cc{K} \simeq \Shv(\Spc\cc{K})$ of underlying $\infty$-topoi. 
\end{theorem}

\begin{lemma}\label{lem:admissibleprincipal}
  For any $\cc{K} \in \twoCAlg$, the composite $\cc{K}/- : \Idl(\cc{K})^{\op} \to (\twoCAlg_{\cc{K}/})^{\op} \simeq \Pro(\GBal)_{/\cc{K}}$ induces an equivalence $\mm{Prin}(\cc{K})^{\op} \simeq \Pro(\GBal)^{\ad}_{/ \cc{K}}$.
\end{lemma}
\begin{proof}
  \cref{lem:generatingclass} implies that $\Pro(\GBal)^{\ad}_{/ \cc{K}}$ consists exactly of those morphisms $\cc{K} \to \cc{K}'$ which are base-changed from the universal quotient $\Sp^{\omega}\{x\} \to \Sp^{\omega}$; in particular, it consists exactly of principal Verdier quotients of $\cc{K}$. The resulting identification follows from \cref{prop:quotientposet}.
\end{proof}

\begin{proof}[Proof of \cref{thm:2zariski_balmerspectrum}]
  Let $P:=\mm{Prin}(\cc{K})$. From the equivalence of \cref{lem:admissibleprincipal}, the pro-admissible topology is identified as the following Grothendieck topology on $P$: a set of inclusions $\{\langle x \rangle \subset \langle y_i \rangle \mid i \in I\}$ generates a covering sieve if and only if there exists a finite subset $I_0 \subset I $ such that $\bigotimes_{i \in I_0} y_i \in \sqrt{x}$. By \cref{lem:radicalgeneration}, \cref{prop:latticeiscoherent}, and the fact that $\sqrt{-}\colon \Idl(\cc{K}) \to \Rad(\cc{K})$ is a left adjoint, taking radicals supplies an order-preserving surjection $f\colon P \to \Rad(\cc{K})^{\omega}$ which preserves finite meets and joins. Since $f^{\op}$ also preserves covering sieves, applying $\Shv(-)$ supplies a left adjoint morphism  $f^{\ast}\colon \Shv(P^{\op}) \to \Shv(\Rad(\cc{K})^{\vee}) \in \LTop^{[1]}$ which we will show is an equivalence.

  \cref{rec:0topoi} reduces us to checking that this induces an isomorphism between their frames of subterminal objects; that is to say those sheaves valued in the full subcategory of $\cS$ spanned by $\{\emptyset, \ast\}$. Unwinding the definitions, we must show that the right adjoint $f_{\ast}$ gives a bijection between the following
        \begin{enumerate}
            \item upward closed subsets $S \subset P$ such that if for every covering $\{\langle x \rangle \subset \langle y_i \rangle \mid i \in I\}$ satisfying $\langle y_i \rangle \in S$ for all $i$ the principal ideal $\langle x \rangle \in S$
            \item Limit-preserving functors $(\Rad(\cc{K})^{\vee})^{\op} \to [1]$. 
            \end{enumerate}

            \noindent and using the coherence of $\Rad(\cc{K})^{\vee}$, we see that item (b) above may be identified with:

            \begin{enumerate}
            \item[(b$^{\prime}$)] Ideals of $(\Rad(\cc{K})^{\omega})^{\op}$, i.e., upward closed subsets of the distributive lattice $\Rad(\cc{K})^{\omega}$ closed under finite meets, otherwise known as \tdef{filters}.
            \end{enumerate}

            We write $\Phi\colon \mathbf{2}^{P} \to \mathbf{2}^{\Rad(\cc{K})^{\omega}}$ to denote the map which sends a subset of $P$ to its direct image under $f$. Let $\Psi$ denote the map which assigns to a filter $F \subseteq \Rad(\cc{K})^{\omega}$ the subset of $P$ consisting of principal ideals whose radicals are in $F$. It is easy to see that $\Psi(F)$ is in the set (a), and that this in fact is an expression for the direct image $f_{\ast}F \in \Shv(P^{\op})$, where $F$ is naturally considered as an object of $\Shv(\Rad(\cc{K})^{\vee})$. It will thus suffice to prove that $\Phi$ supplies an inverse to $\Psi$.

            Given a filter $F \subset \Rad(\cc{K})^{\omega}$, the fact that $\Phi(\Psi(F)) = F$ is an immediate consequence of the surjectivity of $f^{\op}$. Conversely, let $S \subset P$ be a subset as in (a). Clearly $S \subset \Psi(\Phi(S))$, and we must show $\Psi(\Phi(S)) \subset S$. Let $ x \in \cc{K} $ be an element satisfying $\langle x \rangle \in \Psi(\Phi(S))$. By definition $\sqrt{x} \in \Phi(S)$, so there exists $ y \in \cc{K} $ satisfying $\langle y \rangle \in S $ and $\sqrt{x} = \sqrt{y}$. From $\langle y \rangle \subset \langle x \oplus y \rangle$ it follows that $\langle x \oplus y \rangle \in S$. Since $\{\langle x \rangle \subset \langle x \oplus y \rangle\}$ is a covering, $\langle x \rangle \in S$. 
\end{proof}

\begin{notation}
  We henceforth write $\mdef{\Spec \cc{K}}$ to refer to the absolute $\GBal$-spectrum of a 2-ring $\cc{K}$. We write $\mdef{|\! \Spec \cc{K}|}$ to denote the spectral space associated to $\Spec \cc{K}$ via the composite $|-| \circ \mm{Sub}$. By \cref{thm:2zariski_balmerspectrum}, this is identified with the Balmer spectrum $\Spc \cc{K}$.
\end{notation}

The results cited in \cref{rec:0topoi} and \cref{rec:stoneduality} combine to supply the following corollary.

\begin{corollary}
    For $\cc{K} \in \twoCAlg$ one has natural identifications between:
    \begin{enumerate}
        \item The space of maps $\Spec \cc{K} \to \cS \in \LTop^{[1]}$.
        \item Points of the space $|\! \Spec \cc{K}|$.
        \item Prime ideals $\cc{P} \subseteq \cc{K}$.
    \end{enumerate}
\end{corollary}

Given a map $f\colon \cc{K} \to \cc{L} \in \twoCAlg^{[1]}$ it is not hard to see that the proof of \cref{thm:2zariski_balmerspectrum} also allows us to compute the behavior of the induced map $\Spec \cc{L} \to \Spec \cc{K}$ on points. After the identification above, points of $\Spec \cc{K}, \Spec \cc{L}$ correspond to prime ideals of $\cc{K}, \cc{L}$, and the induced map sends $\{\cc{P} \subseteq \cc{L}\} \mapsto \{f^{-1}\cc{P} \subseteq \cc{K}\}$. We will demonstrate this as a corollary of a different statement in the following section. 

\subsection{Locally 2-ringed topoi}\label{subsec:locally2ringed} In analogy with the classical Zariski topology, we collect the following definitions.

\begin{notation}
  Henceforth, we will abusively identify a $\G$-structure $\cO_{\cX} \in \Str^{\loc}_{\G}(\cX)$ with its associated sheaf under the composite \[\Str^{\loc}_{\G}(\cX) \to \Fun^{\lex}(\G, \cX) \simeq \Shv(\cX; \Ind(\G^{\op})).\]
\end{notation}

\begin{definition} Set $\G = \GBal$.
  \begin{enumerate}
  \item Let $\mdef{\RTop_{\twoCAlg}}:= \LTop(\G_{\disc})^{\op}$ denote the $\infty$-category of \tdef{2-ringed topoi}.  Objects of $\RTop_{\twoCAlg}$ are pairs $\cX \in \RTop$, $\cO_{\cX} \in \Shv(\cX;\twoCAlg)$. Informally, morphisms $(\cX, \cO_{\cX}) \to (\cY,\cO_{\cY})$ are given by pairs \[f_{\ast}:\cX \to \cY \in \RTop^{[1]},\ \cO_{\cY} \to f_{\ast}\cO_{\cX} \in \Shv(\cY;\twoCAlg)^{[1]}.\]
  \item   Let $\mdef{\RTop^{\loc}_{\twoCAlg}} := \LTop(\cG)^{\op}$ denote the $\infty$-category of \tdef{locally 2-ringed topoi}. Objects of $\RTop^{\loc}_{\twoCAlg}$ consist of pairs $(\cX, \cO_{\cX})$ as above where $\cO_{\cX} \in \Str^{\loc}_{\G}(\cX)$. Informally, morphisms $(\cX,\cO_{\cX}) \to (\cY, \cO_{\cY})$ are given by pairs
  \[f_{\ast}:\cX \to \cY \in \RTop^{[1]},\ \cO_{\cY} \to f_{\ast}\cO_{\cX} \in \Shv(\cY;\twoCAlg)^{[1]}\]
  such that the associated mate $f^{\ast}\cO_{\cY} \to \cO_{\cX} \in \Shv(\cX;\twoCAlg)^{[1]}$ admits a lift to $\Str_{\G}^{\loc}(\cX)^{[1]}$. Note that since being a local transformation of $\G$-structures is a property, such a lift is essentially unique if it exists.
  \end{enumerate}
\end{definition}

We first remark on the connection between this notion and the notion of ``locality'' of 2-rings as in \cite[\S 4]{balmerSpectraSpectraSpectra2010}. We propose a slight modification to this notion.

\begin{definition}\label{def:locality} Let $\cc{K}\in \twoCAlg$. $\cc{K}$ is \tdef{local} if the thick tensor ideal $\{0\} \subset \cc{K}$ is prime. We write $\mdef{\twoCAlg^{\loc}}$ to denote the category of local 2-rings, with morphisms given by maps of 2-rings whose underlying functor is conservative. 
\end{definition}

\begin{proposition}\label{prop:localstructuresonapoint}
  There is an equivalence $\Str^{\loc}_{\GBal}(\cS) \simeq \twoCAlg^{\loc}$. 
\end{proposition}
\begin{proof}
  Consider the composite  \begin{equation}\label{eq:structuresonapoint} \Str^{\loc}_{\GBal}(\cS) \hookrightarrow \Fun^{\lex}(\twoCAlg^{\omega,\op}, \cS) =: \Ind(\twoCAlg^{\omega}) \simeq \twoCAlg \end{equation}  whose essential image we claim consists exactly of the local 2-rings. To this end, let $\cc{K}$ be in the essential image of this functor. Given $\cc{L} \in \twoCAlg^{\omega}$, the condition of arising from a local $\GBal$-structure implies that the following maps are surjective \begin{equation}\label{eq:localatapoint}\coprod_{i=1}^{n} \pi_{0}\Map\left(\cc{L}/\langle a_{i}\rangle, \cc{K}\right)  \twoheadrightarrow \pi_{0}\Map(\cc{L}, \cc{K})\end{equation} for every collection $a_{1},\dotsc,a_{n} \in \cc{L}$ satisfying $\sqrt{\cap_{i=1}^{n} \langle a_{i} \rangle} = \sqrt{0}$; equivalently, that $a_{1} \otimes \dotsb \otimes a_{n} \in \sqrt{0}$ by \cref{lem:radicalgeneration}. Now let $a, b \in \cc{K}$ be any pair such that $a \otimes b = 0$. This yields a map $\Sp^{\omega}\{x, y\}/\langle x \otimes y \rangle\to \cc{K}$ sending $x \mapsto a,\ y \mapsto b$, and the surjection of \eqref{eq:localatapoint} implies this must send either $\langle x \rangle$ or $\langle y \rangle$ to $0$. In particular, $\{0\} \in \Idl(\cc{K})$ is prime, implying that $\cc{K} \in \twoCAlg^{\loc}$. Conversely, assuming $\cc{K} \in \twoCAlg^{\loc}$ we must show that the map in \eqref{eq:localatapoint} is surjective. Given any $\cc{L} \in \twoCAlg^{\omega}$ and $a_{1},\dotsc,a_{n} \in \cc{L}$ satisfying $(a_{1} \otimes \dotsb \otimes a_{n})^{\otimes k} = 0$, there is an induced map $\Sp^{\omega}\{x_{1},\dotsc,x_{n}\}/\langle (x_{1} \otimes \dotsb \otimes x_{n})^{\otimes k}\rangle \to \cc{L}$ sending $x_{i} \mapsto a_{i}$. Given any map $\cc{L} \to \cc{K}$, the restriction $\Sp^{\omega}\{x_{1},\dotsc,x_{n}\}/\langle (x_{1} \otimes \dotsb \otimes x_{n})^{\otimes k}\rangle \to \cc{K}$ must send $x_{i} \mapsto 0$ for some $i$, since $\{0\} \in \Idl(\cc{K})$ is assumed to be prime. By \cref{cor:universalquotient}, we see that the map $\cc{L} \to \cc{K}$ must factor through $\cc{L}/\langle a_{i} \rangle$ for some $i$, yielding the first claim.

It remains to show that given $\cc{K}, \cc{K}' \in \twoCAlg^{\loc}$, a map $f\colon \cc{K} \to \cc{K}' \in \twoCAlg^{[1]}$ corresponds to a local transformation of associated $\GBal$ structures if and only if $f$ is conservative. Suppose first that $f$ arose from a local transformation of $\GBal$-structures: namely, that for any $\cc{L} \in \twoCAlg^{\omega}$ and $a \in \cc{L}$, the following square
  \begin{equation}\label{eq:localmapsonapoint}
    \begin{aligned}
      \xymatrix{
      \Map(\cc{L}/ \langle a \rangle, \cc{K}) \ar[r] \ar[d] & \Map(\cc{L}/\langle a \rangle, \cc{K}') \ar[d] \\
      \Map(\cc{L}, \cc{K}) \ar[r] & \Map(\cc{L}, \cc{K}') \\
      }
      \end{aligned}
  \end{equation}
  is Cartesian. Setting $\cc{L} = \Sp\{x\}$ and $a = x$ gives the Cartesian diagram
  \[
    \xymatrix{
      0 \ar@{=}[r] \ar[d] & 0 \ar[d] \\
      \cc{K} \ar[r] & \cc{K}' \\
    }
  \]
  implying that $f\colon \cc{K} \to \cc{K}'$ must be conservative. Conversely, assuming $g\colon \cc{K} \to \cc{K}'$ is conservative, we must show that the square in \eqref{eq:localmapsonapoint} is Cartesian. Let $\cc{L} \in \twoCAlg$ and $a \in \cc{L}$ arbitrary. Recall that \cref{prop:quotientposet} implies that the map $\Map(\cc{L}/\langle a \rangle, -) \to \Map(\cc{L}, -)$ has $(-1)$-truncated fibers, given exactly by $\Map_{\twoCAlg_{\cc{L}/}}(\cc{L}/\langle a \rangle, -)$. The desired claim thus reduces to showing that the map
  \[\pi_{0}\Map(\cc{L}/\langle a \rangle, \cc{K}) \to \pi_{0}\Map(\cc{L}, \cc{K}) \times_{\pi_{0}\Map(\cc{L},\cc{K}')} \pi_{0}\Map(\cc{L}/\langle a \rangle, \cc{K}') \] is surjective, or equivalently that any map $h: \cc{L} \to \cc{K}$ satisfies $(g \circ h)(a) = 0$ only if $h(a) = 0$. Since $g$ was assumed conservative, we may conclude. 
\end{proof}

\begin{remark} The following conditions are shown to be equivalent in \cite[\S 4]{balmerSpectraSpectraSpectra2010}:
  \begin{enumerate}
  \item The ideal $\sqrt{0} \subset \cc{K}$ is a minimal prime.
  \item The Balmer spectrum $\Spc \cc{K}$ admits a unique closed point.
  \item $\Spc \cc{K}$ is a \emph{local topological space} in the sense that any open cover $\coprod U_{i} \twoheadrightarrow \Spc \cc{K}$ admits a splitting $\Spc\cc{K} \to U_{i}$ for some $i$.
  \end{enumerate}
  Thus, \cref{def:locality} is equivalent to the condition that $\cc{K}\in \twoCAlg$ satisfies the equivalent conditions above and moreover that it contains no tensor-nilpotent elements. We view \cref{prop:localstructuresonapoint} as proposed justification for this addendum.
\end{remark}

\begin{notation}\label{not:stuffaboutadjoints}
  Given a map $f_{\ast}\colon\cX \to \cY \in \RTop^{[1]}$ and $\cc{C} \simeq \Ind(\cc{C}_{0})$, there is an induced adjunction
  \[
    f^{\ast}\colon \Fun^{\lex}(\cc{C}_{0}^{\op};\cY) \rightleftarrows \Fun^{\lex}(\cc{C}^{\op}_{0}, \cX) \noloc f_{\ast}
  \]
given by postcomposition with the respectively named adjoints. Under the equivalence of \cref{obs:sheafislexfunct}, $f_{\ast}\colon\Shv(\cX;\cc{C}) \to \Shv(\cY;\cc{C})$ is the functor which sends a sheaf $\cc{F}$ to $(\cc{F} \circ f^{\ast})\colon \cY^{\op} \to \cc{C}$, see for example the argument of \cref{lem:identificationofGstructure}. As before, we will freely use the notation $f^{\ast} \dashv f_{\ast}$ to refer to the adjunction either on categories of $\cG_{\disc}$-structures or of $\Ind(\cG^{\op})$-valued sheaves. 
\end{notation}

 We obtain the following corollary as an immediate consequence of the previous proposition, the slogan being that the stalks of a locally 2-ringed topos are local 2-rings, and maps between locally 2-ringed topoi induce conservative morphisms on stalks.

\begin{corollary}\label{cor:mapsonstalks}
  Let $(\cX, \cO_{\cX}) \in \RTop^{\loc}_{\twoCAlg}$.
  \begin{enumerate}
  \item Given any point $p_{\ast}\colon \cS \to \cX \in \RTop^{[1]}$, the pullback $p^{\ast}\cO_{\cX} \in \twoCAlg$ is a local 2-ring. 
  \item Given any point $p_{\ast}\colon \cS \to \cX \in \RTop^{[1]}$ and a map of locally 2-ringed topoi $f\colon (\cX, \cO_{\cX}) \to (\cY,\cO_{\cY})$, the induced map of 2-rings $p^{\ast}f^{\ast}\cO_{\cY} \to p^{\ast}\cO_{\cX} \in \twoCAlg^{[1]}$ is conservative. 
  \end{enumerate}
\end{corollary}

In the special case of $\Spec \cc{K} \in \RTop^{\loc}_{\twoCAlg}$ for some $\cc{K} \in \twoCAlg$, the equivalence of \cref{thm:2zariski_balmerspectrum} additionally endows us with a full understanding of the $\GBal$-structure on the points of $\Spec \cc{K}$. We will first need the following lemma, which roughly says that the sheafification functor does not affect stalks. 

\begin{lemma}\label{lem:sheafificationandstalks}
  Let $F$ be a frame, $\cc{C} \simeq \Ind(\cc{C}_{0})$, and $x \colon F \to [1]$ a point. There is a commutative square
  \[\xymatrix{
      \Fun(F^{\op}, \cc{C}) \ar[d] \ar[r]^{x_{!}} & \Fun([1]^{\op}, \cc{C}) \ar[d]^{\cc{F} \mapsto \cc{F}(1)} \\
      \Shv(F; \cc{C}) \ar[r]^{x^{\ast}} & \Shv(\ast; \cc{C}) \simeq \cc{C}
    }
  \]
  where the vertical morphisms are the natural sheafification maps and $x_{!}$ denotes the left Kan extension along $x^{\ast}$.
\end{lemma}
\begin{proof}
  We will demonstrate the above in the case $\cc{C} = \cS$, as the full case will follow from an application of $\Fun^{\lex}(\cc{C}_{0}^{\op},-)$, see \cref{not:stuffaboutadjoints}. Note that the sheafification map $\cc{P}(F) \to \Shv(F)$ is determined by its restriction to the Yoneda image $F \to \Shv(F)$, using the universal property of the Yoneda embedding \cite[\S 5.3.6]{LurieHTT}. We moreover obtain a homotopy commutative diagram
  \[\xymatrix{
      F \ar[r]^{\Yo} \ar[d]^{x^{\ast}}\ar[r]^{\Yo} &\cc{P}(F) \ar[d]^{x_{!}} \ar[r] & \Shv(F)  \ar[d] \\
      [1] \ar[r]^{\Yo} &  \cc{P}([1]) \ar[r] & \Shv(\ast)
      }
    \]
    from the naturality of the Yoneda embedding. It remains to show that the sheafification map $\cc{P}([1]) \to \Shv(\ast)$ is given by evaluation at $1$, which is an immediate consequence of the fact that the right adjoint inclusion is the functor $\Shv(\ast) \to \cc{P}([1])$ sending a space $X$ to the unique morphism $X \to \ast \in \cS^{[1]^{\op}}$. 
\end{proof}

\begin{lemma}\label{lem:identificationofstalks}
  Let $\cc{K} \in \twoCAlg$, and let $\cc{P} \subseteq \cc{K}$ be a prime ideal with associated point $x_{\cc{P}, \ast}\colon \cS\to \Spec \cc{K}$. Under the equivalence of \cref{prop:localstructuresonapoint}, the pullback $x_{\cc{P}}^{\ast}(\cO_{\cc{K}}) \in \Str^{\loc}_{\GBal}(\cS)$ corresponds exactly to $\cc{K}/\cc{P} \in \twoCAlg^{\loc}$.
\end{lemma}
\begin{proof}
  From \cref{rec:0topoi} and \cref{thm:2zariski_balmerspectrum} the map $x_{\cc{P}}^{\ast}$ arises from application of $\Shv(-)$ to the map of frames $x_{\cc{P}}\colon\Rad(\cc{K})^{\vee} \to [1]$ which sends a radical ideal $\cc{I}$ to $1$ if and only if $\cc{I} \subseteq \cc{P}$, and $0$ otherwise. \cref{lem:sheafificationandstalks} implies that $x^{\ast}_{\cc{P}}\cO_{\cc{K}} \in \twoCAlg$ may be computed as the value at $1$ of the left Kan extension along $x_{\cc{P}}^{\op}$ of $\widetilde{\cO}_{\cc{K}}\colon \Rad(\cc{K})^{\omega}\to \twoCAlg$ which sends $\cc{I} \mapsto \cc{K}/\cc{I}$. We obtain the expression
  \[
    x^{\ast}_{\cc{P}}(\cO_{\cc{K}}) \simeq {\varinjlim}_{\Rad(\cc{K})^{\omega}_{/\cc{P}}} \cc{K}/\cc{I}
  \]
  and by \cref{prop:quotientposet} this is exactly $\cc{K}/\cc{P}$.
\end{proof}

We arrive at the application promised at the end of the previous subsection.

\begin{corollary}
  Given a map $f: \cc{K} \to \cc{L} \in \twoCAlg^{[1]}$, the induced map $f_{\ast}\colon \Spec \cc{L} \to \Spec \cc{K} \in \RTop^{[1]}$ sends points of the form $x_{\cc{P},\ast}: \cS \to \Spec \cc{L}$ to $x_{f^{-1}\cc{P}, \ast}: \cS \to \Spec \cc{K}$, for $\cc{P} \subseteq \cc{L}$ is a prime ideal.
\end{corollary}
\begin{proof}
  Note that $f_{\ast}x_{\cc{P},\ast}$ corresponds to a prime ideal $\cc{Q} \subseteq \cc{K}$. Let $x_{\cc{Q},\ast}\coloneq f_{\ast}x_{\cc{P},\ast}$. \cref{cor:mapsonstalks} implies that the induced map $x_{\cc{Q}}^{\ast}\cc{O}_{\cc{K}} \to x_{\cc{P}}^{\ast}\cc{O}_{\cc{L}} $ is identified with a map $\cc{K}/\cc{Q} \to \cc{L}/\cc{P}$ in $\twoCAlg^{[1]}_{\cc{K}/}$. \cref{prop:quotientposet} shows that $\cc{Q} \subseteq f^{-1}\cc{P}$, and the condition of conservativity from \cref{lem:identificationofstalks} forces $\cc{Q} = f^{-1}\cc{P}$. 
\end{proof}

\subsection{Comparison transformations}\label{ssec:zariski_comparison}

We now take a detour to discuss comparison transformations between the classical and higher Zariski geometries. The main result is \cref{prop:diraccomp}, which recovers a comparison map originally constructed in \cite[Theorem 5.3]{balmerSpectraSpectraSpectra2010}. This map is one of the early innovations of the tt-geometric perspective towards the classification of tensor ideals.

\begin{notation}
  We write $\mdef{\cc{R}_{(-)}}$ to indicate the functor sending $\cc{C} \in \CAlg(\PrLst)$ to the endomorphism ring spectrum $\Hom_{\cc{C}}(\unit, \unit) \in \CAlg$.
\end{notation}

\begin{construction}\label{rem:smallschwedeshipley}
    The results of \cite[\S 4.8.5]{LurieHA} and \cite[7.3.2.12]{LurieHA} supply an adjunction
  \[
    \Mod\colon \CAlg \rightleftarrows \CAlg(\PrLst) \noloc \cc{R}_{(-)}
  \]
  with fully faithful left adjoint. Note that all compact objects in a category of modules are in the thick subcategory generated by the unit \cite[7.2.4.2]{LurieHA}, and in particular for $R \in \CAlg$, $\cc{C} \in \CAlg(\PrLstomega)$ any exact functor $\Mod_{R} \to \cc{C}$ that preserves the unit must preserve compact objects. In particular, there is a refined adjunction \[\Perf\colon \CAlg \rightleftarrows  \twoCAlg \noloc \cc{R}_{(-)} \] after composing with the equivalence $\CAlg(\PrLstomega) \simeq \twoCAlg$ from \cref{rem:bigandsmall}.
\end{construction}
 
 \begin{lemma}\label{lem:endfilteredcols}
   The functor $\cc{R}_{(-)}: \twoCAlg \to \CAlg$ preserves filtered colimits. 
 \end{lemma}
 
 \begin{proof}
   First, note that filtered colimits in $\twoCAlg$ can be computed in $\Catperf$ and hence in $\Catex$, using that the localization $\Catex \to \Catperf$ preserves compact objects. Recall that $\prod_{\bb{Z}}\Omega^{\infty + n}: \Sp \to \prod_{\bb{Z}} \mathcal{S}_{\ast}$ creates filtered colimits. It suffices then to show that the composite \begin{equation}\label{eq:composite}\CAlg(\Catex) \xrightarrow{\cc{R}_{(-)}} \CAlg \xrightarrow{\prod_{\bb{Z}}\Omega^{\infty + n}} \prod\nolimits_{\bb{Z}} \mathcal{S}_{\ast}\end{equation} preserves filtered colimits. In the $n$th component this is given by 
   \[\cc{K} \mapsto \Omega^{\infty + n}\Hom_{\cc{K}}(\unit, \unit) \simeq \Map_{\cc{K}}(\unit, \Sigma^{-n}\unit)\]
   which in particular commutes with filtered colimits in $\CAlg(\Catex)$, since these may be computed in $\Catex$ and thus directly in $\Cat$.
 \end{proof}

\begin{corollary}\label{cor:perfpreservescompacts}
    The functor $\Perf: \CAlg \to \twoCAlg$ preserves compact objects.
\end{corollary}

Using this, we may compare the Zariski geometry on commutative 2-rings and the classical Zariski geometry on ring spectra. 

\begin{proposition}\label{prop:2-1Zariski_comparison}
    The assignment $A\mapsto\Perf_{A}$ defines a morphism of geometries $\GDir\to\GBal$.
\end{proposition} 

\begin{proof}
  \cref{cor:perfpreservescompacts} implies that $\Perf$ indeed defines a functor $\CAlg^{\omega} \to \twoCAlg^{\omega}$. Let $A$ be an $\bE_{\infty}$-ring. For $a\in\pi_{\ast}A$, $\Perf_{A}\to\Perf_{A[a^{-1}]}$ is a localization with kernel $\langle \cofib(A\xrightarrow{a} A)\rangle$, and hence it preserves admissible morphisms. Now consider $\{a_{i}\}_{i \in I}\in\pi_{\ast}A$ a finite collection of homogeneous elements such that $\{a_i\}_{i \in I}$ generates the unit ideal. We wish to see that $ \mm{Kos}(a_{1},\dotsc,a_{n}) := \cofib(A\xrightarrow{a_1}A) \otimes_A \dotsb \otimes_A \cofib(A\xrightarrow{a_n}A) \simeq 0$. For all $i$, $a_{i}^{2}$ annihilates $\cofib(A \xrightarrow{a_{i}} A)$, and in particular must annihilate $\mm{Kos}(a_{1},\dotsc,a_{n})$. We have shown that $(a_{1},\dotsc,a_{n}) \subseteq  \mm{ann}_{\mm{Kos}(a_{1},\dotsc,a_{n})}$ and by assumption the former generates the unit ideal, yielding the claim.
\end{proof}

\begin{remark}
  The same proof as above implies that $\Perf$ defines a morphism of geometries $\GZar \to \GBal$.
\end{remark}

Recall that \cref{obs:restrictionfunctor} associates to the morphism of geometries $\Perf: \GDir \to \GBal$ a functor \[\mm{res}_{\Perf} \colon \LTop(\GBal) \to \LTop(\GDir) \] 
which is pointwise given by the assignment $(\cX, \cO) \mapsto (\cX, \cO \circ \Perf)$.

\begin{lemma}\label{lem:affinizingdiracrestriction}
  There is an equivalence of functors from $\LTop(\GZar)$ to $\CAlg$
        \[\Gamma_{\GDir}\circ \mm{res}_{\Perf} \simeq \cc{R}_{(-)} \circ \Gamma_{\GBal}.\]
\end{lemma}
\begin{proof}
  Unwinding the definitions, the composite $\Gamma_{\GDir} \circ \mm{res}_{\Perf}\colon \LTop(\GBal) \to \CAlg \subseteq \Fun(\GZar, \cS)$ is identified with the following composite
  \begin{align*}
     \LTop(\GBal)\rightarrow \Fun(\GBal, \overline{\LTop}) \xrightarrow{- \circ \Perf}  & \Fun(\GZar, \overline{\LTop}) \xrightarrow{\Gamma \circ -} \Fun(\GZar,\cS).
  \end{align*}
  There is a commutative diagram as below, from the naturality of precomposition
  \[\xymatrix{
      \Fun(\GBal, \overline{\LTop}) \ar[d]_{-\circ \Perf} \ar[r]^-{\Gamma \circ -}& \Fun(\GBal, \cS) \ar[d]^{- \circ \Perf} \\
      \Fun(\GZar, \overline{\LTop}) \ar[r]^-{\Gamma \circ -} &  \Fun(\GZar, \cS)
    }
  \]
  from which one obtains the commutativity of the following diagram
  \[\xymatrix{
    \LTop(\GBal) \ar[d]_{\mm{res}_{\Perf}} \ar[r]^-{\Gamma_{\GBal}}& \Ind(\GBal^{\op}) \subseteq \Fun(\GBal, \cS) \ar[d]^{- \circ \Perf} \\
      \LTop(\GZar)  \ar[r]^-{\Gamma_{\GZar}} &  \Ind(\GZar^{\op}) \subseteq \Fun(\GZar, \cS).
    }
  \]

  Note that the precomposition $- \circ \Perf\colon \Ind(\GBal^{\op}) \to \Ind(\GZar^{\op})$ may be identified with the right adjoint to the functor 
  \[\Ind(\GBal^{\op}) \to \Ind(\GZar^{\op})\] induced from the map $\Perf\colon \GZar^{\op} \to \GBal^{\op}$. Under the equivalences $\Ind(\GZar^{\op}) \simeq \CAlg$, $\Ind(\GBal^{\op}) \simeq \twoCAlg$, this left adjoint may be identified with $\Perf\colon \CAlg \to \twoCAlg$. Thus, its right adjoint $- \circ \Perf\colon \Ind(\GBal^{\op}) \to \Ind(\GZar^{\op})$ is given by $\cc{R}_{(-)}$ under the same equivalences, from which the diagram above supplies the claim.
\end{proof}

\begin{definition}
Let $\mdef{\RTop^{\mm{Dir}}_{\CAlg}} \coloneq \LTop(\GDir)^{\op}$ denote the $\infty$-category of \tdef{Dirac-locally spectrally ringed topoi}.
\end{definition}

\begin{construction}
Consider the following counit transformation associated to the adjunction of \cref{lem:globalsections}
\begin{equation}\label{eq:diraccompstep1}
  \mm{id} \Rightarrow \Gamma_{\GBal}(\Spec(-), \cc{O}) \in \Fun(\twoCAlg, \twoCAlg)^{[1]}.
\end{equation} Composing \eqref{eq:diraccompstep1} with the functor $\cc{R}_{(-)}$ and applying \cref{lem:affinizingdiracrestriction}, we obtain
\begin{equation}\label{eq:diraccomponsections}\cc{R}_{(-)} \Rightarrow \Gamma_{\GDir}(\Spec(-) , \cO \circ \Perf) \in \Fun(\twoCAlg, \CAlg)^{[1]}. \end{equation} Passing to mates, we obtain a transformation of the form
  \begin{equation}\label{eq:diraccomp}
    \mdef{\rho} \colon (\Spec(-), \cO \circ \Perf) \Rightarrow \Spec^{\GDir}\cc{R}_{(-)} \in \Fun(\twoCAlg^{\op},\RTop^{\Dir}_{\CAlg})^{[1]}
  \end{equation} 
which is pointwise given by the adjoint of the map \eqref{eq:diraccomponsections}. 
\end{construction}

The transformation $\rho$ described in \eqref{eq:diraccomp} generalizes the comparison transformation of \cite[Theorem 5.3]{balmerSpectraSpectraSpectra2010}, as we now show.

\begin{proposition}\label{prop:diraccomp}
  Let $\cc{K} \in \twoCAlg$. The underlying morphism of $\infty$-topoi for the transformation of \eqref{eq:diraccomp} evaluated on $\cc{K}$ is given on points by the map
  \[\cc{P}\in |\! \Spec \cc{K}| \mapsto \{f \in \pi_{\ast}R_{\cc{K}} \text{ homogeneous}\mid \cofib(f\colon \unit \to \unit) \notin \cc{P}\} \in \Spec^{\mm{h}}\! \pi_{\ast}R_{\cc{K}}\]
\end{proposition}
\begin{proof}
  Let $\cc{P} \subseteq \cc{K}$ be a prime ideal, and $x_{\ast, \cc{P}}\colon \cS \to \Spec \cc{K} \in \RTop^{[1]}$ the associated point. Note that \cref{lem:resintermsofsheaves} implies that $\cc{R_{O_{K}}} \in \Shv(\Spec \cc{K}; \CAlg)$ is the sheaf corresponding to the $\GDir$-structure $\cO_{\cc{K}} \circ \Perf$ on $\Spec \cc{K}$. This provides an equivalence \begin{equation}\label{eq:idofstalksdirac}
    x_{\cc{P}}^{\ast}\cc{R_{O_{K}}} \simeq {\varinjlim}_{\Rad(\cc{K})^{\omega}_{/\cc{P}}} \cc{R}_{\cc{K}/\cc{I}}\end{equation} using the same argument as \cref{lem:identificationofstalks}. Since $\Rad(\cc{K})^{\omega}_{/\cc{P}}$ admits finite joins, it is filtered, and applying \cref{lem:endfilteredcols} implies that the right-hand side of \eqref{eq:idofstalksdirac} is equivalent to $\cc{R}_{\cc{K}/\cc{P}}$. By \cref{thm:diracspectrum} the point $\gamma_{\ast}x_{\cc{P},\ast}$ corresponds to a homogeneous prime ideal $\mathfrak{p} \subseteq \pi_{\ast}\cc{R_{K}}$, so we may write $x_{\mathfrak{p}, \ast}\coloneq \gamma_{\ast}x_{\cc{P},\ast}$. Furthermore, the induced morphism \[x^{\ast}_{\mathfrak{p}}\cc{O_{R_{\cc{K}}}} \to x^{\ast}_{\cc{P}}\cc{R_{O_{K}}} \in \CAlg_{\cc{R}_{\cc{K}}/}^{[1]}\] is identified with a $\pi_{\ast}$-local map $\cc{R}_{\cc{K}, \mathfrak{p}} \to \cc{R_{K/P}}$ by \cref{lem:diracstalks}, \cref{lem:diraclocality}. Since the map
  \[
    \cc{R_{K}} \coloneq \Hom_{\cc{K}}(\unit, \unit) \to \Hom_{\cc{K}/\cc{P}}(\unit, \unit) \eqcolon \cc{R_{K/P}}
  \]
  inverts exactly those maps with cofiber in $\cc{P}$, it follows that $\mathfrak{p}\subseteq \cc{R_{K}}$ as a set must consist exclusively of maps with cofiber not contained in $\cc{P}$. 
\end{proof}

\begin{remark}
  The result above holds in the setting where $\GDir$ and $\Spec^{h}\pi_{\ast}\cc{R_{\cc{K}}}$ are replaced by $\GZar$ and $\Spec\pi_{0}\cc{R_{K}}$, respectively, using essentially the same proof. 
\end{remark}

\begin{observation}\label{obs:alternativepaths}
  Since $\Spec^{h}\pi_{\ast}\cc{R_{K}}$ is a spectral space \cite[Proposition 2.24]{hesselholtDiracGeometryCommutative2023}, it arises from a spatial frame. We thus note that the comparison map
  \[
    \Spec \cc{K} \to \Shv(\Spec^{h}\pi_{2\ast}\cc{R_{K}}) \in \RTop^{[1]}
  \]
  is fully determined by the induced morphism $\cc{U}(\Spec^{h}\pi_{2\ast}R) \to \Rad(\cc{K})^{\vee}$ obtained from passage to subterminal objects, see \cref{rec:0topoi}. From the description on points given by \cref{prop:diraccomp}, we see that any quasicompact open subset $D(f) \in \cc{U}(\Spec^{h}\pi_{2\ast}R)$ has preimage given by $\langle \cofib(f\colon \unit \to \unit)\rangle \in \Rad(\cc{K})^{\vee}$ which is itself quasicompact. It follows that the comparison transformation above is induced by a map of coherent frames, and hence from the following map of posets
  \begin{equation}\label{eq:perfrestrictstoproad} \Perf\colon\Pro(\GZar)^{\ad}_{/\cc{R}_{\cc{K}}} \to \Pro(\GBal)^{\ad}_{/\cc{K}} \end{equation}
  which sends a localization $R \to R[f^{-1}]$ to the associated map $\Perf_{R} \to \Perf_{R[f^{-1}]}$.
\end{observation}

\begin{warning}
    Note that the left-hand object in \eqref{eq:perfrestrictstoproad} is not a distributive lattice, but only a \emph{distributive lower semi-lattice}. This does not affect the assertion that the cited map uniquely determines the comparison transformation $\rho$ described in \eqref{eq:diraccomp}; since the open subsets $D(f)$ form a basis of quasicompact opens of $\Spec^{h} \pi_{2\ast}R_{\cc{K}}$, any object of $\cc{U}(\Spec^{h} \pi_{2\ast}R_{\cc{K}})$ is a formal colimit of objects of the form $D(f)$ and hence the map is uniquely determined by its restriction to $\Pro(\GZar)^{\ad}_{/\cc{R}_{\cc{K}}} \subseteq \cc{U}(\Spec^{h} \pi_{2\ast}R_{\cc{K}})^{\omega}$
\end{warning}

\begin{remark}
  In fact, it is possible to prove \cref{prop:diraccomp} in a more streamlined fashion by directly demonstrating that the comparison transformation is induced by the map \eqref{eq:perfrestrictstoproad}; in fact, this is a general result about the counit of the adjunction of \cref{thm:relativespec} evaluated on an absolute spectrum. Given the prior detour into the behavior of points and requisite preliminaries on Stone duality, we find our present approach to be the most conceptual path. 
\end{remark}

The observation above supplies the following result, first proved as a bijection of posets in the Noetherian case by work of Hopkins \cite{hopkinsGlobalMethodsHomotopy1987} and Neeman \cite{neemanChromaticTower1992}. A prototype of the $\GZar$-structured statement is first recorded on the level of $\pi_{0}$ in \cite{Balmer02}. The result below is obtained on the level of frames as \cite[Theorem 2.1.9]{KockPitsch17}, and we merely sketch the necessary modifications of that argument.

\begin{theorem}\label{thm:hopkinsneemanbalmer}
  Given an ordinary commutative ring $R \in \CAlg^{\heartsuit}$, the comparison transformation
  \[
    \Spec \Perf_{R} \to \Spec^{\GZar} R \in \RTop(\GZar)^{[1]}
  \]
  is an equivalence.
\end{theorem}
\begin{proof}
   After \cref{obs:alternativepaths}, the statement on the level of $\RTop$ will follow from demonstrating that the map of \eqref{eq:perfrestrictstoproad} induces an equivalence upon passage to ideal frames. Since $\Perf$ is fully faithful, it will suffice in fact to show that any ideal of $\Rad(\Perf_{R})^{\op}$ is determined by its intersection with the ideals in the image of $\Pro(\GZar)^{\ad}_{/R}$. Unwinding the definitions and using the coherence of $\Rad(\Perf_{R})^{\vee}$, we are tasked with showing that any finitely generated radical ideal of $\cc{K}$ is generated by finitely many objects of the form $\cofib(f)$ for $f \in R$, which is the statement of \cite[Proposition 2.1.13]{KockPitsch17}. For the identification of $\cG$-structures, we note that $\Perf_{R}$ is rigid and \cref{thm:rigid_subcanonical} applies; by \cref{lem:resintermsofsheaves} the restricted $\GZar$-structure on $\Spec \Perf_{R}$ may thus be identified with the sheaf assigning to an open subset $D(f)$ the endomorphism ring spectrum of the unit of $\Perf_{R}/\langle \cofib(f)\rangle \simeq \Perf_{R[f^{-1}]}$, which is $R[f^{-1}]$. This manifestly agrees with the structure sheaf on $\Spec^{\GZar}R$. 
\end{proof}

\subsection{Thomason's theorem}\label{ssec:tt_thomason}
\cref{thm:hopkinsneemanbalmer} admits a globalization to the setting of quasicompact quasiseparated schemes, demonstrated in \cite{Balmer02}. In the language of Zariski geometries, this is a specialization of the following fact, which will be proven in \cite{ChedRecon}.

\begin{theorem}\label{thm:affineness}
  Let $\cc{G} = \GZar$ and let $\cc{G}' = \GBal$. For any quasicompact quasiseparated nonconnective spectral scheme $X$ considered as a locally ringed topos, there is an equivalence of locally 2-ringed topoi
  \[
    \Spec^{\cc{G}'}_{\cc{G}}X \simeq \Spec \Perf_{X}
  \]
which is moreover natural in $X$. The same result holds if $X$ is assumed to be a spectral Dirac scheme, suitably defined.
\end{theorem}

The proof of this result will eschew the use of support-theoretic techniques, and instead follow from a general study of \emph{affine 2-schemes} for the geometry $\GBal$.  
\vspace{5pt}

\section{Zariski Descent}\label{sec:structuresheaf}
Classical Zariski geometry rests on the ability to perform \emph{descent} with respect to the Zariski topology: this, for example, is what enables the gluing of maps and the construction of schematic moduli problems. In this section we will prove several instances of descent over the Zariski spectra of 2-rings. 

\subsection{Zariski descent and Mayer--Vietoris} 
In this subsection we supply our basic reduction schema for proving Zariski descent, \cref{prop:sheavesonspec}. We first recall the following characterization of sheaves on a frame, recorded as~\cite[Theorem 3.19]{aokiSheavesspectrumAdjunction2023}.

\begin{theorem}\label{thm:mvsuffices} Let $F$ be a frame. For an $\infty$-category with limits $\cc{C}$, the subcategory $\Shv(F; \cc{C}) \subseteq \Fun(F^{\op}, \cc{C})$ consists exactly of those presheaves $\mathcal{F}$ satisfying:
  \begin{enumerate}
  \item The value $\mathcal{F}(\bot)$ is the final object.
  \item\label{thmmvsuffices(b)} For any opens $U, U' \in F$, the square \[ \xymatrix{ \mathcal{F}(U \vee U') \ar[r] \ar[d] & \mathcal{F}(U')\ar[d] \\ \mathcal{F}(U) \ar[r] & \mathcal{F}(U \wedge U')
      } \] is Cartesian.
  \item The canonical morphism \[\mathcal{F}\left(\bigvee D\right) \to \varprojlim_{U \in D}\mathcal{F}(U)\] is an equivalence for any directed subset $D \subseteq F$.
  \end{enumerate}
\end{theorem}

\begin{lemma}\label{lem:qcmvsuffices}
  For a coherent frame $F$, the restriction $\Shv(F;\cc{C}^{\op}) \hookrightarrow \Fun(F^{\omega,\op}, \cc{C})$ is fully faithful with essential image those functors $\mathcal{F}$ satisfying conditions (a) and (b) of \cref{thm:mvsuffices} on compact elements.
\end{lemma}
\begin{proof}
  By assumption, $F^{\omega}$ generates $F$ under directed colimits, so the full faithfulness follows from condition (c) of the cited theorem. Now let $\mathcal{F} \in \Fun(F^{\op}, \cc{C})$ be a functor satisfying conditions (a) and (c), such that its restriction to $\mathcal{F}^{\omega}$ satisfies condition (b).  Let $U, U' \in F$ arbitrary opens of $F$. It is possible to exhaust $\{U, U', U \vee U', U \wedge U'\}$ by a directed sequence of quasicompact opens $\{U_{i}, U_{i}', U_{i}\vee U_{i}', U_{i} \wedge U_{i}'\}$ by the assumption of coherence. Taking a directed limit of the Cartesian squares associated to $U_{i}, U_{i}'$ and invoking condition (c) yields condition (b) on all of $F$, yielding the claim. 
\end{proof}

\begin{remark}
  In fact, the above proof only requires that $F$ is quasiseparated and has a basis of quasicompact opens: namely, that $F$ is \emph{locally coherent}.
\end{remark}

\begin{definition}
  Let $\cc{K} \in \twoCAlg$. We say that a map $f\colon \cc{K} \to \cc{K}'$ is a \tdef{Zariski cover} if $f$ exhibits $\cc{K}'$ as a finite product of Karoubi quotients $\cc{K} \to \prod_{i=1}^{n} \cc{K}_{i}$ away from principal ideals $\cc{I}_{i}$ satisfying $\bigcap_{i=1}^{n}\cc{I}_{i} \subseteq \sqrt{0}$.  We refer to the Grothendieck topology on $\twoCAlg^{\op}$ generated by Zariski covers as the \tdef{Zariski topology} on $\twoCAlg^{\op}$.
\end{definition}

\begin{remark}
  The proof of \cref{lem:admissibleprincipal} implies that the Zariski topology on $\twoCAlg^{\op} \simeq \Pro(\GBal)$ recovers the pro-admissible topology for the Zariski geometry on the latter.
\end{remark}

\begin{lemma}\label{lem:preimagesofidealsandradicals}
  Let $\cc{K} \in \twoCAlg$. Given $\cc{I} \in \Idl(\cc{K})$, the map $f^{-1}\colon \Idl(\cc{K}/\cc{I}) \to \Idl(\cc{K})$ which sends a thick tensor ideal to its preimage is fully faithful with image $\Idl(\cc{K})_{\cc{I}/}$. The same result holds upon restricting $f^{-1}$ to $\Rad(-)$,  to $\mm{Prin}(-)$ when $\cc{I} \in\mm{Prin}(\cc{K})$, and to $\Rad(\cc{K})^{\omega}$ when $\cc{I} \in \Rad(\cc{K})^{\omega}$. 
\end{lemma}
\begin{proof}
 Let $\cc{I} \in \Idl(\cc{K})$ be arbitrary, and $f\colon \cc{K} \to \cc{K}/\cc{I}$ the associated localization. There is an induced map $\Idl(\cc{K}/\cc{I}) \to \Idl(\cc{K})$ given by sending $\cc{J} \to f^{{-1}}(\cc{J})$ for $\cc{J} \in \Idl(\cc{K}/\cc{J})$. The proof of \cref{cor:retractsandleftcancel} shows that this map fits into the following square
 \begin{equation}
   \begin{aligned}
     \xymatrix{
\Idl(\cc{K}/\cc{I}) \ar[d]_{(\cc{K}/\cc{I})/-} \ar[r]^{f^{-1}} & \Idl(\cc{K}) \ar[d]^{\cc{K}/-} \\
\twoCAlg_{\cc{K}/\cc{I}} \ar[r] &  \twoCAlg_{\cc{K}}} 
   \end{aligned}
 \end{equation} where the bottom map is fully faithful since $\cc{K} \to \cc{K}/\cc{I}$ is a localization, and hence $\Idl(\cc{K}/\cc{I}) \to \Idl(\cc{K})$ is fully faithful with image $\Idl(\cc{K})_{\cc{I}/}$.  The left adjoint of the bottom map is given by $\cc{K}/\cc{I}\otimes_{\cc{K}}-$, and the closure of Karoubi quotients under base change implies that this restricts to a left adjoint of $f^{-1}\colon \Idl(\cc{K}/\cc{I}) \to \Idl(\cc{K})$. Applying $\ker(-)$ identifies the right adjoint to this map with the inclusion $\Idl(\cc{K})_{\cc{I}/} \to \Idl(\cc{K})$, whose left adjoint is given taking the join $\cc{I} \vee -$. Altogether, we find that $\mm{ker}(-) \circ (\cc{K}/\cc{I} \otimes_{\cc{K}} -)\colon \Idl(\cc{K}) \to \Idl(\cc{K})_{\cc{I}/}$ is given by the join $\cc{I} \vee -$.

 Let us now prove that under the assumption that $\cc{I}$ is principal, the map $f^{-1}$ restricts to $\mm{Prin}$. From the equivalence of \cref{prop:quotientposet}, we are reduced to showing that the induced map $\twoCAlg_{(\cc{K}/\cc{I})/} \to \twoCAlg_{\cc{K}/}$ sends compact localizations under $\cc{K}/\cc{I}$ to compact localizations under $\cc{K}$. Let $\cc{K}/\cc{I} \to \cc{K}/\cc{J}$ be a localization which is compact as an object of $\twoCAlg_{(\cc{K}/\cc{I})/}$, and let $F\colon I \to \twoCAlg_{\cc{K}/}$ be an arbitrary directed system which admits a map \[h\colon \cc{K}/\cc{J} \to {\varinjlim}_{I}F \in \twoCAlg_{\cc{K}/}^{[1]}.\] Since $\cc{I}$ was assumed to be a principal ideal of $\cc{K}$, \cref{prop:quotientposet} implies that the composite map $g\colon \cc{K}/ \cc{I} \to \varinjlim_{I} F$  lifts through a (necessarily unique) map $g_{i}\colon \cc{K}/ \cc{I} \to F(i)$ for $i \in I$. Reindexing $F$ over $I_{\geq i}$, we have that $F$ admits a lift to $\twoCAlg_{(\cc{K}/\cc{I})/}$ and furthermore that $\varinjlim_{I}F$ may be computed in either $\twoCAlg_{\cc{K}/}$ or $\twoCAlg_{\cc{K}/\cc{I}}$, as the latter fully faithfully embeds into the former and fully faithful embeddings create colimits. Since $\cc{K}/ \cc{J}$ was assumed to be compact as an object of $\twoCAlg_{(\cc{K}/\cc{I})/}$, it follows that the map $h\colon \cc{K}/ \cc{J} \to \varinjlim_{I}F$ lifts through a (necessarily unique) map $h_{j}\colon \cc{K}/ \cc{J} \to F(j)$ for some $j \in I$. Thus, $\cc{K}/ \cc{J}$ is compact as an object of $\twoCAlg_{\cc{K}}$, yielding the claim. 

 It now remains to approach the corresponding statement on $\Rad$. Given any $\cc{J} \in \Rad(\cc{K}/\cc{I})$, $x^{\otimes n} \in f^{-1}(\cc{J})$ if and only if $f(x^{\otimes n}) \in \cc{J}$; it follows that that there is an identification $f^{-1} \circ \sqrt{-} = \sqrt{-} \circ f^{-1}$ of maps from $\Idl(\cc{K}/\cc{I})$ to $\Idl(\cc{K})$. This implies that $f^{-1} \circ \sqrt{-} = f^{-1}$ on $\Rad(\cc{K}/\cc{I})$ takes values in $\Rad(\cc{K})$, and similarly for $\Rad(\cc{K}/\cc{I})^{\omega}$ under the assumption that $\cc{I} = \sqrt{\cc{I}'}$ for $\cc{I}'$ a principal ideal. 
\end{proof}

\begin{proposition}\label{prop:sheavesonspec}
  Let $\cc{C}$ be a full subcategory of $\twoCAlg$ closed under finite products and Zariski covers. Let $\cc{D}$ be an $\infty$-category admitting all small limits and $\cc{F}\colon \cc{C} \to \cc{D}$ any functor. Then for every $\cc{K} \in \cc{C}$, the induced functor
  \[
    \Rad(\cc{K})^{\omega} \xrightarrow{\cc{K}/-} \cc{C}_{\cc{K}/} \xrightarrow{\cc{F}} \cc{D} 
  \]
  extends to sheaf on $\Spec \cc{K} \simeq \Shv(\Rad(\cc{K})^{\vee})$ if and only if for every $\cc{L} \in \cc{C}$ and Zariski cover $\cc{L} \to \cc{L}_{1} \times \cc{L}_{2}$, the following square is Cartesian
  \[\xymatrix{
      \cc{F}(\cc{L}) \ar[r] \ar[d] & \cc{F}(\cc{L}_{1})\ar[d] \\
      \cc{F}(\cc{L}_{2}) \ar[r] & \cc{F}((\cc{L}_{1} \otimes_{\cc{L}}\cc{L}_{2})/\sqrt{0}).
    }
  \]
\end{proposition}

\begin{proof}
  By \cref{lem:qcmvsuffices} and the coherence of $\Rad(\cc{K})^{\vee}$, the claim that $\cc{F}$ extends to a sheaf on $\Spec \cc{K}$ is equivalent to showing that for every $\cc{I} = \sqrt{\cc{I}'\vee \cc{I}''} \in \Rad(\cc{K})^{\omega}$, the following square is Cartesian
  \begin{equation}\label{eq:almostmvnotquite}
    \begin{aligned}
    \xymatrix{
      \cc{F}(\cc{K}/\cc{I}' \wedge \cc{I}'') \ar[r] \ar[d] & \cc{F}(\cc{K}/ \cc{I}'')\ar[d] \\
      \cc{F}(\cc{K}/\cc{I}') \ar[r] & \cc{F}(\cc{K}/\cc{I}).
      }
    \end{aligned}
  \end{equation}
From the identification $\Rad(\cc{K})^{\omega}_{\sqrt{\cc{I}' \vee \cc{I}''}/ } \simeq \Rad(\cc{K}/ \sqrt{\cc{I}'\vee \cc{I}'})^{\omega}$ of \cref{lem:preimagesofidealsandradicals}, the map $\cc{K}/ \cc{I}' \wedge \cc{I}'' \to \cc{K}/ \cc{I}' \times \cc{K}/\cc{I}''$ is easily seen to be a Zariski cover of the source. Furthermore, one has \[\cc{K}/\sqrt{\cc{I}' \vee \cc{I}''} = (\cc{K}/\cc{I}' \wedge \cc{I}'')/\cc{I}'\vee \cc{I}'' = \cc{K}/\cc{I}' \otimes_{\cc{K}/\cc{I}'\wedge \cc{I}''} \cc{K}/\cc{I}'' \] since $\cc{K}/\cc{I}'\wedge \cc{I}''\colon \Idl(\cc{K}/\cc{I}') \to \twoCAlg_{\cc{K}/\cc{I}'\wedge \cc{I}''}$ is a left adjoint. Once again invoking \cref{lem:preimagesofidealsandradicals} supplies an identification $(\cc{K}/\cc{I}' \vee \cc{I}'')/\sqrt{0} \simeq \cc{K}/\sqrt{\cc{I}' \vee \cc{I}''}$. Replacing $\cc{K}$ with $\cc{L}$ and $\cc{K}/\cc{I}', \cc{K}/\cc{I}''$ with $\cc{L}_{1}, \cc{L}_{2}$ now yields both directions of the claim.
\end{proof}

\subsection{The structure sheaf and full faithfulness of the absolute spectrum}
 This subsection is dedicated to the proof of \cref{thmalph:sheaf}.

\begin{notation}  As before, we continue to identify a $\GBal$-structure $\cO$ with its associated $\twoCAlg$-valued sheaf, and abusively utilize the same notation to refer to either object. In light of this identification and \cref{lem:globalsectionsisglobalsections}, we write $\Gamma = \Gamma_{\GBal}$.
\end{notation}

\begin{theorem}\label{thm:rigid_subcanonical}
    For $\cc{K} \in \twoCAlgrig$, the structure presheaf $\widetilde{\cO}_{\cc{K}}$ of \cref{def:spec} satisfies \[\widetilde{\cO}_{\cc{K}} \in \Shv(\Spec \cc{K}; \twoCAlg) \subseteq \Fun((\Pro(\GBal)^{\ad}_{-/\cc{K}})^{\op}, \twoCAlg)\] i.e., sheafification supplies an equivalence $\widetilde{\cO}_{\cc{K}} \simarrow \cO_{\cc{K}}$. 
\end{theorem}

\begin{observation}\label{obs:identifyingOK}
    Recall that $\Pro(\GBal)^{\ad}_{/\cc{K}} \simeq (\Rad(\cc{K})^{\omega})^{\op}$ by \cref{lem:admissibleprincipal} and \cref{lem:rigid_radicalideal}. Under this identification, \cref{lem:identificationofGstructure} implies that the functor
\[ \widetilde{\cO}_{\cc{K}} \colon \Rad(\cc{K})^{\omega} \to \twoCAlg \] may be identified with the assignment $\cc{K}/- \colon \cc{I} \mapsto \cc{K}/\cc{I}$. 
\end{observation}
\begin{remark}
    Recall the equivalence $\Spec \cc{K} \simeq \Shv(\Spc \cc{K})$ of \cref{thm:2zariski_geometry}. Under this identification, $\widetilde{\cO}_{\cc{K}}$ supplies a presheaf on $\Spc \cc{K}$ which sends a quasicompact open subset  $U \subseteq |\! \Spec \cc{K}|$ associated to a principal ideal $\langle a \rangle$ to the object $\cc{K}/\langle a \rangle$. After passing to homotopy categories, this assignment partially recovers the ``structure presheaf'' considered in \eqref{eq:presheafofttcats}. 
\end{remark}

We will in fact show the following stronger statement, from which we will obtain the desired statement after passing to subcategories of compact objects.

\begin{theorem}\label{thm:indsheaf}
  For $\cc{K} \in \twoCAlgrig$, the following composite \[ \widetilde{\cO}^{\Ind}_{\cc{K}}\colon (\Pro(\GBal)_{-/\cc{K}}^{\ad})^{\op} \xrightarrow{\widetilde{\cO}_{\cc{K}}} \twoCAlg^{\rig} \xrightarrow{\Ind} \CAlg(\Catbig)\] extends to a $\CAlg(\Catbig)$-valued sheaf on $\Spec\cc{K}$. 
\end{theorem}

\begin{notation}
  In general, we will write $\mdef{\cO^{\Ind}_{\cc{K}}}$ to refer to the sheafification of the assignment $\widetilde{\cO}_{\cc{K}}$ above. Thus, the theorem above states that sheafification supplies an equivalence $\widetilde{\cO}_{\cc{K}}^{\Ind} \simarrow \cO^{\Ind}_{\cc{K}}$.
\end{notation}

\begin{recollection}\label{rec:compactgeneration}
Let $\cc{C} \in \PrLst$ be a presentable stable $\infty$-category. Recall that a \tdef{localizing subcategory} $\cc{E} \subseteq \cc{C}$ is a full stable subcategory which is closed under all small colimits in $\cc{C}$. 
\begin{enumerate}
    \item We say a set $\{X_\alpha\}_{\alpha \in A} \subseteq \cc{C}$ \tdef{generates} $\cc{C}$ if the smallest localizing subcategory $\cc{C}$ which contains all $X_{\alpha}$ is $\cc{C}$ itself.
    \item A generating collection $\{X_{\alpha}\}_{\alpha \in A}$ is said to \tdef{compactly generate} $\cc{C}$  if $X_{\alpha} \in \cc{C}^{\omega}$ for every $\alpha \in A$. 
\end{enumerate}
 Recall that by convention, we use the term ``compactly generated'' and ``$\omega$-presentable'' interchangeably; this is not an overloading of notation, as every compactly generated presentable stable $\infty$-category $\cc{C}$ is $\omega$-presentable; in fact, it is of the form $\Ind(\cc{C}^{\omega})$. In particular, the set of compact objects in a compactly generated presentable stable $\infty$-category itself forms a system of compact generators. By \cite[Lemma 7.6]{mathewNilpotenceDescentEquivariant2017}, condition (a) above is equivalent to the following. 
 \begin{enumerate}
     \item[(a')] A set $\{X_\alpha\}_{\alpha \in A} \subseteq \cc{C}$ \tdef{generates} $\cc{C}$ if for every object $Y \in \cc{C}$, $Y \simeq 0$ if and only if $\Hom(X_{\alpha}, Y) \simeq 0$ for every $\alpha \in A$.
 \end{enumerate}
 We will henceforth use the term ``generate'' to refer to either of the conditions (a) or (a') above without specification.
\end{recollection}

Before proceeding to the proof of our main theorem, we require the following ``big'' variant of \cref{lem:radicalgeneration} for rigid 2-rings.

\begin{lemma}\label{lem:intersectinglocalizingsubs}
  Given $\cc{K} \in \twoCAlgrig$ with $x_{1}, x_{2} \in \cc{K}$, there is an identification $\Ind(\langle x_{1} \rangle) \cap \Ind(\langle x_{2}\rangle) = \Ind(\langle x_{1}\otimes x_{2}\rangle)$. 
\end{lemma}

\begin{proof}
  The containment $\langle x_{1} \otimes x_{2} \rangle \subseteq \Ind(\langle x_{1} \rangle) \cap \Ind(\langle x_{2}\rangle)$ follows from \cref{lem:radicalgeneration}. Note that the subcategory $\Ind(\langle x_{1} \rangle) \cap \Ind(\langle x_{2}\rangle)$ is itself presentable stable as $\PrLst$ admits pullbacks, and moreover that its inclusion into $\cc{C}$ preserves colimits. We now claim that the collection $S = \{a_{1} \otimes a_{2} \mid a_{1} \in \langle x_{1}\rangle, a_{2} \in \langle x_{2} \rangle\}$ generates $\Ind(\langle x_{1} \rangle) \cap \Ind(\langle x_{2}\rangle)$ in the sense that the collection $\Hom(a_{1} \otimes a_{2} ,-) \colon \Ind(\langle x_{1} \rangle) \cap \Ind(\langle x_{2}\rangle) \to \Sp$ is jointly conservative as $a_{1} \otimes a_{2}$ ranges over $S$. To this end, let  $b \in \Ind(\langle x_{1} \rangle) \cap \Ind(\langle x_{2}\rangle)$ such that $\Hom(a_{1} \otimes a_{2},b) \simeq 0$ for every  $a_{1}\otimes a_{2} \in S$. As $b \in \Ind(\langle x_{1} \rangle)$ and $\langle x_{1} \rangle$ is closed under tensoring with any object of $\cc{K}$, the collection of formal filtered colimits of elements of $\langle x_{1}\rangle$ is closed under tensoring with objects of $\cc{K}$ and hence $a_{2}^{\vee} \otimes b \in \Ind(\langle x_{1} \rangle)$. The assumptions yield that \[\Hom(a_{1}, a_{2}^{\vee} \otimes b) \simeq 0 \implies a_{2}^{\vee} \otimes b \simeq 0\] for every $a_{2} \in \langle x_{2} \rangle$ and since $b \in \Ind(\langle x_{2} \rangle)$, a similar argument implies $b \simeq 0$. \cref{rec:compactgeneration} implies that $\Ind(\langle x_{1} \rangle) \cap \Ind(\langle x_{2}\rangle)$ is the closure under colimits in $\Ind(\cc{K})$ of the set $S$, which yields the claim.
\end{proof}

\begin{proof}[Proof of \cref{thm:indsheaf}]
  \cref{obs:identifyingOK} implies that $\widetilde{\cO}^{\Ind}_{\cc{K}}$ may be identified with the composite \[\Rad(\cc{K}) \xrightarrow{\cc{K}/-} \twoCAlgrig_{\cc{K}/} \xrightarrow{\Ind} \CAlg(\Catbig)\] and we now apply \cref{prop:sheavesonspec} in the case $\cc{C} = \twoCAlgrig$ and $\cc{F} = \Ind$ to reduce to the case below.

 Let $\cc{L}\in \twoCAlgrig$ arbitrary and $x, y \in \cc{L}$ such that $x \otimes y \simeq 0$. Write $L_{x}, L_{y} \in \mm{Idem}_{\cc{L}}^{\mm{fin}}$ for the smashing idempotents associated to $x, y$ by \cref{cor:fidemtoideal}. We claim that the following square is Cartesian
  \begin{equation} \label{eq:indsheafmvprelude}
 \begin{aligned} 
  \xymatrix{
      \unit \ar[r] \ar[d] & L_{x} \ar[d] \\
      L_{y}\ar[r] & L_{x} \otimes L_{y} \simeq L_{x \oplus y}.
    }
    \end{aligned}
  \end{equation}
  Passing to horizontal fibers yields a map $f\colon \fib(\unit \to L_{x}) \to L_{y} \otimes \fib(\unit \to L_{x})$ and it suffices to show that this is an equivalence. However, $\fib f$ is in the kernel of smashing with both $L_{x}, L_{y}$, and thus must be in $\Ind(\langle x \rangle) \cap \Ind(\langle y \rangle)$ by part (2) of \cref{prop:verdexists}. The assumptions on $x, y$ and \cref{lem:intersectinglocalizingsubs} imply that this is the zero subcategory, whence $f$ must have been an equivalence.

  We now demonstrate that the following square is Cartesian, which will yield the claim \begin{equation}\label{eq:indsheafmv}
  \begin{aligned}
  \xymatrix{
       \Ind(\cc{L}) \ar[r] \ar[d] & \Ind(\cc{L}/\langle y \rangle) \ar[d]\\
       \Ind(\cc{L}/\langle y \rangle) \ar[r]& \Ind(\cc{L}/\langle x \oplus y \rangle).
    }
    \end{aligned}
  \end{equation}

  By \cite[Theorem B]{Horev2017}, there is an adjunction of the form
  \[F \colon \Ind(\cc{L}) \to \Ind(\cc{L}/\langle x \rangle) \times_{\Ind(\cc{L}/\langle x \oplus y \rangle)} \Ind(\cc{L}/\langle y \rangle) \noloc G\] where $F$ is pointwise given by projecting an element $a\in \Ind(\cc{L})$ to its image in the corresponding localizations and $G$ is pointwise given by the formula
  \[
  (c_{x}, c_{y}, \alpha\colon L_{x \oplus y}c_{x} \simarrow L_{x \oplus y} \otimes c_{x}) \mapsto c_{x} \times_{L_{x \oplus y}c_{x}} c_{y}
  \]
  where $c_{x}$ is identified with their images in $\Ind(\cc{L})$ under the fully faithful right adjoint $\Ind(\cc{L}/\langle x \rangle) \hookrightarrow \Ind(\cc{K})$ and similarly for $c_{y}$, and latter pullback is computed in $\Ind(\cc{L})$. Note first that for every $a \in \Ind(\cc{L})$ the counit $a \to GF(a)$ is exactly the map $a \mapsto (L_{x}\otimes a)\times_{(L_{x \oplus y} \otimes a)} (L_{y} \otimes a)$ induced by tensoring the square of \eqref{eq:indsheafmvprelude} with $a$; since the cited square has already been shown to be Cartesian, this map is an equivalence, and thus $F$ is fully faithful.

  To demonstrate that $F \dashv G$ are inverse equivalences, it thus suffices to show that $G$ is conservative. To this end, let 
  \[h \colon (c_{x}, c_{y}, \alpha) \to (c_{x}', c_{y}', \alpha') \in (\Ind(\cc{L}/\langle x \rangle) \times_{\Ind(\cc{L}/\langle x \oplus y \rangle)} \Ind(\cc{L}/\langle y \rangle))^{[1]}\] an arbitrary map satisfying $G(h) \in \Ind(\cc{L})^{[1]}$ is an equivalence. Note that the counit associated to $F \dashv G$ projected to the first coordinated yields a map $L_{x} \otimes G(c_{x}, c_{y}, \alpha) \to L_{x} \otimes c_{x}$ which appears as the top horizontal arrow in following Cartesian square 
\begin{equation}\label{eq:indsheafrightadjoint}
 \begin{aligned} 
  \xymatrix{
      L_{x} \otimes G(c_{x}, c_{y}, \alpha) \ar[r] \ar[d] & L_{x} \otimes c_{x} \ar[d] \\
      L_{x} \otimes L_{y} \otimes c_{y}\ar[r] & L_{x} \otimes L_{y} \otimes c_{x} \simeq L_{x \oplus y} \otimes c_{x}.
    }
    \end{aligned}
  \end{equation}
 and since the bottom arrow is clearly an equivalence, the desired counit map must be as well. Hence, $L_{x} \otimes G(h) \colon c_{x} \to c_{x}'$ recovers the map $h$ projected to the coordinate $\Ind(\cc{L}/\langle x \rangle)$, regarded as a full subcategory of $\Ind(\cc{L})$; the symmetric claim holds with $y$ in the place of $x$. It follows that the map $FG(h)\colon (c_{x}, c_{y}, \alpha) \to (c_{x}', c_{y}', \alpha')$ agrees with $h$ on the first two coordinates via the counit map, and thus $h$ must itself have been an equivalence.
 \end{proof}
 
\begin{proof}[Proof of \cref{thm:rigid_subcanonical}]
  Note that the functor $(-)^{\mm{dbl}}: \CAlg(\Catbig) \to \CAlg(\Catbig)$ commutes with all limits. By \cref{lem:rigidlycompactlygenerated} we have the identification of presheaves \[(\widetilde{\cO}^{\Ind}_{\cc{K}})^{\mm{dbl}} \coloneq \Ind (\widetilde{\cO}_{\cc{K}})^{\mm{dbl}} \simeq \widetilde{\cO}_{\cc{K}} \in \Fun((\Pro(\GBal)^{\ad}_{/\cc{K}})^{\op}, \twoCAlg)\] implying that the latter object must also agree with its sheafification.
\end{proof}

\begin{corollary}\label{cor:rigidspec_fullyfaithful}
  The absolute spectrum functor $\Spec\colon \twoCAlgrig \to \LTop(\GBal)$ is fully faithful. 
\end{corollary}
\begin{proof}
  Let $\cc{K} \in \twoCAlgrig$ arbitrary. From \cref{thm:rigid_subcanonical} and \cref{obs:identifyingOK}, the canonical map $\widetilde{\cO}_{\cc{K}} \to \cO_{\cc{K}}$ induces an identification $\cc{K} \simarrow \Gamma(\Spec \cc{K}, \cO_{\cc{K}})$. This is the unit of the adjunction of \cref{thm:spec}, and so the left adjoint $\Spec$ is fully faithful on the full subcategory $\twoCAlgrig \subseteq \twoCAlg$.
\end{proof}

\begin{warning}\label{warning:nonrigid}
    The above statements may in general fail outside the rigid setting. As an example, consider the commutative 2-ring $\cc{K} \coloneq \Fun([1], \Sp)^{\omega}$ equipped with the pointwise multiplicative structure. It is shown in \cite[Theorem 4.3]{Aoki2023} that the maps $s,t \colon \cc{K} \to \Sp^{\omega} \times \Sp^{\omega}$ corresponding to source and target are Karoubi quotients inducing an equivalence $\Spc\cc{K} \simeq \Spc \Sp^{\omega} \sqcup \Spc \Sp^{\omega}$, although $\Fun([1], \Sp)^{\omega} \simeq \Fun([1], \Sp^{\omega}) \neq \Sp^{\omega} \times \Sp^{\omega}$ via the map $s \times t$. 
\end{warning}

Before concluding this section, we collect an example explaining the failure of the presheaf of \eqref{eq:presheafofttcats} to satisfy gluing of objects outside of the idempotent complete setting, following \cref{rem:puntingreaders}. In particular, \cref{thm:rigid_subcanonical} cannot be expected to work if one defines the Zariski geometry on symmetric monoidal stable $\infty$-categories without the assumption of idempotent completeness.
    
     \begin{example}\label{ex:whyidempotentcomplete}
In this example we freely make use of tt-geometric terminology. Consider the tt-category $\cc{K}\coloneq \Ho \Perf_{k}$ for $k$ a field, and consider the full triangulated subcategory of $\cc{K} \times \cc{K}$ generated under cofibers and shifts by the objects $(k,k), (k^{\times 2}, 0), (0, k^{\times 2})$, denoted by $\cc{K}'$. It is easy to see that every object of $\cc{K}\times \cc{K}$ may be obtained as a retract of an object in $\cc{K}'$; namely, $\cc{K}' \subset \cc{K} \times \cc{K}$ is a \emph{dense} or \emph{\'epaisse} subcategory of $\cc{K}\times \cc{K}$. The induced map from $ K_{0}(\cc{K'})\to \cc{K}_{0}(\cc{K} \times \cc{K})$ is injective by a theorem of Thomason \cite[Theorem A.3.2]{calmesHermitianKtheoryStable2025} and corresponds exactly to the proper subgroup of $K_{0}(\cc{K} \times \cc{K}) \simeq \bb{Z} \oplus \bb{Z}$ generated by the elements $(1,1), (2,0), (0,2)$; the characterization of $\cc{K}'$ as exactly those objects with $K_0$ classes in this subgroup also shows that the tensor product on $\cc{K} \times \cc{K}$ restricts to one on $\cc{K}'$, since the specified subgroup is also a subring of $\bb{Z} \times \bb{Z}$ with the ring structure induced by the tensor product on $\cc{K} \times \cc{K}$. Now, it is easy to show that $\Spc \cc{K}' \cong \ast \coprod \ast$, induced by the maps $\pi_{1}, \pi_{2}$ from $\cc{K}'$ to $\cc{K}$ sending an object $(x,y)$ to its first or second component respectively. It is also easy to show that these maps are Verdier localizations away from the thick $\otimes$-ideals generated by $(k^{\times 2},0)$ and $(0, k^{\times 2})$. Together, these form an open cover $\Spc \cc{K}/\langle (k^{\times 2}, 0)\rangle \coprod \Spc \cc{K}/\langle (0,k^{\times 2})\rangle \twoheadrightarrow \Spc \cc{K}'$ whose associated restriction map
    \begin{equation}
      \widetilde{\cc{O}}_{\cc{K}'}(\Spc \cc{K}') \to \widetilde{\cc{O}}_{\cc{K}'}(\Spc \cc{K}/\langle (k^{\times 2}, 0)\rangle) \times \widetilde{\cc{O}}_{\cc{K}'}(\Spc \cc{K}/\langle (0,k^{\times 2})\rangle)
    \end{equation}
    may be identified with the fully faithful inclusion $\cc{K}' \to \cc{K} \times \cc{K}$. Since this latter map fails to be an equivalence on $K_{0}$, it cannot be essentially surjective; thus, the presheaf on $\Spc \cc{K}'$ defined by sending an open subset to its corresponding localization of $\cc{K}'$ cannot satisfy gluing for objects. 
  \end{example}

\subsection{Zariski descent for modules}\label{ssec:ZariskidescentMod}
In this section we will prove \cref{thmalph:mod}. Before commencing the proof, we will need the following observation on the behaviour of Karoubi quotients on modules.

\begin{notation}
  We write $\mdef{\twoModbig}\colon \twoCAlg \to \CAlg(\PrL)$ to denote the composite
  \[
    (\Mod \circ \Ind) \colon \twoCAlg \to \CAlg(\PrL) \to \CAlg(\Catbig)
    \]
  which sends a 2-ring $\cc{K} \mapsto \Mod_{\Ind(\cc{K})}(\PrL)$.
\end{notation}

\begin{lemma}\label{lem:tensormoduleswithlocs}
    For any $2$-ring $\cc{K}$, module $\cc{M} \in \twoModbig_{\cc{K}}$, and ideal $\cc{J} \subseteq \cc{K}$, the map \[\cc{M} \to  \Ind(\cc{K}/\cc{J})\otimes_{\Ind(\cc{K})} \cc{M}\] is identified with the (large) Karoubi quotient away from the closure under small colimits of the essential image of $\Ind(\cc{J}) \times \cc{M} \to   \Ind(\cc{K}) \otimes_{\Ind(\cc{K})}\cc{M} \simeq \cc{M}$.
  \end{lemma}
  \begin{proof}

    Let $\cc{J}_{\cc{M}}$ denote the closure under small colimits of the image of $\Ind(\cc{J}) \times \cc{M} \to \cc{M}$; note that since $\Ind(\cc{J})$ is a thick tensor-ideal, $\cc{J}_{\cc{M}}$ is closed under $\Ind(\cc{K})$-tensors in $\cc{M}$. We first show that $\cc{M}/\cc{J}_{\cc{M}}$ naturally admits the structure of an $\Ind(\cc{K}/\cc{J})$-module in $\CAlg(\PrL)$; for this, consider an arbitrary map \[h = (h_{1},\dotsc,h_{n}, h_{\cc{M}})\colon (a_{1},\dotsc,a_{n}, m) \to (a'_{1},\dotsc,a'_{n}, m') \in (\Ind(\cc{K})^{\times n} \times \cc{M})^{[1]}\] such that $\cofib(h) = (\cofib(h_{1}),\dotsc,\cofib(h_{n}), \cofib(h_{\cc{M}}))$ satisfies either $\cofib(h_{i}) \in \Ind(\cc{J})$ for some $i$ or $\cofib(h_{\cc{M}}) \in \cc{J}_{\cc{M}}$. Then $\cofib(h_{1} \otimes \dotsb \otimes h_{n} \otimes h_{\cc{M}}) \in \cc{J}_{\cc{M}}$; in the example that $\cofib(h_{i}) \in \Ind(\cc{J})$, we may express this as the cofiber of the composite
    \[
      \cofib\left((h_{1}\otimes \dotsb \otimes \id\otimes \dotsb \otimes h_{n}, h_{\cc{M}}) \circ (\id\otimes \dotsb \otimes h_{i}\otimes \dotsb \otimes \id) \right)
    \] whence it may be represented as the cofiber of the map \[(h_{1}\otimes \dots \otimes \id\otimes \dots \otimes h_{n}, h_{\cc{M}}) \colon \cofib(\id\otimes \dotsb \otimes h_{i}\otimes \dotsb \otimes \id) \to \cofib(\id\otimes \dotsb \otimes h_{i}\otimes \dotsb \otimes \id)\]
    by expressing the desired cofiber as an iterated pushout\footnote{also known as the \emph{octahedral axiom}.}. Since $\cofib(h_{i}) \in \Ind(\cc{J})$ by assumption, $\cofib(\id\otimes \dotsb \otimes h_{i}\otimes \dotsb \otimes \id) \in \cc{J}_{\cc{M}}$ and thus $\cofib(h_{1} \otimes \dotsb \otimes h_{n} \otimes h_{\cc{M}}) \in \cc{J}_{\cc{M}}$. The same proof works in the case that $h_{i}$ is replaced with $h_{\cc{M}}$.

    We have just shown that the localization maps $\Ind(\cc{K}) \to \Ind(\cc{K}/\cc{I})$, $\cc{M} \to \cc{M}/\cc{J}_{\cc{M}}$ are \emph{compatible with the $\mm{LMod}^{\otimes}$-structure} in the sense of \cite[Proposition 2.2.1.6]{LurieHA}, and thus \cite[Proposition 2.2.1.9]{LurieHA} supplies a unique left $\Ind(\cc{K}/\cc{J})$-module structure on $\cc{M}/\cc{J}_{\cc{M}}$ such that the localization map $\cc{M} \to \cc{M}/\cc{J_{M}}$ is strongly left $\Ind(\cc{K})$-linear. Furthermore, the $\Ind(\cc{K}/\cc{J})$-module structure on $\cc{M}/\cc{J}_{\cc{M}}$ is a localization of the $\Ind(\cc{K})$-module structure on $\cc{M}$ and therefore preserves colimits in every variable; hence, $\cc{M} \in \Mod_{\Ind(\cc{K}/\cc{J})}(\PrL)$.

    By this last observation, there is an induced $\Ind(\cc{K}/\cc{J})$-linear functor $\Ind(\cc{K}/\cc{J}) \otimes_{\Ind(\cc{K})} \cc{M} \to \cc{M}/\cc{J}_{\cc{M}}$. Note furthermore that the unit $\cc{M} \to \Ind(\cc{K}/\cc{J})\otimes_{\Ind(\cc{K})} \cc{M}$ sends the image of $\Ind(\cc{J}) \times \cc{M}$ under the action map on $\cc{M}$ to $0$; as this is a left adjoint functor of presentable $\Ind(\cc{K})$ modules, it must send the closure of this image under colimits to 0, i.e., $\cc{J}_{\cc{M}} \mapsto 0$. This supplies an induced map $\cc{M}/\cc{J}_{\cc{M}} \to \Ind(\cc{K}/\cc{J})\otimes_{\Ind(\cc{K})}\cc{M}$ by the universal property of Karoubi quotients. The composite $\cc{M}/ \cc{J}_{\cc{M}} \to \Ind(\cc{K}/\cc{J}) \otimes_{\Ind(\cc{K})} \cc{M} \to \cc{M}/\cc{J}_{\cc{M}}$ agrees with the map induced by the Karoubi quotient $\cc{M} \to \cc{M}/\cc{J}_{\cc{M}}$ by construction, which is the identity. Similarly, the map $\cc{M}/\cc{J}_{\cc{M}} \to \Ind(\cc{K}/\cc{J}) \otimes_{\Ind(\cc{K})} \cc{M}$ is induced from the unit map $\cc{M} \to \Ind(\cc{K}/\cc{J})\otimes_{\Ind(\cc{K})}\cc{M}$, and hence the composite
    \[
      \Ind(\cc{K}/\cc{J})\otimes_{\Ind(\cc{K})}\cc{M} \to \cc{M}/\cc{J}_{\cc{M}} \to \Ind(\cc{K}/\cc{J}) \otimes_{\Ind\cc{K}} \cc{M}
    \]
    may be identified with the $\Ind(\cc{K}/\cc{J})$-linear map induced from the unit map $\cc{M} \to \Ind(\cc{K}/\cc{J}) \otimes_{\Ind(\cc{K})} \cc{M}$, which is also the identity, yielding the desired inverse equivalence. 
  \end{proof}

\begin{corollary}
  In the situation of \cref{lem:tensormoduleswithlocs}, the unit map $\cc{M} \to \Ind(\cc{K}/\cc{J}) \otimes_{\Ind(\cc{K})} \cc{M}$ is a localization in the sense of \cite[Definition 5.2.7.2]{LurieHA}, i.e., it has a fully faithful right adjoint.
\end{corollary}
\begin{proof}
    We enlarge universes so that $\Ind(\cc{M})$ is a presentable $\infty$-category in a larger universe. By \cref{lem:tensormoduleswithlocs} and \cref{prop:verdexists}, the induced functor $\Ind(\cc{K}/\cc{J}) \otimes_{\Ind(\cc{K})} \cc{M} \to \Ind(\cc{M})$ is fully faithful. Since the unit map $\Ind(\cc{M}) \to \Ind(\cc{K}/\cc{J}) \otimes_{\Ind(\cc{K})}\cc{M}$ was itself a left adjoint, the embedding above may uniquely be identified with the composite 
    \[
    \Ind(\cc{M}) \to \Ind(\cc{K}/\cc{J}) \otimes_{\Ind(\cc{K})}\cc{M} \to \cc{M} \xrightarrow{\Yo} \Ind(\cc{M})
    \]
    where the first arrow is the right adjoint to the unit map, yielding the desired claim.
\end{proof}
  
  We are now ready to prove \hyperref[thmalph:moda]{\cref*{thmalph:mod}\,(a)}.

\begin{theorem}\label{thm:2MODisasheaf} Let $\cc{K} \in \twoCAlgrig$.
  \begin{enumerate}
  \item For any $\cc{M} \in \twoModbig_{\cc{K}}$, the assignment \[\cc{O_{M}} \colon U \mapsto \cO^{\Ind}_{\cc{K}}(U) \otimes_{\Ind(\cc{K})} \cc{M}\] for $U \subseteq |\! \Spec \cc{K}|$ a quasicompact open subset extends to a $\twoModbig_{\cc{K}}$-valued sheaf on $\Spec \cc{K}$.
  \item The assignment which sends a quasicompact open subset $U \subseteq |\! \Spec \cc{K}|$ to $\twoModbig_{\cc{O}_{\cc{K}}(U)}$ extends to a $\CAlg(\Catbig)$-valued sheaf on $\Spec \cc{K}$.
  \end{enumerate}
\end{theorem}
\begin{proof}
 Given $\cc{K} \in \twoCAlg$ arbitrary, $\cc{J}\subseteq \cc{K}$ a thick tensor ideal, and $\cc{M} \in \twoModbig_{\cc{K}}$ we write $\cc{M}/\cc{J}$ to denote $\Ind(\cc{K}/ \cc{J})\otimes_{\Ind(\cc{K})} \cc{M}$, and $\cc{M}_{\cc{J}}$ to denote the kernel of the map $\cc{M} \to \cc{M}/\cc{J}$.

    Let us now prove part (a). As in the proof of \cref{thm:indsheaf}, we use the reduction schema of \cref{prop:sheavesonspec} for $\twoCAlgrig$. Let $\cc{L} \in \twoCAlgrig$, $\cc{L} \to \cc{L}_{1} \times \cc{L}_{2}$ a Zariski cover, where $\cc{L} \to \cc{L}_{i}$ has kernel $\cc{I}_{i}$, and $\cc{M} \in \twoModbig_{\cc{L}}$. We must show that the following square is Cartesian:
\begin{equation}
\begin{aligned}
\label{eq:modulemv}\xymatrix{
     \cc{M} \ar[r] \ar[d] & \cc{M}/\cc{I}_{1} \ar[d]\\
      \cc{M}/\cc{I}_{2} \ar[r]& \cc{M}/\langle \cc{I}_{1}, \cc{I}_{2}\rangle.
    }
    \end{aligned}
  \end{equation}
  Combining \cref{prop:finitelocissmashing} and \cite[Proposition 4.1]{ben-zviIntegralTransformsDrinfeld2010a}, the localization $\cc{M} \to \cc{M}/\cc{I}_{i}$ may be identified as the left adjoint in the adjunction
  \[
  L_{i} \otimes - \colon \cc{M} \rightleftarrows \Mod_{L_{i}}(\cc{M}) \noloc \mm{fgt} 
  \]
  where $L_{i}\in \Idem_{\cc{L}}^{\mm{fin}}$ is the finite idempotent associated to the ideal $\cc{I}_{i}$, $i = 1,2$. There is thus an induced adjunction of the form 
  \[
  F \colon \cc{M} \rightleftarrows \cc{M}/\cc{I}_{1} \times_{\cc{M}/\langle \cc{I}_{1}, \cc{I}_{2}\rangle} \cc{M}/\cc{I}_{2} \noloc G
  \]
  by \cite[Theorem B]{Horev2017}, where $F$ is pointwise given by sending an element $m \in \cc{M}$ to its corresponding images in $\cc{M}/\cc{J_{i}}$ under the localization maps, and $G$ is pointwise given by the formula 
  \[
  (m_{1}, m_{2}, \alpha\colon L_{1} \otimes L_{2} \otimes m_{1} \simarrow L_{1} \otimes L_{2} \otimes m_{2}) \mapsto m_{1} \times_{L_{1} \otimes L_{2} \otimes m_{1}} m_{2} 
  \]
  where $m_{1}, m_{2}$ are naturally regarded as elements of $\cc{M}$ under the fully faithful right adjoints $\cc{M}/\langle \cc{I}_{i}\rangle \to \cc{M}$, and the pullback on the right is computed in $\cc{M}$. The rest of the proof proceeds exactly as in the proof of \cref{thm:indsheaf}.
  
  We may now turn to the proof of part (b), for which we employ the same reduction schema and maintain the same notation as in the previous proof. We are tasked with showing that the map
  \begin{equation}\label{eq:twomodpb}
    \twoModbig_{\cc{L}} \to \twoModbig_{\cc{L}_{1}} \times_{\twoModbig_{\cc{L}_{1} \otimes_{\cc{L}} \cc{L}_{2}}}\twoModbig_{\cc{L}_{2}}
  \end{equation} induced by base-change along $\cc{L} \to \cc{L}_{1} \times \cc{L}_{2}$ is an equivalence. Once again invoking \cite[Theorem B]{Horev2017} the right adjoint to the functor in \eqref{eq:twomodpb} is pointwise given by the assignment
  \begin{equation}\label{eq:twomodpb2}
    (\cc{M}_{1} , \cc{M}_{2}, \alpha\colon \cc{M}_{1}/\cc{I}_{2} \simarrow \cc{M}_{2}/\cc{I}_{1}) \mapsto \cc{M}_{1} \times_{\cc{M}_{1}/\cc{I}_{2}} \cc{M}_{2} 
  \end{equation}
  where the right-hand side pullback is computed in $\twoModbig_{\cc{L}}$. Part (a) now implies that the left adjoint functor in \eqref{eq:twomodpb} is fully faithful, since for any $\cc{M} \in \twoModbig_{\cc{L}}$ the counit $\cc{M} \to \cc{M}/\cc{I}_{1} \times_{\cc{M}/\langle \cc{I}_{1}, \cc{I}_{2}\rangle} \cc{M}/\cc{I}_{2}$ is an equivalence. It will thus suffice to show that the right adjoint of \eqref{eq:twomodpb2} is conservative.

  To this end, consider $(\cc{M}_{1}, \cc{M}_{2}, \alpha)$ an arbitrary element in the right-hand side of \eqref{eq:twomodpb}. By \cite[Proposition A.1.18, Theorem A.3.12]{calmesHermitianKtheoryStable2025}, all of the arrows in the Cartesian square below are (large) Karoubi quotients, since the bottom horizontal and right-hand vertical arrows have already been shown to be Karoubi quotients
  \[\xymatrix{
      \cc{M}_{1} \times_{\cc{M}_{1}/\cc{I}_{2}}\cc{M}_{2} \ar[r]^-{\phi_{1}} \ar[d]_{\phi_{2}} &  \cc{M}_{1} \ar[d]^{\psi_{1}}\\
      \cc{M}_{2} \ar[r]^{\psi_{2}} & \cc{M}_{1}/\cc{I}_{2}. 
    }
  \]

  We claim that the map $\phi_{1}$ in the square above may be identified with the localization $\cc{M}_{1} \times_{\cc{M}_{1}/\cc{I}_{2}} \cc{M}_{2} \to (\cc{M}_{1} \times_{\cc{M}_{1}/\cc{I}_{2}} \cc{M}_{2})/\cc{I}_{1}$; the analogous claim for the $\phi_{2}$ follows by symmetry. Given that $\phi_{1}$ is already known to be a localization, we are left to show that $\ker \phi_{1}$ is generated under colimits by objects of the form $i \otimes m$ for $i \in \Ind(\cc{I}_{1}), m \in \cc{M}$ arbitrary. By assumption, $\cc{M}_{1}$ is restricted from an $\Ind(\cc{L}/\cc{I}_{1})$ module and hence $i \otimes m \in \ker \phi_{1}$ for every $i,m$ as aforementioned. Since the square above is Cartesian and $\psi_{2}$ is a Karoubi quotient, the map $\phi_{2}$ restricts to an equivalence $ \phi_{2}\colon \ker \phi_{1} \simarrow \ker \psi_{1}$. As $\phi_{2}$ is an $\Ind(\cc{L})$-linear left adjoint, it follows that the $\ker\phi_{1} \subseteq \cc{M}_{1} \times_{\cc{M}_{1}/\cc{I}_{2}} \cc{M}_{2}$ is the smallest subcategory which is closed under colimits and contains $i \otimes m_{2}$ for every $i \in \Ind(\cc{I}_{1})$ and $m_{2}$ in the image of $\cc{M}_{2} \to \cc{M}$.

  Thus, given any map of triples \[(\delta_{1}, \delta_{2}, \epsilon) \colon (\cc{M}_{1}, \cc{M}_{2}, \alpha) \to (\cc{M}'_{1}, \cc{M}'_{2}, \alpha') \in (\twoModbig_{\cc{L}_{1}} \times_{\twoModbig_{\cc{L}_{1} \otimes_{\cc{L}} \cc{L}_{2}}}\twoModbig_{\cc{L}_{2}})^{[1]}\] such that the induced map $\delta_{1} \times \delta_{2} \colon \cc{M}_{1} \times_{\cc{M}_{1}/\cc{I}_{2}} \cc{M}_{2} \to \cc{M}'_{1} \times_{\cc{M}'_{1}/\cc{I}_{2}} \cc{M}'_{2}$ is an equivalence in $\twoModbig_{\cc{L}}$, the induced map $\delta_{1}\colon \cc{M}_{1} \to \cc{M}_{1}'$ may be identified with $\Ind(\cc{L}/\cc{I}_{1}) \otimes_{\Ind(\cc{L})} (\delta_{1} \times \delta_{2})$ and is thus an equivalence, symmetrically for $\delta_{2}$, implying that the right adjoint of \eqref{eq:twomodpb2} is conservative and hence an inverse equivalence. 
\end{proof}

\begin{remark}
    It is possible to prove \cref{thm:indsheaf} and \cref{thm:2MODisasheaf} by noting that an intermediate claim in the proof of the former result already implies that $L_{x} \times L_{y}$ is \emph{descendable} in the sense of \cite{Mathew16}. An earlier iteration of the results above proceeded through this route, which is a shorter path, albeit one which assumes more. 
\end{remark}

We now approach the second and third parts of \cref{thmalph:mod}, for which we will require a ``Zariski local-to-global principle'' for compact generation. Apart from some modifications to account for the generality, the proof proceeds in the same way as in \cite[Theorem 6.11]{antieauBrauerGroupsEtale2014}. We first collect the following useful lemma.

\begin{lemma}\label{lem:the-one-about-compacts}
  Let $\cc{K} \in \twoCAlgrig$, and $\cc{M} \in \twoModbig_{\cc{K}}$ any module over $\Ind(\cc{K})$.
  \begin{enumerate}
  \item Given any map $\cc{K} \to \cc{L} \in \twoCAlgrig$, the induced unit map $\cc{M} \to \Ind(\cc{L}) \otimes_{\Ind(\cc{K})}\cc{M}$ preserves compact objects.
  \item If $\cc{M}$ is compactly generated, then the $\Ind(\cc{K})$-module structure on $\cc{M}$ is induced from a $\cc{K}$-module structure on its subcategory of compact objects. In other words, $\cc{M}$ is in the essential image of the functor $\Ind \colon \Mod_{\cc{K}}(\Catperf) \to \twoModbig_{\cc{K}}$.
  \end{enumerate}
\end{lemma}
\begin{proof}
We first claim that given any object $a \in \cc{K}$ and $m \in \cc{M}^{\omega}$, the object $a \otimes m \in \cc{M}^{\omega}$. Note first that the duality datum on $a$ furnishes an adjunction $(a \otimes -) \dashv (a^{\vee}\otimes - )$ of endofunctors of $\cc{M}$. Given an arbitrary filtered system $F\colon I \to \cc{M}$, the left-hand arrow in the following diagram is an equivalence
    \[\xymatrix{
        \Map_{\cc{M}}(a \otimes m, \varinjlim_{I}F_{i}) \ar[r]^{\sim} \ar[d] & \Map_{\cc{M}}(m, a^{\vee} \otimes \varinjlim_{I}F_{i}) \ar[d] \\
        \varinjlim_{I} \Map_{\cc{M}}(a \otimes m, F_{i}) \ar[r]^{\sim} &  \varinjlim_{I} \Map_{\cc{M}}(m, a^{\vee} \otimes F_{i})} 
      \]
      since the right-hand vertical arrow is an equivalence by the compactness of $m$. Thus, $a \otimes m$ must also have been compact. 
      
      From the above, it follows that the action map $\Ind(\cc{K}) \otimes \cc{M} \to \cc{M}$ lifted to a morphism in $\PrLstomega$, whence $\cc{M} \in \Mod_{\Ind(\cc{K})}(\PrLstomega)$. The symmetric monoidal equivalence of \cref{rem:bigandsmall} furnishes an induced symmetric monoidal equivalence \[\Ind\colon \Mod_{\cc{K}}(\Catperf) \simeq \Mod_{\Ind(\cc{K})}(\PrLstomega)\] and part (b) follows.

      For part (a), note that the symmetric monoidal equivalence above supplies an equivalence $\Ind(\cc{L} \otimes_{\cc{K}} \cc{M}^{\omega}) \simarrow \Ind(\cc{L}) \otimes_{\Ind(\cc{K})} \cc{M}$ via the counit map, and furthermore an identification of the unit map $\cc{M} \to \Ind(\cc{L} \otimes_{\cc{K}} \cc{M}^{\omega})$ with the application of $\Ind$ to the unit map $\cc{M}^{\omega} \to \cc{L}\otimes_{\cc{K}} \cc{M}^{\omega}$, yielding the result. 
\end{proof}

\begin{proposition}\label{prop:zariskilocaltoglobal}
  Let $\cc{K} \in \twoCAlg^{\rig}$, and $\cc{M} \in \twoModbig_{\cc{K}}$. Suppose that there exists a Zariski cover $\coprod_{I}\Spec \cc{K}_{i} \twoheadrightarrow \Spec \cc{K}$ such that $\prod_{I}\cc{O_{M}}(\Spec \cc{K}_{i})$ admits a compact generator. Then $\cc{M}$ admits a compact generator. 
\end{proposition}
\begin{proof}
  Note that the case where $|I| \leq 1$ is trivial. By induction, it suffices to demonstrate the claim in the case $|I| = 2$. To this end, let the open cover $\cc{K} \to \cc{K}_{1} \times \cc{K}_{2}$ be given by Karoubi quotients away from principal ideals $\cc{I}_{1}, \cc{I}_{2}$.  In the notation of the proof of \cref{thm:2MODisasheaf}, the demonstrated theorem yields a Cartesian square of the form
    \begin{equation}\label{eq:patchingsquare}
    \begin{aligned}
    \xymatrix{
        \cc{M} \ar[r] \ar[d] & \cc{M}/\cc{I}_{2} \ar[d] \\
        \cc{M}/\cc{I}_{1} \ar[r] & \cc{M}/\langle \cc{I}_{1}, \cc{I}_{2}\rangle
      }
    \end{aligned}
    \end{equation}
    where $\cc{M}/\cc{I}_{1},\ \cc{M}/\cc{I}_{2}$ admit compact generators $P_{1}, P_{2}$ by assumption, and the image of $P_{1}$ generates $\cc{M}/\langle \cc{I}_{1}, \cc{I}_{2}\rangle$ since the bottom horizontal arrow is a localization.
    
    By \cref{lem:the-one-about-compacts}, the right hand vertical map and the bottom horizontal maps above preserve compact objects. Using this, we claim that any object $M\in \cc{M}$ which restricts to compact objects $M_{i} \in (\cc{M}/\cc{I}_{i})^{\omega}$ for $i = 1,2$ must itself be in $\cc{M}^{\omega}$. To see this, note that \eqref{eq:patchingsquare} yields a Cartesian square
    \[\xymatrix{
        \Map_{\cc{M}}(M, M') \ar[r] \ar[d] &  \Map_{\cc{M}/\cc{I}_{1}}(M_{1}, M_{1}') \ar[d] \\
        \Map_{\cc{M}/\cc{I}_{2}}(M_{2}, M_{2}') \ar[r] & \Map_{\cc{M}/\langle \cc{I}_{1}, \cc{I}_{2}\rangle}(M_{12}, M_{12}')
      }
      \] for every object $M^{\prime} \in \cc{M}$, which is moreover natural in the choice of $M'$. The left-exactness of filtered colimits implies that $\Map_{\cc{M}}(M, -)$ commutes with filtered colimits if all other vertices of the square do so. Since the right-hand vertical map in \eqref{eq:patchingsquare} preserves compactness, $M_{12}$ is compact and hence $M$ is compact.
      
      Now note that $\cc{I}_{1}, \cc{I}_{2}$ are assumed to be principal, generated by elements $a_{1}, a_{2}$ and hence $\Ind(\cc{I}_{1}), \Ind(\cc{I}_{2})$ admit compact generators $a_{1}, a_{2}$. Thus, the image of $\cc{I}_{1} \times \cc{M}/\cc{I}_{2} \to \cc{M}/\cc{I}_{2}$ must be generated by the object $a_{1} \otimes P_{2}$, and the symmetric fact holds in the case where $1,2$ are swapped. With this in mind, we may proceed to construct a compact generator of $\cc{M}$.  Consider the following steps.

      \begin{enumerate}
      \item  Passing to vertical fibers in the Cartesian square \eqref{eq:patchingsquare} we see that the object $a_{2} \otimes P_{2}\in (\cc{M}/\cc{I}_{2})_{\cc{I}_{1}}$ lifts to a generator $Q$ of the subcategory $\cc{M}_{\cc{I}_{1}} \subseteq \cc{M}$. $Q \in \cc{M}^{\omega}$ as it has compact image in $\cc{M}/\cc{I}_{2}$ and image $0$ (hence compact) in $\cc{M}/\cc{I}_{1}$. 
      \item  By a theorem of Thomason \cite[Theorem A.3.2]{calmesHermitianKtheoryStable2025}, the image of  $P_{1} \oplus \Sigma P_{1}$ in $\cc{M}/\langle \cc{I}_{1}, \cc{I}_{2}\rangle$ is in the image of the localization map from $(\cc{M}/\cc{I}_{2})^{\omega}$, as it has a trivial class in $K_{0}$. Let $P_{1}' \in \cc{M}/\cc{I}_{2}$ be some choice of lift of this object. The objects $P_{1} \oplus \Sigma P_{1}, P'_{1}$ have equivalent images in $(\cc{M}/\langle \cc{I}_{1}, \cc{I}_{2}\rangle)^{\omega}$ by construction, and the Cartesian square of \eqref{eq:patchingsquare} supplies an object $P \in \cc{M}^{\omega}$ which lifts the objects $P_{1} \oplus \Sigma P_{1}, P'_{1}$.
      \end{enumerate}

We claim that the object $P \oplus Q \in \cc{M}^{\omega}$ generates $\cc{M}$. Let $N \in \cc{M}$ be any object such that $\Map_{\cc{M}}(P \oplus Q, N) = 0$. Then on the one hand, $\Map_{\cc{M}/\langle{I}_{1}}(P_{1} \oplus \Sigma P_{1}, N) = 0$ and thus $N$ has trivial restriction in $\cc{M}/\cc{I}_{1}$, implying that $N \subseteq \cc{M}_{\cc{I}_{1}}$. As $Q$ compactly generates the latter, $\Map_{\cc{M}_{\cc{I}_{1}}}(Q, N) = 0$ implies $N = 0$.
\end{proof}

We now turn to the proof of \hyperref[thmalph:modb]{\cref*{thmalph:mod}\,(b)}.
We first define the principal object of concern. 

\begin{construction}\label{xf0qhf}
  Recall the co-Cartesian fibration
  ${\Mod(\PrLst)}^{\otimes}\to \CAlg(\PrLst) \times
  \Fin_{\ast}$ of~\cite[Theorem~4.5.3.1]{LurieHA},
  which is classified by the
  functor
  \begin{equation*}
    {\Mod}\in \Fun(\CAlg(\PrLst) \times \Fin_{\ast}, \Catbig) \simeq
    \Fun(\CAlg(\PrL), \Fun(\Fin_{\ast}, \Catbig))
  \end{equation*}
  sending a presentably symmetric monoidal category $\cc{C}$
  to the commutative monoid object $\Mod_{\cc{C}}(\PrLst)$.
  Let $\mdef{\Mod^{\mm{cg}}(\PrLst)^{\otimes}}\subseteq \Mod(\PrLst)^{\otimes}$
  denote the full subcategory spanned by objects $(\cc{C}, \cc{M}_{1},\dotsc,\cc{M}_{n})$
  (\(\cc{C}\in\CAlg(\PrLst)\) and \(\cc{M}_{1}\), \dots, \(\cc{M}_{n}\in\Mod_{\cc{C}}(\PrLst)\))
  such that
  \(\cc{C}\) is rigidly compactly generated
  and the underlying \(\infty\)-categories of~\(\cc{M}_{1}\), \dots,~\(\cc{M}_{n}\)
  are compactly generated.
\end{construction}

\begin{remark}\label{Z-1miv}
  In the definition of \(\Mod^{\mm{cg}}(\PrLst)\) in \cref{xf0qhf},
  the rigidity assumption is important;
  if we just impose \(\cc{C}\in\CAlg(\PrLstomega)\),
  \cref{xwpbfz} below does not hold.
\end{remark}

\begin{lemma}\label{xwpbfz}
  The induced functor
  \begin{equation*}
    {\Mod^{\mm{cg}}(\PrLst)}^{\otimes}
    \to{\CAlg(\PrLstomega)}^{\rig} \times \Fin_{\ast}
  \end{equation*}
  is again a co-Cartesian fibration.
\end{lemma}

\begin{proof}
  We prove that it is full on co-Cartesian morphisms.
  We consider an object \((\cc{C},\cc{M}_{1},\dotsc,\cc{M}_{m})\in \Mod^{\mm{cg}}(\PrLst)\),
  a morphism \(F\colon\cc{C}\to\cc{D}\) in \(\CAlg(\PrLstomega)^{\rig}\),
  and a morphism \(f\colon\langle m\rangle\to\langle n\rangle\) in \(\Fin_{\ast}\).
  We then have to show that the co-Cartesian lift in \(\Mod(\PrLst)^{\otimes}\),
  which exists by~\cite[Theorem~4.5.3.1]{LurieHA},
  lies in \(\Mod^{\mm{cg}}(\PrLst)^{\otimes}\).
  By unpacking the definition,
  we must show that for each \(i\in\langle n\rangle^{\circ}\),
  \begin{equation*}
    \cc{D}\otimes_{\cc{C}}\biggl(\bigotimes_{j\in f^{-1}(i),\cc{C}}\cc{M}_{j}\biggr)
  \end{equation*}
  is compactly generated.
  We only have to consider the following cases,
  where \(i\) is the unique element in \(\langle1\rangle^{\circ}\).
  \begin{enumerate}
    \item
      \(F\) is an equivalence and
      \(f\) is the active morphism \(\langle0\rangle\to\langle1\rangle\).
    \item
      \(F\) is an equivalence and
      \(f\) is the active morphism \(\langle2\rangle\to\langle1\rangle\).
    \item
      \(f\) is the identity \(\langle1\rangle\to\langle1\rangle\).
  \end{enumerate}
  The first case is clear
  and the third case is reduced to the content of the second case.

  Therefore,
  we are reduced to showing that
  for \((\cc{C},\cc{M}_{1},\cc{M}_{2})\in \Mod^{\mm{cg}}(\PrLst)^{\otimes}\),
  the presentable \(\infty\)-category
  \(\cc{M}_{1}\otimes_{\cc{C}}\cc{M}_{2}\) is compactly generated.
  First, we claim that the action of~\(\cc{C}^{\omega}\) on~\(\cc{M}_{i}\)
  restricts to \(\cc{M}_{i}^{\omega}\).
  Consider an object \(c\in\cc{C}^{\omega}\).
  We need to show that \(c\otimes{-}\colon\cc{M}_{i}\to\cc{M}_{i}\)
  preserves compact objects.
  This follows from the observation that its right adjoint \(c^{\vee}\otimes{-}\)
  also preserves colimits.
  Therefore, \(\cc{M}_{1}\otimes_{\cc{C}}\cc{M}_{2}\)
  is equivalent
  to \(\Ind(\cc{M}_{1}^{\omega}\otimes_{\cc{C}^{\omega}}\cc{M}_{2}^{\omega})\),
  and hence compactly generated,
  where inside \(\Ind\) we consider the relative tensor product in \(\Catperf\).
\end{proof}

\begin{definition}
  Recall that by \cref{lem:rigidlycompactlygenerated}
  the functor $\Ind$ induces an equivalence between
  $\twoCAlgrig$ and the full subcategory $\CAlg(\PrLstomega)^{\rig} \subseteq \CAlg(\PrLstomega)$.
  By \cref{xwpbfz},
  we obtain a co-Cartesian fibration
  \begin{equation*}
    \Mod^{\mm{cg}} \to \twoCAlgrig \times \Fin_{\ast}
  \end{equation*}
  which is classified by a functor
  \begin{equation*}
    \mdef{\twoModcg}\in\Fun(\twoCAlgrig, \Fun(\Fin_{\ast}, \Catbig))
    \simeq\Fun(\twoCAlgrig, \Fun(\Fin_{\ast}, \Catbig)).
  \end{equation*}
  Moreover, $\twoModcg$ lifts
  to a functor $\twoCAlgrig \to \CAlg^{\times}(\Catbig) \subseteq \Fun(\Fin_{\ast}, \Catbig)$,
  as it is a full subfunctor of $\twoModbig \coloneq
  \Mod \circ \Ind$ which did so
  by~\cite[Theorem~4.5.3.1]{LurieHA}.
  Explicitly, this functor sends a rigid $2$-ring $\cc{K}$
  to the full symmetric monoidal subcategory of
  $\twoModbig_{\cc{K}}$ consisting of modules~$\cc{M}$
  whose underlying presentable $\infty$-categories are compactly generated. 
\end{definition}

We now prove \hyperref[thmalph:modb]{\cref*{thmalph:mod}\,(b)}.

\begin{theorem}\label{thm:2modcgisasheaf} Let $\cc{K}\in \twoCAlgrig$. The assignment sending a quasicompact open  $U \subseteq |\! \Spec \cc{K}|$ to $\twoModcg_{\cO_{\cc{K}}(U)}$ extends to a $\CAlg(\Catbig)$-valued sheaf on $\Spec \cc{K}$.
\end{theorem}
\begin{proof}
  Let $\cc{L}\in \twoCAlgrig$ be arbitrary. Since $\twoModcg$ is constructed as a full subfunctor of the sheaf defined in \cref{thm:2MODisasheaf}, we are reduced to showing that any object $\cc{M} \in \twoModbig_{\cc{L}}$ which is locally compactly generated must itself be compactly generated; where by \emph{locally compactly generated} we mean that there exists a Zariski cover $\cc{L} \to \cc{L}_{1} \times \dotsb \times \cc{L}_{n}$ for which $\cc{O}_{\cc{M}}(\Spec \cc{L}_{1}) \times \dotsb \times \cc{O}_{\cc{M}}(\Spec \cc{L}_{n})$ is compactly generated. \cref{prop:sheavesonspec} provides a further reduction to the case where $n = 2$.

  Let the kernels of the maps $\cc{L} \to \cc{L}_{1},\ \cc{L} \to \cc{L}_{2}$ be denoted $\cc{I}_{1}, \cc{I}_{2}$. Preserving the notation of \cref{prop:zariskilocaltoglobal}, let $\{m_{i}\}_{I_{1}}$ and $\{m'_{j}\}_{I_{2}}$ be small sets of compact generators of $\cc{M}/\cc{I}_{1}$ and $\cc{M}/\cc{I}_{2}$, respectively. The same construction as in the proof of the cited proposition provides a subset $S\coloneq \{P_{i} \oplus Q_{j}\}_{i \in I_{1}, j\in I_{2}} \subseteq \cc{M}^{\omega}$ such that the collection $\{Q_{j}\}_{j \in I_{2}}$ compactly generates $\cc{M}_{\cc{I}_{1}} \subseteq \cc{M}$ and the image of the collection $\{P_{i}\}_{i \in I_{1}}$ in $\cc{M}/\cc{I}_{1}$ compactly generates the latter. We may now employ the same argument as before to show that $S$ generates $\cc{M}$. 
\end{proof}

\hyperref[thmalph:modc]{\cref*{thmalph:mod}\,(c)} is obtained as a corollary of the above, which concludes this subsection.

\begin{corollary}\label{cor:2modisasheaf}
  Let $\cc{K} \in \twoCAlgrig$. The assignment sending a quasicompact open subset $U \subseteq |\! \Spec \cc{K}|$ to $\Mod_{\cO_{\cc{K}}(U)}(\Catperf)$ extends to a $\CAlg(\PrL)$-valued sheaf on $\Spec \cc{K}$. 
\end{corollary}
\begin{proof}
  By \cref{lem:the-one-about-compacts}, for any $\cc{L} \in \twoCAlgrig$ and $\cc{M} \in \twoModcg_{\cc{L}}$, the object $\cc{M}$ lifts to an object of $\Mod_{\Ind(\cc{L})}(\PrLomega)$. From the proof of the same lemma, there are equivalences \[\Ind \colon \Mod_{\cc{L}}(\Catperf) \simeq \Mod_{\Ind(\cc{L})}(\PrLomega) \noloc (-)^{\omega}\] which are natural in $\cc{L}$. Moreover, there is an inclusion $\Mod_{\Ind(\cc{L})}(\PrLomega)\subseteq \twoModcg_{\cc{L}}$ as the wide subcategory consisting of the functors which preserve compact objects, also natural in $\cc{L}$. In light of \cref{thm:2modcgisasheaf} and \cref{lem:qcmvsuffices}, we are reduced to showing that for any $\cc{L} \in \twoCAlgrig$ and map \[\cc{M} \to \cc{N} \in (\twoModcg_{\cc{L}})^{[1]}\] admitting a Zariski cover $\cc{L} \to \cc{L}_{1} \times \cc{L}_{2}$ such that the induced functor \[\cc{O}_{\cc{M}}(\Spec \cc{L}_{1}) \times \cc{O}_{\cc{M}}(\Spec \cc{L}_{2}) \to \cc{O}_{\cc{N}}(\Spec \cc{L}_{1}) \times \cc{O}_{\cc{N}}(\Spec \cc{L}_{2})\] preserves compact objects, the original map $\cc{M}\to \cc{N}$ preserves compact objects. Let $m \in \cc{M}^{\omega}$ arbitrary with image $n \in \cc{N}$, which we must show is compact. The restriction of $m$ to the cover $(m_{1}, m_{2}) \in \cc{O}_{\cc{M}}(\Spec \cc{L}_{1}) \times \cc{O}_{\cc{M}}(\Spec \cc{L}_{2})$ is compact, and by assumption its image in $\cc{O}_{\cc{N}}(\Spec \cc{L}_{1}) \times \cc{O}_{\cc{N}}(\Spec \cc{L}_{2})$ must be compact. Thus, $n \in \cc{N}$ admits compact restrictions, and the proof of \cref{prop:zariskilocaltoglobal} implies that it must itself be compact.
\end{proof}
\vspace{5pt}

\appendix

\section{Support Data and \texorpdfstring{$\GBal$}{GZar}-Structures}\label{sec:supp}

In \cite{balmerSpectrumPrimeIdeals2005}, the author characterizes the underlying the Balmer spectrum of a tensor-triangulated category $\cc{K}$ as the \emph{final support datum} for $\cc{K}$; this was hitherto the primary way to produce maps from the Balmer spectrum into an arbitrary frame. \cref{cor:absoluteGBalspectrum} provides an alternative approach to producing maps from the Balmer spectrum to an $\infty$-topos when the latter is equipped with a $\GBal$-structure. This section is aimed at reconciling these notions.

\begin{definition}\label{def:supportdata}
  A \tdef{support datum} for $\cc{K} \in \twoCAlg$ is a pair $(F, d)$ where $F$ is a frame and $d: \pi_{0}\cc{K}  \to F^{\op}$ is a map satisfying:
  \begin{enumerate}
  \item $d(0) = \bot_{F^{\op}}$ and $d(\unit) = \top_{F^{\op}}$
  \item $\forall a \in \cc{K}$, $d(a) = d(\Sigma a)$
  \item $\forall a,b \in \cc{K}$, $d(a \oplus b) = d(a) \vee_{F^{\op}} d(b)$
  \item $\forall a,b \in \cc{K}$, $d(a\otimes b) = d(a) \wedge_{F^{\op}} d(b)$
  \item $\forall f: a\to b \in \cc{K}^{[1]}$, $d(b) \leq d(a) \vee_{F^{\op}} d(\cofib(f))$
  \end{enumerate}
  A morphism of support data from $(F,d) \to (F',d')$ is a morphism of frames $f: F \to F'$ such that $f \circ d = d'\circ f$. We write $\mdef{\mm{Supp}_{\cc{K}}}$ to denote the category of support data for $\cc{K}$.
\end{definition}

\begin{example}
  For a topological space $X$, a support datum on $X$ is a support datum on the associated frame of open subsets $\cc{U}(X)$. Under the identification $\cc{U}(X)^{\op} \cong \mm{Cl}(X)$ between the opposite of the frame of open subsets and the poset of closed subsets of $X$, a support datum on $X$ is an assignment of closed subsets to elements of $\cc{K}$ satisfying the conditions of \cref{def:supportdata}. This is the context originally considered in \cite{balmerSpectrumPrimeIdeals2005}.
\end{example}

\begin{example}\label{ex:radsupp}
  The map $\sqrt{-} \colon \pi_{0}\cc{K} \to \Rad(\cc{K}) \subseteq (\Rad(\cc{K})^{\vee})^{\op}$ which sends an element $a \in \cc{K}$ to the radical of the thick tensor-ideal generated by $a$ is a support datum. To see this, note that conditions (a) and (b) of \cref{def:supportdata} are immediate, while conditions (c) and (e) are consequences of the fact that the assignment $\langle -\rangle\colon \cc{K} \to \Idl(\cc{K})^{\omega}$ obeys these conditions coupled with the fact that $\sqrt{}\colon \Idl(\cc{K}) \to \Rad(\cc{K})$ is a left adjoint and thus preserves joins. Condition (c) is recorded in \cref{lem:radicalgeneration}. 
\end{example}
 
\begin{proposition}\label{prop:rad-is-universal-support}
  Any support datum $d\colon \pi_{0}\cc{K}  \to F^{\op}$ admits an essentially unique factorization
  \[\xymatrix{
      \pi_{0}\cc{K}  \ar[dr]_{\sqrt{-}} \ar[rr]^{d}  & & F^{\op} \\
      & \Rad(\cc{K})^{\omega}\ar@{..>}[ur]_{\mm{supp}} & \\
    }
  \]
  where $\mm{supp}$ is a map of distributive lattices.
\end{proposition}
\begin{proof}
  We first claim that if $a,b \in \cc{K}$ such that $a \in \langle b \rangle$, then $d(a) \leq d(b)$ in $F^{\op}$. This follows inductively; if $a' \simeq \cofib(x \to y)$ where $x, y \in \langle b \rangle$ satisfy $d(x) \wedge_{F^{\op}} d(y) \leq d(b)$, then condition (e) of \cref{def:supportdata} implies $d(a') \leq d(b)$. Furthermore, conditions (b) and (d) of the above imply that $d(\Sigma^{k} b \otimes y) \leq d(b)$ for every $k \in \bb{Z}$, $y \in \cc{K}$. Since $\langle b \rangle$ is the closure of objects of the form $\Sigma^{k}b \otimes y$ under iterated cofiber sequences, it follows that every $a \in \langle b \rangle$ must satisfy $d(a) \leq d(b)$.

  Now suppose $a, b \in \cc{K}$ satisfy $\sqrt{a} = \sqrt{b}$. Then there exists $l \in \bb{Z}$ such that $b^{l} \in \langle a \rangle$, implying that $d(b^{l}) \leq d(a)$ and thus that $d(b) \leq d(a)$, since $d(b^{l}) = d(b)$ by condition (d) of \cref{def:supportdata}. A symmetric argument implies that $d(a) \leq d(b)$ and hence that $d(b) = d(a)$. It follows that the map $d\colon \pi_{0}\cc{K} \to F^{\op}$ factors through an essentially unique map $\Rad(\cc{K})^{\simeq} \to F^{\op}$ induced from the identification of elements up to their radicals. Moreover, we have shown that the composite map $\mm{Prin})(\cc{K})^{\simeq} \to (\Rad(\cc{K})^{\omega})^{\simeq} \to F^{\op}$ respects the partial order by containment and thus the map $\Rad(\cc{K})^{\omega} \to F^{\op}$ is a map of posets. The fact that it respects meets and joins follows from \cref{ex:radsupp}. 
\end{proof}

It is easy to see that given any support datum $(F,d) \in \mm{Supp}_{\cc{K}}$ and a map $f\colon F \to L \in \mm{Frm}^{[1]}$, the object $(L, f \circ d) \in \mm{Supp}_{\cc{K}}$. This observation and \cref{prop:rad-is-universal-support} immediately furnish the following universal property for the frame $\Rad(\cc{K})^{\vee}$, using the Stone duality for distributive lattices and coherent frames discussed in \cref{rec:stoneduality}.

\begin{corollary}\label{cor:univ-prop-of-spc}
 For any $F\in \mm{Frm}$ there are equivalences of sets
  \[ \{\text{Support data for }\cc{K}\text{ valued in } F\} \simeq \Map_{\mm{DLat}}(\Rad(\cc{K})^{\omega}, F^{\op}) \simeq \Map_{\mm{Frm}}(\Rad(\cc{K})^{\vee}, F).
  \] In particular, the functor $\mm{Frm}_{\Rad(\cc{K})^{\vee}} \to \mm{Supp_{\cc{K}}}$ given by
    \[\{f\colon \Rad(\cc{K})^{\vee} \to L\} \mapsto (L, f\circ \sqrt{-})\] is an equivalence of categories.
\end{corollary}

The characterization of $\Rad(\cc{K})^{\vee}$ as a \emph{universal support datum} above is recorded for support data with values in spatial frames in \cite{balmerSpectrumPrimeIdeals2005}, where it is phrased in the category of topological spaces. In the form above, one may extract it from \cite[Theorem 3.2.3]{KockPitsch17}. For us, the characterization above will be used to compute the maps to $\Spec \cc{K}$ in the category of locally 2-ringed topoi, on the level of underlying $\infty$-topoi. Let us first note that the adjunction of \cref{rec:0topoi} upgrades the universal property of \cref{cor:univ-prop-of-spc} from frames to $\infty$-topoi.

\begin{lemma}\label{lem:supportadjunction}
  There is an adjunction of the form
  \[ \mm{Supp}_{\cc{K}} \rightleftarrows \LTop_{\Shv(\Rad(\cc{K})^{\vee})/} 
    \]
    with fully faithful left adjoint given by $\Shv\colon \mm{Supp}_{\cc{K}} \simeq \mm{Frm}_{\Rad(\cc{K})^{\vee}/} \hookrightarrow \LTop_{\Shv(\Rad(\cc{K})^{\vee})/}$ and with right adjoint given by $\{f^{\ast}\colon \Shv(\Rad(\cc{K}))^{\vee} \to \cX\} \mapsto (\mm{Sub}(\cX), f^{\ast} \circ \sqrt{-})$.
\end{lemma}

The maps $\Shv(\Rad(\cc{K})^{\vee}) \to \cX$ which are of principal concern to us will arise as the underlying map of $\infty$-topoi associated to a map $\Spec \cc{K} \to (\cX, \cO) \in \LTop(\GBal)^{[1]}$ for some $\GBal$-structure on $\cX$. In this case, the associated support datum has a very explicit form. By the adjunction of \cref{thm:spec}, one has an identification

\[
  \LTop(\GBal)_{\Spec \cc{K}/} \simeq \LTop(\GBal) \times_{\twoCAlg} \twoCAlg_{\cc{K}/}
\]
i.e., the category of triples $(\cX, \cO,\eta)$ where $(\cX, \cO) \in \LTop(\GBal)$ and $\eta\colon \cc{K} \to \Gamma(\cX, \cO) \in \twoCAlg^{[1]}$. Such objects come equipped with a natural notion of support, as we now show.

\begin{construction}\label{cons:open-vanishing-adjunction}
  Let $(\cX, \cO, \eta) \in \LTop(\GBal) \times_{\twoCAlg} \twoCAlg_{\cc{K}/}$. We abusively let $\cO$ refer either to the $\GBal$ structure on $\cX$ or the associated object of $\Shv(\cX;\twoCAlg)$ under the equivalence of \cref{obs:sheafislexfunct}. Note that the sheaf associated to $\cO$ naturally lifts to a sheaf valued in $\twoCAlg_{\cc{K}/}$, using \cref{lem:globalsectionsisglobalsections} and precomposing with $\eta\colon \cc{K} \to \Gamma(\cX, \cO)$. Consider the composite
  \[
   \cX^{\op} \xrightarrow{\cO} \twoCAlg_{\cc{K}/} \xrightarrow{\ker} \Idl(\cc{K})
  \]  which sends an object $U \in \cX$ to the kernel of the associated map $\cc{K} \to \Gamma_{\GBal}((\cX, \cO)|_{U})$. For any $\infty$-category $\cc{C}$ and object $c \in \cc{C}$, the forgetful functors $\cc{C}_{c/} \to \cc{C}$ create small limits. It follows that the above functor preserves limits, since $\cO$ preserves limits and $\ker$ is a right adjoint by \cref{prop:quotientposet}. It follows that the functor above admits a left adjoint, since the restricted Yoneda map $\mm{Idl}(\cc{K}) \to \Fun(\cX, \cS)$ lands in the full subcategory $\Fun^{\colim}(\cX, \cS) \simeq \cX^{\op} \subseteq \Fun(\cX, \cS)$. We denote this adjoint pair as follows
  \begin{equation}
    \mdef{U(-)}\colon \Idl(\cc{K}) \rightleftarrows \cX^{\op} \noloc \mdef{\cc{I}(-)}
  \end{equation}
  where $U(-)$ is referred to as the \tdef{vanishing locus} functor and $\cc{I}(-)$ is the \tdef{vanishing ideal} functor. 
\end{construction}

\begin{lemma}\label{lem:vanishinglocusissubterminal}
 For any $(\cX, \cO) \in \LTop(\GBal) \times_{\twoCAlg} \twoCAlg_{\cc{K}/}$, the functor $U(-)\colon \Idl(\cc{K}) \to \cc{X}^{\op}$ has essential image in the full subcategory $\mm{Sub}(\cX)^{\op} \subseteq \cX^{\op}$ of subterminal objects.
\end{lemma}
\begin{proof}
  Since $U(-)^{\op} \colon \Idl(\cc{K})^{\op} \to \cX$ preserves limits and $\Idl(\cc{K})$ is a poset, the object $U(\cc{I}) \in \cX$ satisfies $U(\cc{I}) \simeq U(\cc{I}) \times U(\cc{I})$ via the diagonal map for every $\cc{I} \in \Idl(\cc{K})$. In particular, the representable functor associated to any $U(\cc{I})$ lies in the full subcategory $\Fun^{\lim}(\cX^{\op}, \tau_{\leq -1}\cS) \subseteq \Fun^{\lim}(\cX^{\op}, \cS) \simeq \cX$, yielding the claim. 
\end{proof}

\begin{lemma}\label{lem:vanishinglociipullback}
  Let $f\colon (\cX,\cO_{\cX}) \to (\cY, \cO_{\cY})$ be a morphism in $\LTop(\GBal)\times_{\twoCAlg} \twoCAlg_{\cc{K}/}$, i.e., a morphism of $\infty$-topoi with local $\GBal$-structure equipped with a $\cc{K}$-linear structure on the associated map on global sections. There is a commutative square of the form
  \[\xymatrix{
      \cY^{\op} \ar[d]_{U(-)} \ar[r]^{f_{\ast}} & \cX^{\op} \ar[d]^{U(-)} \\
      \Idl(\cc{K}) \ar@{=}[r] & \Idl(\cc{K}).
    }
    \]
\end{lemma}
\begin{proof}
The above claim will follow from the existence of a commutative square as below
  \[\xymatrix{
      \cY^{\op} \ar[d]_{\cO_{\cY}} \ar[r]^{f_{\ast}} & \cX^{\op} \ar[d]^{\cO_{\cX}} \\
      \twoCAlg_{\cc{K}/} \ar@{=}[r] & \twoCAlg_{\cc{K}/}
    }
  \]
  where we have abusively written $\cO_{\cX}, \cO_{\cY}$ to refer to the sheaves associated to the corresponding $\GBal$-structures under the equivalence of \cref{obs:sheafislexfunct}. This follows from a similar argument to \cref{lem:globalsectionsisglobalsections}.
\end{proof}

\begin{lemma}\label{lem:universalsupport}
  The restriction of the vanishing locus functor for $\Spec \cc{K}$ to $\Rad(\cc{K})^{\omega}$ may be identified with the Yoneda embedding
  \[
    \Yo\colon \Rad(\cc{K})^{\omega} \to \Shv(\Rad(\cc{K})^{\vee})^{\op} \subseteq \cc{P}(\Rad(\cc{K})^{\omega, \op})^{\op}
    \]
\end{lemma}
\begin{proof}
  From \cref{lem:vanishinglocusissubterminal} and \cref{thm:2zariski_geometry}, the vanishing locus functor factors through 
  \[
    \Rad(\cc{K})^{\omega} \xrightarrow{U(-)} \Rad(\cc{K})^{\vee, \op} \subseteq \Shv(\Rad(\cc{K})^{\vee})^{\op}
  \]
  and it suffices to show that this map is equivalent to the canonical inclusion. For any $\sqrt{a} \in \Rad(\cc{K})^{\omega}$, the kernel of $\cc{K} \to \cO_{\cc{K}}(\sqrt{a})$ factors through $\cc{K} \to \widetilde{\cO}_{\cc{K}}(\sqrt{a}) \simeq \cc{K}/\cc{J}$ by \cref{obs:identifyingOK}. It follows that one has a canonical map $\sqrt{a} \to U(\sqrt{a}) \in \Rad(\cc{K})^{\vee}$ by adjunction, and the claim reduces to showing that this is an equivalence.

  Under the identification $\Rad(\cc{K})^{\vee} \simeq \Spc \cc{K}$, $U(\sqrt{a})$ may be identified as the largest open subset of $\Spc \cc{K}$ satisfying $\sqrt{a} \subseteq \ker(\cc{K} \to \cO_{\cc{K}}(U(\sqrt{a})))$. Since $\Spc \cc{K}$ is spectral, it has a basis of quasicompact open subsets. In particular, any point $x_{\cc{P}} \in \Spc \cc{K}$ satisfies $x_{\cc{P}} \in U(\cc{J})$ if and only if there is a quasicompact open subset $x_{\cc{P}} \in U' \subseteq U(\cc{J})$ satisfying $\sqrt{a} \subseteq \ker(\cc{K} \to \cO_{\cc{K}}(U))$. By \cref{lem:identificationofstalks}, such a quasicompact open subset exists only if $\sqrt{a} \subseteq \ker(\cc{K} \to \cc{K}/\cc{P})$, or if $\sqrt{a} \subset \cc{P}$. It follows that $U(\sqrt{a})$ is contained in the unique open subset of prime ideals containing $a$, which is the open subset corresponding to $\sqrt{a}$ under the Stone duality equivalence $\Rad(\cc{K})^{\omega} \simeq \cc{U}(\Spc \cc{K})^{\omega,\op}$ of \cref{rec:stoneduality}. Thus, the map $\sqrt{a} \to U(\sqrt{a})$ must be an equivalence.
\end{proof}

\begin{remark}
  The lemma above demonstrates that the construction described in \cref{cons:open-vanishing-adjunction} is Hochster dual to the adjoint appearing in the equivalence of \cite[Theorem 4.10]{balmerSpectrumPrimeIdeals2005} in the case $(\cX, \cO) = \Spec \cc{K}$.
\end{remark}

\begin{notation}
  Given any object $(\cX, \cO) \in \LTop(\GBal) \times_{\twoCAlg} \twoCAlg_{\cc{K}/}$, let $\mdef{\sigma_{\cO}}$ refer to the following composite
  \[
    \pi_{0}\cc{K}  \xrightarrow{\langle - \rangle} \Idl(\cc{K}) \xrightarrow{U(-)} \mm{Sub}(\cX)^{\op}.
  \]
  where $\langle - \rangle$ is the map sending an element $a \mapsto \langle a \rangle$, the thick tensor ideal it generates.
\end{notation}

\begin{theorem}
  For $(\cX, \cO) \in \LTop(\GBal) \times_{\twoCAlg} \twoCAlg_{\cc{K}/}$, the pair $(\mm{Sub}(\cX), \sigma_{\cO})$ is a support datum for $\cc{K}$. Furthermore, under the identification $\LTop(\GBal) \times_{\twoCAlg} \twoCAlg_{\cc{K}/} \simeq \LTop(\GBal)_{\Spec \cc{K}/}$, the underlying map of $\infty$-topoi $\Spec \cc{K} \to \cX$ corresponds to the map induced by $\sigma_{\cO}$ under the adjunction of \cref{lem:supportadjunction}. 
\end{theorem}
\begin{proof}
  From \cref{lem:universalsupport} and \cref{lem:vanishinglociipullback} one has that the function $\sigma_{\cO}$ is associated to the direct image of the $\Rad(\cc{K})^{\vee}$-valued support theory of \cref{ex:radsupp} under the induced map $\Rad(\cc{K})^{\vee} \to \mm{Sub}(\cX)$, from which it follows that $\sigma_{\cO}$ must itself be a support theory. For the identification of the underlying map $\Spec \cc{K} \to \cX \in \LTop^{[1]}$, first note that this map must induce a map of associated support theories $(\Rad(\cc{K})^{\vee}, \sigma_{\cO_{\cc{K}}}) \to (\mm{Sub}(\cX), \sigma_{\cO})$ by noting that the commutative square of \cref{lem:vanishinglociipullback} yields a commutative square of associated left adjoint functors. By \cref{lem:universalsupport}, $\sigma_{\cO_{\cc{K}}}$ may be identified with the universal support theory $\sqrt{-}$, from which the map $\Spec \cc{K} \to \cX$ is uniquely determined by \cref{lem:supportadjunction}.
\end{proof}

\begin{remark}
    In particular, for any $(\cX, \cO) \in \LTop(\GBal) \times_{\twoCAlg} \twoCAlg_{\cc{K}/}$ the structure map $\Spec \cc{K} \to \cX$ above can be identified with the unique map induced from the map of distributive lattices $U(-)\colon \Rad(\cc{K})^{\op} \to \mm{Sub}(\cX)$ using the Stone duality for distributive lattices of \cref{rec:stoneduality} and the adjunction of \cref{rec:0topoi}.
\end{remark}
\vspace{5pt}

\section{Stalk-Locality of the Telescope Conjecture}\label{sec:slp}

In this section we prove \cref{thmalph:slp}. We refer the reader to \cref{rec:smashingloc} and the ensuing discussion for a refresher on the notions invoked below. 

\begin{definition}
  Let $\cc{K} \in \twoCAlgrig$. We say that $\cc{K}$ \tdef{satisfies the telescope conjecture} if the inclusion $\Idem_{\cc{K}}^{\mm{fin}} \subseteq \Idem_{\cc{K}}$ is an equivalence.
\end{definition}

\begin{theorem}\label{thm:slp}
  Let $\cc{K} \in \twoCAlgrig$. Then $\cc{K}$ satisfies the telescope conjecture if and only if $\cc{K}/\cc{P}$ satisfies the telescope conjecture for every prime $\cc{P} \in \Rad(\cc{K})$.
\end{theorem}

\begin{remark}
  Our primary interest in this result arises from the applications sketched in \cite{hrbekTelescopeConjectureHomological2025}, which demonstrates a characterization of the telescope conjecture via the use of homological residue fields under the assumptions of the \emph{Nerves-of-Steel Conjecture} and the \emph{Hereditary Stalk-Locality Principle} for a given rigid 2-ring. \cref{thm:slp} demonstrates that this latter condition is in fact automatically satisfied for all rigid 2-rings. Future work of the third author will study other stalk-locality principles for 2-rings using similar techniques to those invoked below.
\end{remark}

Both $\Idem$ and $\Idem^{\mm{fin}}$ are functorial in $\twoCAlgrig$; idempotent algebras are clearly preserved by symmetric monoidal functors, and the functoriality of $\Idem^{\mm{fin}}$ follows from this fact, \cref{cor:fidemtoideal}, and \cref{cor:basechange}. We will need the following lemma, which allows the telescope conjecture to be approached descent-theoretically.

\begin{lemma}
  Let $\cc{K} \in \twoCAlgrig$.
  \begin{enumerate}
  \item The assignment $U \mapsto \Idem_{\cO_{\cc{K}}(U)}$ extends to a sheaf of posets on $\Spec\cc{K}$.
  \item The assignment $U \mapsto \Idem_{\cO_{\cc{K}}(U)}^{\mm{fin}}$ extends to a sheaf of posets on $\Spec \cc{K}$.
  \end{enumerate}
\end{lemma}
\begin{proof}
  The fact that $\Idem(\cc{C})$ is a set for any $\cc{C} \in \CAlg(\PrL)$ is shown in \cite[Theorem 3.13]{aokiSheavesspectrumAdjunction2023}. Part (a) now follows from \cref{thm:indsheaf} and the fact that the functor $\Idem(-)\colon \CAlg(\PrL) \to \mm{Poset}$ preserves all limits, see for example \cite[Corollary 9.9]{barthelDescentTensorTriangular2024a}. 
  
  From the equivalence of \cref{cor:fidemtoideal}, the restriction of $\Idem_{\cO_{\cc{K}}}^{\mm{fin}}$ to $\mm{Prin}(\cc{K}) \subseteq (\Rad(\cc{K})^{\vee})^{\op}$ may be identified with the assignment $\cc{I} \mapsto \Rad(\cc{K})_{\cc{I}/}$, which sends an inclusion $\cc{I} \subseteq \cc{J}$ to the map
  \[
    \Rad(\cc{K})_{\cc{I}/} \to \Rad(\cc{K})_{\cc{J}/} \text{ via } \cc{I}' \mapsto \langle \cc{J}, \cc{I}'\rangle. 
  \]
  By \cref{lem:qcmvsuffices}, part (b) will follow from demonstrating that this assignment preserves pullbacks. We claim that this is a general feature of distributive lattices. To this end, let $F \in \mm{DLat}$ arbitrary and $x, y \in F$. We will show that there are mutually inverse equivalences as below
  \[
    (- \vee x, - \vee y)\colon F_{x \wedge y/ } \rightleftarrows F_{x/ }\times_{F_{x \vee y/ }} F_{y/ }\noloc - \wedge -.
  \]
  Given $z \in F_{x \wedge y}$, the relation $(z \vee x) \wedge (z \vee y) = z \vee (x \wedge y) = z$ follows from the assumption of distributivity. Similarly, given $(x', y') \in F_{x/} \times_{F_{x \vee y}}F_{y/}$, the relation $(x'\wedge y') \vee x = (x' \vee x) \wedge (y'\vee x) = x' \wedge (x' \vee y) = x$ follows from distributivity and the assumption on $(x',y')$.   
\end{proof}

\begin{remark}
    Part (a) of the lemma above also follows from the distributivity of the smashing spectrum shown in \cite{BKS2018}, and is completely independent of the results of this paper. As such, the proof of \cref{thm:slp} works equally well if one wishes to restrict to the triangulated setting. 
\end{remark}

\begin{lemma}\label{lem:this-weird-thing-about-localizations-that-isnt-super-necessary-lol}
    Let $I$ be a filtered category with an initial object, and $F \colon I \to \PrL$ be a filtered diagram such that every morphism is a localization. Then the forgetful functor $\PrL \to \Catbig$ preserves the colimit of $F$.
\end{lemma}
\begin{proof}
Denote the initial object of $I$ by $\emptyset$. We write $\cc{C}\coloneq F(\emptyset)$ and $L_{i} \dashv R_{i}$ for the corresponding localizations $\cc{C} \to F(i)$, given $i \in I$. From the equivalence $\PrL \simeq \Pr^{\mm{R},\op}$, the object $\cc{C}' \coloneq \varinjlim_{I} F \in \PrL$ has underlying $\infty$-category identified with $\varprojlim_{I^{\op}} F^{\op} \in \Pr^{\mm{R}}$, and the right adjoint map $R\colon \cc{C}' F^{\op} \to \cc{C} \in \Pr^{\mm{R}, [1]}$ is computed as the limit of the right adjoint functors $R_{i}\colon F(i) \to \cc{C}$. As $R$ is a limit of fully faithful functors along a diagram of fully faithful maps, it is itself fully faithful, and hence the adjoint pair $L\colon \cc{C} \rightleftarrows \cc{C}' F \noloc R$ in $\Pr^{\mm{L},[1]}$ is a localization. 

We now freely employ the language of localizations and local objects described by Bousfield, see for example \cite[\S 5.5.4]{LurieHTT}; we will also overload notation by simply writing $L$ and $L\cc{C}$ to refer to the monad $RL \colon \cc{C} \to \cc{C}$ and the image of $\cc{C}'$ under $R$, and similarly for $L_{i} \dashv R_{i}$. By \cite[Proposition 5.2.7.12]{LurieHTT}, for every $\cc{D} \in \Catbig$, composition with $L$ induces a fully faithful embedding \[\Fun(L\cc{C}, \cc{D}) \hookrightarrow \Fun(\cc{C}, \cc{D})\] consisting of exactly those functors which invert the collection $S \subseteq \cc{C}$ of morphisms such that $Lf \in \cc{C}^{[1]}$ is an equivalence. Applying \cite[Proposition 5.5.4.2, Proposition 5.5.4.15]{LurieHTT} there is a small collection $S_{0} \subseteq S$ such that $L\cc{C}$ is exactly the collection of $S_{0}$-local objects and $S$ is identified with the class of $S_{0}$-equivalences. Similarly, the cited results also imply that $L_{i}\cc{C}$ may be described as exactly the $S_{0}^{i}$-local objects for small collections of morphisms $S_{0}^{i}$; since by construction, $L\cc{C} \simeq \varprojlim_{I^{\op}} L_{i} \cc{C}$, we have that $L\cc{C} \subseteq \cc{C}$ is the full subcategory of $\bigcap_{I} S_{0}^{i}$ local objects, whence $S$ may be identified as the collection $\bigcap_{I} S_{i}$ consisting of the $\bigcap_{I} S_{0}^{i}$-equivalences. It follows that composition with $L$ induces an identification \[\Fun(L\cc{C}, \cc{D}) \simeq {\varprojlim}_{I^{\op}}\Fun(L_{i} \cc{C}, \cc{D}) \hookrightarrow \Fun(\cc{C}, \cc{D})\] for all $\cc{D} \in \Catbig$ and hence the canonical map $\varinjlim_{I} F(i) \to \cc{C}'$ in $\Catbig$ is an equivalence.
\end{proof}

\begin{lemma}\label{lem:idem-is-presentable}
    Given $\cc{C} \in \CAlg(\PrLst)$, $\Idem(\cc{C})$ is presentable.
\end{lemma}
\begin{proof}
      Note that $\Idem(\cc{C})$ is accessible by \cite[Proposition 2.14, Lemma 3.15]{aokiSheavesspectrumAdjunction2023}. Since $\cc{C}$ is presentable, $\CAlg(\cc{C})$ must be as well; under the equivalence $\Fun^{\mm{L}}(\cS, \cc{C}) \simeq \CAlg(\cc{C})$ given by sending $F \mapsto F(\ast)$, the full subcategory $\Idem(\cc{C}) \subseteq \CAlg(\cc{C})$ corresponds to the full subcategory of functors $\Fun^{L}(\cS, \cc{C})$ which invert the map $\ast \to \ast \coprod \ast$. This latter property is clearly preserved under colimits, and thus $\Idem(\cc{C})$ admits all colimits.  
\end{proof}

\begin{proof}[Proof of \cref{thm:slp}]
  Let us first prove the only if direction, namely that if $\cc{K}$ satisfies the telescope conjecture then $\cc{K}/\cc{I}$ satisfies the telescope conjecture for every ideal $\cc{I} \in \Rad(\cc{K})$. Let $L_{\cc{I}}$ denote the finite idempotent associated to $\cc{I}$ by \cref{prop:finitelocissmashing}. We claim that the forgetful functor \[\Idem_{\cc{K}/\cc{I}}\simeq \Idem(\CAlg(\Mod_{\cc{L}_{\cc{I}}}(\Ind(\cc{K}))) \simarrow \CAlg(\Ind(\cc{K})_{L_{\cc{I}}/}\] induces an equivalence $\Idem_{\cc{K}/\cc{I}} \simeq (\Idem_{\cc{K}})_{L_{\cc{I}}/}$. As the codiagonal map $L_{\cc{I}}\otimes L_{\cc{I}} \to L_{\cc{I}}$ is an equivalence, the forgetful functor $\CAlg(\Ind(\cc{K})_{L_{\cc{I}}/} \to \CAlg(\Ind(\cc{K}))$ preserves and reflects tensor products using the relation \[L_{\cc{I}}\otimes_{L_{\cc{I}}\otimes L_{\cc{I}}} (A \otimes B) \simeq A\otimes_{L_{\cc{I}}}B\] for $A, B \in \CAlg(\Ind(\cc{K}))_{L_{\cc{I}}/}$. It follows that the image of $\Idem_{\cc{K}/\cc{I}}$ consists exactly of idempotent algebras under $L_{\cc{I}}$. The relation $\Idem^{\mm{fin}}_{\cc{K}/\cc{I}} = (\Idem_{\cc{K}})^{\mm{fin}}_{L_{\cc{I}}\unit/ }$ follows from the above and the fact that finite localizations are closed under composition, see \cref{cor:retractsandleftcancel}. If $\Idem_{\cc{K}}^{\mm{fin}} = \Idem_{\cc{K}}$ then $(\Idem_{\cc{K}})^{\mm{fin}}_{L_{\cc{I}}\unit/ } = (\Idem_{\cc{K}})_{L_{\cc{I}}\unit/ }$ which supplies the claim.

  For the other direction, note that the direction shown above implies that $\cc{K}$ satisfies the telescope conjecture if and only if the natural map of sheaves $\Idem_{\cO_{\cc{K}}}^{\mm{fin}} \to \Idem_{\cO_{\cc{K}}}$ is an equivalence. Since this may be identified with a sheaf of posets on a topological space, this map is an equivalence if and only if it is an equivalence on stalks. Let $\cc{P} \in \Rad(\cc{K})$ be a prime ideal with associated point $x_{\cc{P}} \in |\! \Spec \cc{K}|$. The argument of \cref{lem:identificationofstalks} shows that we have an identification
  \begin{equation}\label{eq:idemstalksleft}
    \Idem_{\cO_{\cc{K}}, x_{\cc{P}}} \cong {\varinjlim}_{\mm{Prin}(\cc{K})_{/\cc{P}}}\Idem_{\cc{K}/\cc{I}}
  \end{equation}
  and similarly for $\Idem^{\mm{fin}}_{\cO_{\cc{K}}}$. From the identification $\Idem_{\cc{K}/ \cc{I}} \simeq (\Idem_{\cc{K}})_{L_{\cc{I}}/}$, for any map $\cc{I} \to \cc{J} \in \mm{Prin}(\cc{K})$ we have an adjunction
  \[
    L_{\cc{I}} \otimes -\colon \Idem_{\cc{K}/\cc{I}} \rightleftarrows \Idem_{\cc{K}/\cc{J}} \noloc \mm{fgt}.
  \]
 and by \cref{lem:idem-is-presentable} the filtered diagram of the right hand side of \eqref{eq:idemstalksleft} lifts to a diagram in $\Pr^{\mm{L}}$ such that every morphism is a localization. Applying \cref{lem:this-weird-thing-about-localizations-that-isnt-super-necessary-lol} this filtered colimit may instead be computed in $\PrL$, and since $\Pr^{\mm{L}} \simeq \Pr^{\mm{R}, \op}$, we find that we may instead compute the right hand side of \eqref{eq:idemstalksleft} in $\Cat$ as a limit along the cofiltered diagram of right adjoints. We obtain the first equality below
  \[
    {\bigcap}_{\cc{I}\in \mm{Prin}(\cc{K})}(\Idem_{\cc{K}})_{L_{\cc{I}}} = (\Idem_{\cc{K}})_{L_{\cc{P}}/ } \simeq \Idem_{\cc{K}/\cc{P}}
  \]
  by noting that any idempotent $L$ which receives maps from every $L_{\cc{I}}$ for $\cc{I}\subseteq \cc{P}$ must satisfy $L \otimes x = 0$ for every $x \in \cc{P}$. Using the fully faithful embedding $\Mod\colon \Idem_{\cc{K}} \to \CAlg(\PrL)_{\Ind(\cc{K})/}$, \cref{cor:fidemtoideal} and \cref{prop:quotientposet} imply that $L$ must receive a map from $L_{\cc{P}}$. 
  
Similarly, $\Idem^{\mm{fin}}_{\cc{K}}$ is a coherent frame by \cref{prop:latticeiscoherent} and is thus presentable. The same argument as above shows that $\Idem^{\mm{fin}}_{\cO_{\cc{K}}, x_{\cc{P}}} \simeq \Idem_{\cc{K}/ \cc{P}}^{\mm{fin}}$ and that this is compatible with the comparison map $\Idem^{\mm{fin}} \to \Idem$. 
  
  We conclude that the map of sheaves $\Idem^{\mm{fin}}_{\cO_{\cc{K}}}\to \Idem_{\cO_{\cc{K}}}$ is an equivalence if and only if the induced maps on stalks $\Idem_{\cc{K}/\cc{P}}^{\mm{fin}} \subseteq \Idem_{\cc{K}/\cc{P}}$ are equalities for every prime $\cc{P} \in \Rad(\cc{K})$.
\end{proof}
\vspace{5pt}

\printbibliography

\end{document}